\documentclass[reqno,11pt,a4paper]{amsart}
\usepackage{amsmath,amssymb}
\usepackage{amsmath, amsfonts,amsthm,amssymb,amsbsy,upref,color,graphicx,amscd,hyperref,enumerate}
\usepackage[active]{srcltx}
\usepackage[latin1]{inputenc}
\usepackage{bbm}
\usepackage{algorithm}
\usepackage{array}
\usepackage{algorithmicx}
\usepackage{algpseudocode}
\usepackage{graphicx}
\usepackage{multirow}
\usepackage{wrapfig}
\usepackage{subcaption}
\usepackage{epstopdf}

\numberwithin{equation}{section}
\theoremstyle{plain}
\newtheorem{thm}{Theorem}[section]
\newtheorem{lemma}[thm]{Lemma}
\newtheorem{cor}[thm]{Corollary}
\newtheorem{prop}[thm]{Proposition}

\theoremstyle{remark}
\newtheorem{rem}[thm]{Remark}

\newcommand{\mean}[1]{\,-\hskip-1.08em\int_{#1}} 
\newcommand{\meantext}[1]{\,-\hskip-0.88em\int_{#1}} 
\newcommand{\ind}{\mathbbm{1}}

\renewcommand{\Bbb}{\mathbb}

\newcommand{\wconv}{\rightharpoonup}
\newcommand{\ds}{\displaystyle}
\def\ds{\displaystyle}
\def\eps{{\varepsilon}}

\def\R{\mathbb{R}}

\def\F{\mathcal{F}}
\def\HH{\mathcal{H}}

\def\Dr{D}
\def\vf{\varphi}

\newcommand{\cp}{\operatorname{cap}}

\title[Multiphase spectral optimization]{A multiphase shape optimization problem for eigenvalues: qualitative study and numerical results}

\author{Beniamin Bogosel , Bozhidar Velichkov}

\graphicspath{{./pics/}{./pics_eps/}}

\begin{document}

\begin{abstract}
We consider the multiphase shape optimization problem 
$$\min\Big\{\sum_{i=1}^h\lambda_1(\Omega_i)+\alpha|\Omega_i|:\ \Omega_i\ \hbox{open},\ \Omega_i\subset\Dr,\ \Omega_i\cap\Omega_j=\emptyset\Big\},$$
where $\alpha>0$ is a given constant and $\Dr\subset\R^2$ is a bounded open set with Lipschitz boundary. We give some new results concerning the qualitative properties of the optimal sets and the regularity of the corresponding eigenfunctions. We also provide numerical results for the optimal partitions.   
\end{abstract}

\maketitle

%


\section{Introduction}

In this paper we consider a variational problem in which the variables are subsets of a given \emph{ambient space} or \emph{design region} $\Dr$ and the cost functional depends on the solution of a certain PDE on each of the domains. This type of problems are known as \emph{shape optimization problems} and received a lot of attention from both the theoretical and the numerical community in the last years (we refer to the books \cite{bucurbuttazzo}, \cite{henrot-pierre} and \cite{henroteigs} for an introduction to the topic). A special type of shape optimization problems are the \emph{multiphase shape optimization problems} in which the aim is to find the optimal configuration of $h$ different disjoint sets $\Omega_1,\dots,\Omega_h$ with respect to a certain cost functional $\mathcal{F}$
\begin{equation}\label{optF05}
\min\Big\{\mathcal F\big(\Omega_1,\dots,\Omega_h\big)\ :\ \Omega_i\subset\Dr,\ \Omega_i\cap\Omega_j=\emptyset\Big\}.
\end{equation}
This type of problems may arise in some models studying the population dynamics of several highly competing species or in biology to simulate the behavior of a cluster of cells. In some special cases it is not restrictive from mathematical point of view to assume that the sets $\Omega_i$ fill the entire region $\Dr$. This is for example the case when the functional $\mathcal{F}$ is decreasing with respect to the set inclusion, i.e. if an empty space is left it will be immediately filled by some of the phases $\Omega_i$ decreasing the total optimization cost. Of course, it is always possible to write a multiphase problem as an optimal partition problem by adding the auxiliary phase $\Omega_{h+1}:=\Dr\setminus\Big(\cup_{i=1}^h \Omega_i\Big)$. On the other hand, we notice that in this way we violate the symmetry of the problem since this new phase does not appear in the functional. In some cases this does not change the nature of the problem. Consider for example an optimization cost given by the total length of the boundary $\partial \Big(\cup_{i=1}^h\Omega_i\Big)$, i.e.
$$\F(\Omega_1,\dots,\Omega_h)=\sum_{i=1}^h |\partial\Omega_i|-\sum_{i\neq j} |\partial\Omega_i\cap\partial\Omega_j|.$$
In fact in this case we may introduce the new functional 
$$\widetilde\F(\Omega_1,\dots,\Omega_{h+1})=\frac12\sum_{i=1}^{h+1}|\partial\Omega_i|,$$
which is of the same type. In other cases the introduction of $\Omega_{h+1}$ may change the nature of the problem. Consider for example a functional depending on the principal eigenvalues on each set $\Omega_i$ and the Lebesgue measure $|\Omega_i|$ 
$$\mathcal{F}(\Omega_1,\dots,\Omega_{h})=\frac12\sum_{i=1}^{h}\big(\lambda_1(\Omega_i)+|\Omega_i|\big).$$
Then, the corresponding optimal partition functional is given by 
$$\widetilde\F(\Omega_1,\dots,\Omega_{h+1})=\frac12\sum_{i=1}^{h}\lambda_1(\Omega_i)-|\Omega_{h+1}|,$$
and acts differently on the original sets $\Omega_i$ and the auxiliary set $\Omega_{h+1}$. 

We consider the multiphase shape optimization problem 
\begin{equation}\label{optsum}
\min\Big\{\sum_{i=1}^h\lambda_1(\Omega_i)+\int_{\Omega_i} W_i(x)\,dx\ :\ \Omega_i\ \hbox{open},\ \Omega_i\subset\Dr,\ \Omega_i\cap\Omega_j=\emptyset\Big\},
\end{equation}
where 
\begin{itemize}
\item the ambient space $\Dr$ is a bounded open set with Lipschitz boundary or more generally a compact manifold with or without boundary;
\item $\lambda_1(\Omega_i)$ is the first Dirichlet eigenvalue of $\Omega_i$;
\item $W_i:\Dr\to[0,+\infty]$ are given measurable functions.
\end{itemize}
Our aim is to provide a theoretical and numerical analysis of the problem and to study the qualitative behavior of the solutions from both points of view. We notice that the optimal configurations consists of sets with rounded corners if the weight functions are sufficiently small. This phenomenon can be modeled in a direct way by adding a small curvature term, as $\eps \int_{\partial\Omega}\kappa^2_i$, where $\kappa_i$ is the curvature of $\partial\Omega_i$, but from the numerical point of view the volume term is much simpler to handle and gives the same qualitative behavior. 

In the next two examples we see the optimization problem from two different points of view. 

\begin{rem}[Two limit cases]
In the case $W_i\equiv \alpha$ on $\Dr$, we obtain the following problem: 
\begin{equation}\label{optsum2}
\min\Big\{\sum_{i=1}^h\lambda_1(\Omega_i)+\alpha|\Omega_i|:\ \Omega_i\ \hbox{open},\ \Omega_i\subset\Dr,\ \Omega_i\cap\Omega_j=\emptyset\Big\}.
\end{equation}
The variational problem \eqref{optsum} is widely studied in the literature in the case $\alpha=0$ that corresponds to the classical optimal partition problem. We refer to the papers \cite{coteve03}, \cite{caflin}, \cite{hehote} and \cite{buboou} for a theoretical and numerical analysis in this case. The other limit case appears when the constant $\alpha>0$ is large enough. Indeed, we recall that the solution of the problem 
\begin{equation}
\min\Big\{\lambda_1(\Omega)+\alpha|\Omega|:\Omega\ \hbox{open},\ \Omega\subset\R^2\Big\},
\end{equation}
is a disk of radius $\ds r_\alpha=\left(\frac{\lambda_1(B_1)}{\alpha\pi}\right)^{\frac14}$. It is straightforward to check that if $\alpha>0$ is such that there are $h$ disjoint disks of radius $r_\alpha$ that fit in the box $\Dr$, then the solution of \eqref{optsum} is given by the $h$-uple of these disks. Finding the smallest real number $\overline \alpha>0$, for which the above happens, reduces to solving the optimal packing problem
\begin{equation}\label{optpack}
\max\Big\{r:\ \hbox{there exist}\ h\ \hbox{disjoint balls}\ B_r(x_1),\dots, B_r(x_h)\ \hbox{in}\  \Dr \Big\}.
\end{equation}
\end{rem}

In view of the previous remark the multiphase problem \eqref{optsum}, in variation of the parameter $\alpha$, can be seen as an interpolation between the optimal partition problem (corresponding to the case $\alpha=0$) and the optimal packing problem \eqref{optpack}. It is interesting to notice that in the asymptotic case when $\Dr=\R^2$, the solution of the optimal packing problem consists of balls with centers situated in the the vertices of a infinite hexagonal \emph{honeycomb} partition of the plane. On the other hand, in the case $\alpha=0$ Caffarelli and Lin conjectured that the optimal configuration is precisely the honeycomb partition.
\medskip

\begin{rem}[Competing species with diffusion]
Suppose that $\Omega_i$ represents the habitat of a certain species and that the first eigenfunction $u_i$ on $\Omega_i$, solution of
$$-\Delta u_i=\lambda_1(\Omega_i) u_i\quad\text{in}\quad\Omega_i,\qquad u_i=0\quad\text{on}\quad\partial\Omega_i,\qquad\int_{\Omega_i} u_i^2\,dx=1,$$
is the population distribution. The condition $\Omega_i\cap\Omega_j=\emptyset$ corresponds to the limit assumption that the two species cannot coexists on the same territory. We suppose that $S_i\subset\Dr$ is a closed set representing a distribution of resources and that $\varphi_i:[0,+\infty]\to[0,+\infty]$ is a given increasing function that corresponds to the cost of transportation of resources at a given distance. The population $u_i$ will tend to choose an habitat close to $S_i$. This corresponds to the following multiphase problem 
\begin{equation}\label{optsum3}
\min\Big\{\sum_{i=1}^h\lambda_1(\Omega_i)+\int_{\Omega_i} \varphi_i\big(\text{dist}(x,S_i)\big)\,dx\ :\ \Omega_i\ \hbox{open},\ \Omega_i\subset\Dr,\ \Omega_i\cap\Omega_j=\emptyset\Big\}.
\end{equation}
\end{rem}

\medskip

The first part of the paper is dedicated to the analysis of the solutions of \eqref{optsum}. We summarize the results in the following   
\begin{thm}\label{mainr2}
Suppose that $\Dr\subset\R^2$ is a bounded open set with Lipschitz boundary. Let $0<a\le A$ be two positive real numbers and $W_i:\Dr\to[a,A]$, $i=1,\dots,h$ be given $C^2$ functions. Then there are disjoint open sets $\Omega_1,\dots,\Omega_h\subset\Dr$  solving the multiphase optimization problem \eqref{optsum}. Moreover, any solution to \eqref{optsum} has the following properties:
\begin{enumerate}[(i)]
\item There are no triple points inside $\Dr$, i.e. for every three distinct indices $i,j,k\in\{1,\dots,h\}$ we have $\partial\Omega_i\cap\partial\Omega_j\cap\partial\Omega_k=\emptyset$.
\item There are no double points on the boundary of $\Dr$, i.e. for every pair of distinct indices $i,j\in\{1,\dots,h\}$ we have $\partial\Omega_i\cap\partial\Omega_j\cap\partial\Dr=\emptyset$.
\item If the set $\Dr$ is of class $C^2$, then the first eigenfunctions $u_i\in H^1_0(\Omega_i)$ are Lipschitz continuous on $\overline\Omega_i$. 
\item The set $\ds\Omega=\bigcup_{i=1}^h\Omega_i$ has finite perimeter and the free reduced boundary $\partial^\ast\Omega$ is smooth in $\Dr$. Equivalently the reduced boundary $\partial^\ast\Omega_{h+1}$ of the auxiliary phase $\Omega_{h+1}=\Dr\setminus\Omega$ is smooth in $\Dr$.  
\end{enumerate}
\end{thm}

\begin{rem}
We notice that the above result is still valid in dimension $d>2$. We restrict our attention to dimension $2$ since we can avoid some technicalities in the proofs of the Lipschitz continuity of the eigenfunctions and the decay monotonicity formula Lemma \ref{lem2ph05}. In fact, a key step in the proof of the Lipschitz continuity of the eigenfunctions is to show their non-degeneracy on the boundary in terms of the gradients. This question can be handled easily in two dimensions, while for the case $d>2$ we refer to \cite[Theorem 5.9]{buve}, where the case of the Dirichlet energy was considered. 
\end{rem}

For the computation of the optimal partition we use an approach that has as a starting point the algorithm used in \cite{buboou}. We notice that the first eigenvalue of an open set $\Omega\subset\Dr$ can be formally characterized as $\lambda_1(\Omega,+\infty)$, where 
$$\lambda_1(\Omega,C)=\min_{u\in H^1_0(\Omega)\setminus \{0\}}\ \frac{\int_{\Dr}|\nabla u|^2+C\ind_{\Dr\setminus\Omega} u^2\,dx}{\int_\Dr u^2\,dx}.$$
Replacing the characteristic function of $\Omega$ by a function $\varphi:\Dr\to[0,1]$ we can define
$$\lambda_1(\varphi,C)=\min_{u\in H^1_0(\Omega)\setminus \{0\}}\ \frac{\int_{\Dr}|\nabla u|^2+C(1-\varphi) u^2\,dx}{\int_\Dr u^2\,dx},$$
and then replace the optimal partition problem by 
\begin{equation}\label{optsumC}
\min\Big\{\sum_{i=1}^h\lambda_1(\varphi_i,C)+\int_{\Dr} \varphi_i(x)W_i(x)\,dx\ :\ \Omega_i\ \hbox{open},\ \varphi_i:\Dr\to[0,1],\ \sum_{i=1}^h \varphi_i \le 1\Big\}.
\end{equation}
In \cite{buboou} it was proved that as $C \to +\infty$ and $\varphi$ is the characteristic function of a regular set $\Omega$, then the relaxed eigenvalue $\lambda_k(\varphi,C)$ converges to the actual eigenvalue $\lambda_k(\Omega)$. To the authors knowledge, there was no prior study of the rate of convergence in terms of $C$.

In Section \ref{eigcomp} we observe the numerical error of a few simple shapes in terms of $C$ and the discretization parameter, by comparing the values of the eigenvalues computed in the penalized setting, with the ones computed using MpsPack \cite{mpspack}. We observe that as $C$ and the discretization parameter $N$ increase, the errors decrease.  In Section \ref{error-estimate} we use the results of \cite{potentials-bbv} in order to obtain a theoretical upper bound for the relative error
$ |\lambda_k(\Omega)-\lambda_k(\Omega,C)|/\lambda_k(\Omega)$.
Precisely we will prove the following. 

\begin{thm}
Suppose $D\subset \Bbb{R}^N$ is a bounded open set and $\Omega \subset D$ a set with boundary of class $C^2$. Then there exists a constant $K>0$ depending on $\Omega,D,N$, for which we have
\[ \frac{|\lambda_k(\Omega)-\lambda_k(\mu_C)|}{\lambda_k(\Omega)} \leq KC^{-1/(N+4)}.\]
\label{error-thm}
\end{thm}
 This bound on the error makes the convergence result proved in \cite{buboou} more precise. In addition to this, we observe a good concordance between the theoretical bounds and the numerical errors observed in Section \ref{eigcomp}. 

In Section \ref{optim-algo} we present the main lines of the optimization procedure. One challenging issue was to manage well the non overlapping condition $\ds\sum_{i=1}^h \varphi_i \leq 1$. We introduce an extra phase $\varphi_{h+1}$, which represents the void space. Thus we are left to manage an equality condition, instead of an inequality. This allows us to adapt the framework presented in \cite{buboou} to our problem. We use a standard gradient descent algorithm with a line search procedure in order to accelerate the convergence. We observe good stability properties of our proposed algorithm by performing a few optimizations starting from random densities, and by observing that the resulting shape configuration, and cost values are close. In addition to the finite difference framework on a rectangular grid, we also propose an approach based on finite elements, which can be generalized to general plane domains, and even to surfaces in three dimensions. 

In Section \ref{numerical-obs} we present some of the results obtained using the presented numerical frameworks, as well as some numerical observations which motivate the interest in the study of problem \eqref{optsum2}. First, we mention that the numerical results satisfy the theoretical properties proved in \cite{buve} and in Theorem \ref{mainr2}: the lack of triple points, the lack of triple points on the boundary and the lack of angles. Secondly, we observed an interesting connection between the two interesting cases $\alpha = 0$ and the value of $\alpha$ which gives the circle packing, in the periodic setting. It is well known that the hexagonal circle packing in the plane has the maximal density (result attributed to A. Thue, with first rigorous proof by F. Toth). As mentioned above, in the case $\alpha = 0$ (the spectral partition), it is conjectured that the optimal asymptotic partition is the honeycomb partition. This conjecture was supported numerically by the results of \cite{buboou}. As we already mentioned the problem \ref{optsum2} provides a connection between the established result of the circle packing configuration, and the Caffarelli-Lin conjecture that the regular honeycomb tiling is the solution of the spectral optimal partition problem. In our computations we observe that starting from the parameter $\alpha$ which realizes the circle packing in the periodic setting, and decreasing $\alpha$, the shapes forming the optimal partition grow in a monotone fashion. If this observed monotonicity property could be proved theoretically, then a proof that the honeycomb partition is optimal for $\alpha = 0$ will follow. Note that this also applies in the case of the sphere, where it is expected that for $h \in \{3,4,6,12\}$ the optimal spectral partitions are realized by regular tilling of the sphere.

The paper is organized as follows.
In Section \ref{prel} we recall the known results and we introduce the basic notions that we use in the proof of the above results. Section \ref{proof-mainres} is dedicated to the proof of Theorem \ref{mainr2}. In Section \ref{eigcomp} we present the eigenvalue computation method, and we make a few numerical tests, comparing our results to other methods, or to analytical results. Section \ref{error-estimate} is dedicated to the proof of Theorem \ref{error-thm}.
 In Section \ref{optim-algo} we present the optimization algorithm used for calculating the numerical minimizers of \eqref{optsum}. 
 The numerical results and other observations are discussed in Section \ref{numerical-obs}.

\section{Preliminaries and main tools}\label{prel}

\subsection{Eigenvalues and eigenfunctions}
Let $\Omega\subset\R^2$ be an open set. We denote with $H^1_0(\Omega)$ the Sobolev space obtained as the closure in $H^1(\R^2)$ of $C^\infty_c(\Omega)$, i.e. the smooth functions with compact support in $\Omega$, with respect to the Sobolev norm 
$$\|u\|_{H^1}:=\left(\|\nabla u\|_{L^2}^2+\|u\|_{L^2}^2\right)^{1/2}=\left(\int_{\R^2}|\nabla u|^2+u^2\,dx\right)^{1/2}.$$
We note that $H^1_0(\Omega)$ can be characterized as
\begin{equation}\label{sob}
H^1_0(\Omega)=\Big\{u\in H^1(\R^2):\ \cp\big(\{u\neq0\}\setminus\Omega\big)=0\Big\},
\end{equation}
where the capacity $\cp(E)$ of a measurable set $E\subset\R^2$ is defined as
$$\cp(E)=\min\Big\{\|u\|_{H^1}^2:\ u\ge 1\ \hbox{in a neighbourhood of}\ E\Big\}\footnote{for more details see, for example, \cite{gariepy} or \cite{henrot-pierre}}.$$
We notice that the sets of zero capacity have also zero Lebesgue measure, while the converse might be false. We may use the notion of capacity to choose more regular representatives of the functions of the Sobolev space $H^1(\R^d)$. In fact, every function $u\in H^1(\R^d)$ has a representative which is quasi-continuous, i.e. continuous outside a set of zero capacity. Moreover, two quasi-continuous representatives of the same Sobolev function coincide outside a set of zero capacity. Thus we may consider $H^1(\R^2)$ as a space consisting of quasi-continuous functions equipped with the usual $H^1$ norm.

The $k$th eigenvalue of the Dirichlet Laplacian can be defined through the min-max variational formulation 
\begin{equation}
\lambda_k(\Omega):=\min_{S_k\subset H^1_0(\Omega)}\max_{u\in S_k\setminus\{0\}}\frac{\int_{\Omega}|\nabla u|^2\,dx}{\int_\Omega u^2\,dx},
\end{equation}
where the minimum is over all $k$ dimensional subspaces $S_k$ of $H^1_0(\Omega)$. There are functions $u_1,\dots,u_k,\dots$ in $H^1_0(\Omega)$, orthonormal in $L^2(\Omega)$, that solve the equation 
$$-\Delta u_k=\lambda_k(\Omega)u_k,\qquad u_k\in H^1_0(\Omega),$$
in a weak sense in $H^1_0(\Omega)$. In particular, if $k=1$, then the first eigenfunction $u_1$ of $\Omega$ is the solution of the minimization problem 
\begin{equation}\label{lb1}
\lambda_1(\Omega):=\min_{u\in H^1_0(\Omega)\setminus \{ 0\}}\frac{\int_{\Omega}|\nabla u|^2\,dx}{\int_\Omega u^2\,dx}.
\end{equation}
In the sequel we will often see $\lambda_1$ as a functional on the family of open sets. We notice that this functional can be extended to the larger class of quasi-open sets, i.e. the sets $\Omega\subset\R^2$ such that for every $\eps>0$ there exists an open set $\omega_\eps$ of capacity $\cp(\omega_\eps)\le \eps$ such that $\Omega\cap\omega_\eps$ is an open set. We define $H^1_0(\Omega)$ as the set of Sobolev functions $u\in H^1(\R^2)$ such that $u=0$ \emph{quasi-everywhere} (i.e. outside a set of zero capacity) on $\Omega^c$. The first eigenvalue and the first eigenfunctions are still characterized as the minimum and the minimizer of \eqref{lb1}.

We notice that since $|u_1|$ is also a solution of \eqref{lb1}, from now on we will always assume that $u_1$ is non-negative and normalized in $L^2$. Moreover, we have the following properties of $u_1$ on a generic open\footnote{The same properties hold for the first eigenfunction on quasi-open set of finite measure.} set $\Omega$ of finite measure:
\begin{itemize}
\item $u_1$ is bounded and we have the estimate\footnote{We note that the infinity norm of $u_1$ can also be estimated in terms of $\lambda_1(\Omega)$ only as $\|u_1\|_{L^\infty}\le C\lambda_1(\Omega)^{d/4}$. This estimate is more general and can be found in \cite[Example 8.1.3]{davies}.}
\begin{equation}\label{davies}
\|u_1\|_{L^\infty}\le \frac1\pi\lambda_1(\Omega)|\Omega|^{1/2}.
\end{equation}
\item $u_1\in H^1(\R^2)$, extended as zero outside $\Omega$, satisfies the following inequality in sense of distributions: 
\begin{equation}\label{subh}
\Delta u_1+\lambda_1(\Omega)u_1\ge 0\qquad\hbox{in}\qquad \big[C^\infty_c(\R^2)\big]'.
\end{equation}
\item Every point $x_0\in\R^2$ is a Lebesgue point for $u_1$. Pointwise defined as
$$u_1(x_0):=\lim_{r\to0}\mean{B_r(x_0)}{u(x)\,dx},$$
$u_1$ is upper semi-continuous on $\R^2$.
\item $u_1$ is almost subharmonic in sense that for every $x_0\in\R^2$, we have
\begin{equation}\label{quasish}
u_1(x_0)\le \|u_1\|_{L^\infty}\lambda_1(\Omega)r^2+\mean{B_r(x_0)}{u_1(x)\,dx},\qquad\forall r>0.
\end{equation}
\end{itemize}

\subsection{Sets of finite perimeter and reduced boundary}
In the proof of Theorem \ref{mainr2} {\it (iv)} we will need the notion of a reduced boundary. Let $\Omega\subset\R^d$ be a set of finite Lebesgue measure. If the distributional gradient of its characteristic function $\nabla\ind_\Omega$ is a Radon measure such that its total variation $|\nabla\ind_\Omega|(\R^d)$ is finite, then we say that $\Omega$ is of finite perimeter. The perimeter $P(\Omega)$ is the total variation of the gradient and for regular sets coincides with the usual notion of perimeter as surface integral. The reduced boundary $\partial^\ast\Omega$ of a set $\Omega$ of finite perimeter is defined as the set of points where one can define the normal vector to $\Omega$ in the following sense: $x_0\in\partial^\ast\Omega$, if the limit $\ds\lim_{r\to0}\frac{\nabla\ind_\Omega(B_r(x_0))}{|\nabla\ind_\Omega|(B_r(x_0))}$ exists and has Euclidean norm equal to one. We notice that if a point $x_0$ belongs to the reduced boundary, then the density of $\Omega$ in $x_0$ is precisely $1/2$, i.e.
$\ds\lim_{r\to0}\frac{|\Omega\cap B_r(x_0)|}{|B_r(x_0)|}=\frac12$. For more details on the sets of finite perimeter we refer to the books \cite{giusti} and \cite{maggi-book}.

\subsection{The existence theory of Buttazzo and Dal Maso}
The multiphase shape optimization problems of the form \eqref{optF05} admit solutions for a very general cost functional $\mathcal{F}(\Omega_1,\dots,\Omega_h)$. The main existence result in this direction is well known and is due to the classical Buttazzo-Dal Maso result from \cite{budm93}. The price to pay for such a general result is that one has to relax the problem to a wider class of domains, which contains the open ones. Indeed, one notes that the capacitary definition of a Sobolev space \eqref{sob} can be easily extended to generic measurable sets. In particular, it is well known (we refer, for example, to the books \cite{henrot-pierre} and \cite{bucurbuttazzo}) that it is sufficient to restrict the analysis to the class of \emph{quasi-open} sets, i.e. the level sets of Sobolev functions. Since the definition of the first eigenvalue \eqref{lb1} is of purely variational character, one may also extend it to the quasi-open sets and then apply the theorem of Buttazzo and Dal Maso \cite{budm93} to obtain existence for \eqref{optF05} in the family of quasi-open sets under the minimal assumptions of monotonicity and semi-continuity of the function $F$. Thus, the study of the problem of existence of a solution of \eqref{optF05} reduces to the analysis of the regularity of the optimal quasi-open sets.  The precise statement of the Buttazzo-Dal Maso Theorem that we are going to adopt is the following. 
\begin{thm}\label{existh}
Suppose that $\Dr$ is a bounded open sets, $k_1, \dots, k_h$ are natural numbers, $F:\R^h\to\R$ is a continuous function increasing in each variable and let $W_i:\Dr\to[0,+\infty]$ be given measurable functions. Then there is a solution to the problem 
\begin{equation*}
\min\Big\{F(\lambda_{k_1}(\Omega_1),\dots,\lambda_{k_h}(\Omega_h))+\sum_{i=1}^h\int_{\Omega_i}W_i(x)\,dx\ :\ \Omega_i\subset\Dr\ \text{quasi-open}, \Omega_i\cap\Omega_j=\emptyset\Big\}.
\end{equation*} 
\end{thm}

\subsection{Regularity of the optimal sets for the first eigenvalue}
The regularity of the optimal sets for the Dirichlet eigenvalues is a difficult question and even in the case of a single phase it is open for higher eigenvalues. For the principal eigenvalue of the Dirichlet Laplacian we have the following result by Lamboley and Brian\c con which relies on an adaptation of the classical Alt-Caffarelli regularity theory to the case of eigenfunctions. We state the result here with a smooth weight function as in the original paper \cite{altcaf}.
\begin{thm}\label{regth}
Suppose that $\Dr\subset\R^2$ is a bounded open set, $W:\Dr\to[a,A]$ is a smooth function and $\Omega$ is a solution of the shape optimization problem
\begin{equation}
\min\Big\{\lambda_1(\Omega)+\int_{\Omega}W(x)\,dx\ :\ \Omega\subset\Dr\ \text{quasi-open}\Big\}.
\end{equation}
Then $\Omega$ is open set of finite perimeter and the boundary $\Dr\cap \partial\Omega$ is locally a graph of a smooth function.
\end{thm} 

\subsection{Shape subsolutions and their properties}
We say that the quasi-open set $\Omega\subset\R^2$ is a shape subsolution for the functional $\lambda_1+\alpha|\cdot|$ if for every quasi-open set $\omega\subset\Omega$ we have 
$$\lambda_1(\Omega)+\alpha|\Omega|\le \lambda_1(\omega)+\alpha|\omega|.$$
The notion of a shape subsolution was introduced by Bucur in \cite{bucur-mink} in order to study the existence of an optimal set for the $k$th eigenvalue and then was more extensively studied in \cite{buve}. We recall the main results from \cite{bucur-mink} and \cite{buve} in the following

\begin{thm}\label{subth}
Suppose that $\Omega$ is a shape subsolution for the functional $\lambda_1+\alpha|\cdot|$. Then
\begin{enumerate}[(a)]
\item $\Omega$ is bounded and its diameter $\hbox{diam}(\Omega)$ is estimated by a constant depending on $\alpha$, $\lambda_1(\Omega)$ and $|\Omega|$;
\item $\Omega$ is of finite perimeter and we have the estimate
\begin{equation}\label{perbound}
P(\Omega)\le \alpha^{-1/2}\lambda_1(\Omega)|\Omega|^{1/2};
\end{equation}
\item there is a lower bound on the eigenvalue $\lambda_1(\Omega)$ given by
\begin{equation}\label{lb1bound}
\lambda_1(\Omega)\ge \big(4\pi \alpha\big)^{1/2};
\end{equation} 
\item If $\Omega'$ is also a shape subsolution for the same functional such that $\Omega\cap\Omega'=\emptyset$, then there are disjoint open sets $\Dr$ and $\Dr'$ such that $\Omega\subset\Dr$ and $\Omega'\subset\Dr'$.
\end{enumerate}
\end{thm}

\subsection{Monotonicity formulas for eigenfunctions}
The monotonicity formula of Alt-Caffarelli-Friedman is an essential tool in the study of the behavior of the eigenfunctions in the points of the common boundary of the optimal sets. Since the eigenfunctions are not subharmonic, but satisfy \eqref{subh}, we will need another version of the monotonicity formula from \cite{alcafr}. 
We state here the following monotonicity theorem for eigenfunctions from \cite{coteve03}, which is a version of the Alt-Caffarelli-Friedman monotonicity formula. We use this result to prove that the eigenfunctions of the optimal sets are Lipschitz continuous everywhere in $\Dr$.
\begin{thm}[Two-phase monotonicity formula]\label{teo2phm}
Consider the unit ball $B_1\subset\R^2$. Let $u^+,u^-\in H^1(B_1)\cap L^\infty(B_1)$ be two non-negative functions with disjoint supports, i.e. such that $\int_{B_1}u^+u^-\,dx=0$, and let $\lambda_+,\lambda_-\ge 0$ be two real numbers such that 
$$\Delta u^++\lambda_+u^+\ge 0 \qquad \hbox{and} \qquad \Delta u^-+\lambda_-u^-\ge 0.$$ 
Then there are constants $1/2\ge r_0>0$ and $C>0$, depending on $d$, $\lambda_+$ and $\lambda_-$, such that for every $r\in(0,r_0)$ we have
\begin{equation}\label{teo2phm1}
\left(\frac{1}{r^{2}}\int_{B_r}|\nabla u^+|^2\,dx\right)\left(\frac{1}{r^{2}}\int_{B_r}|\nabla u^-|^2\,dx\right)\le C\left(1+\|u^++u^-\|^2_{L^\infty(B_{2r_0})}\right)^2.
\end{equation} 
\end{thm}

We note that the estimate \eqref{teo2phm1} follows by the more general result by Caffarelli, Jerison and K\"enig (see \cite{cajeke} and also the note \cite{mono}, where the continuity assumption was dropped). In order to obtain \eqref{teo2phm2} we use the idea of Conti, Terracini and Verzini (see \cite{coteve03}) that follows the spirit of the original Alt-Caffarelli-Friedman monotonicity formula. It works exclusively for eigenfunctions (linear or nonlinear), but can be easily refined to obtain finer qualitative results as \eqref{teo2phm2}. 

\medskip

The three-phase version of Theorem \ref{teo2phm} is the main tool that allows to exclude the presence of triple boundary points in the optimal configuration. The following three-phase monotonicity formula was proved for eigenfunctions in \cite{coteve03}, while the general three-phase version of the Caffarelli-Jerison-K\"enig result can be found in \cite{buve} (see also \cite{mono} for the detailed proof). This formula is used in the proof of the fact that in the optimal configuration there are not triple points. 
\begin{thm}[Three-phase monotonicity formula]\label{teo3phm}
Consider the unit ball $B_1\subset\R^2$. Let $u_1,u_2,u_3\in H^1(B_1)\cap L^\infty(B_1)$ be three non-negative functions with disjoint supports, i.e. such that $\int_{B_1}u_iu_j\,dx=0$ for all $i\neq j$, and let $\lambda_1,\lambda_2,\lambda_3\ge 0$ be real numbers such that 
$$\Delta u_i+\lambda_iu_i\ge 0,\qquad \forall i=1,2,3.$$ 
Then there are constants $0<r_0\le 1/2$, $C>0$ and $\eps>0$, depending on $d$, $\lambda_1$, $\lambda_2$ and $\lambda_3$, such that for every $r\in(0,r_0)$ we have
\begin{equation}\label{teo3phm1}
\prod_{i=1}^3\left(\frac{1}{r^{2}}\int_{B_r}|\nabla u_i|^2\,dx\right)\le Cr^{\eps}\left(1+\|u_1+u_2+u_3\|^2_{L^\infty(B_{2r_0})}\right)^3.
\end{equation} 
\end{thm}

The three phase monotonicity formula is not just a consequence of the two phase formula. In fact if we apply the Alt-Caffarelli-Friedman formula to each pair of the tree sets $\Omega_i,\Omega_j$ and $\Omega_k$, then in \eqref{teo3phm1} there will be no decay term $r^\eps$. Roughly speaking the presence of the third phase forces the other two to occupy less space which in turn gives some decay with $\eps>0$. The same phenomenon appears when there are only two phases that cannot occupy a certain sufficiently big region. This is the idea that we develop in Lemma \ref{lem2ph05} which we will use to deduce the lack of double points on the boundary of the design region $\Dr$ and also the regularity of the reduced boundary of the auxiliary phase $\Omega_{h+1}$.

\section{Proof of Theorem \ref{mainr2}}
\label{proof-mainres}

\subsection{Existence of optimal open sets}
An existence of an optimal configuration in the class of quasi-open sets follows by the Buttazzo-Dal Maso Theorem. Let $\Omega_1,\dots,\Omega_h$ be the optimal quasi-open sets. Then for every quasi-open set $\omega_i\subset\Omega_i$ we have that the configuration is not optimal which gives that
$$\lambda_1(\omega_i)-\lambda_1(\Omega_i)\ge \int_{\Omega_i}W_i\,dx-\int_{\omega_i}W_i\,dx\ge a|\Omega_i|-a|\omega_i|.$$
Thus $\Omega_i$ is a shape subsolution for the functional $\lambda_1+a|\cdot|$ and so we can apply the result from \cite{buve} Theorem \ref{subth} {\it (d)}. Thus each of the sets $\Omega_i$ is contained in an open set $\Dr_i$ and solves 
$$\min\Big\{\lambda_1(\Omega)+\int_{\Omega}W_i(x)\,dx\ :\ \Omega\subset\Dr_i\ \text{quasi-open}\Big\}.$$
By Theorem \ref{regth} the sets $\Omega_i$ are open.

\subsection{Lipschitz continuity of the eigenfunctions}\label{lip}
In this section we prove that the first eigenfunctions on the optimal sets for \eqref{optsum} are Lipschitz continuous. To fix the notation, in the rest of this section we will denote with $(\Omega_1,\dots,\Omega_h)$ a generic solution of \eqref{optsum} and with $u_i\in H^1_0(\Omega_i)$ the first eigenfunction on $\Omega_i$, i.e. $u_i$ are non-negative function such that $\int_{\R^2}u_i^2\,dx=1$ satisfying \eqref{davies}, \eqref{subh} and the equation
$$-\Delta u_i=\lambda_1(\Omega_i)u_i,\qquad u_i\in H^1_0(\Omega),$$ 
weakly in $H^1_0(\Omega_i)$.

{\bf Non-degeneracy of the eigenfunctions.} We first note that for every $\omega_i\subset\Omega_i$, the optimality of $(\Omega_1,\dots,\Omega_i,\dots,\Omega_h)$ tested against the $h$-uple of open sets $(\Omega_1,\dots,\omega_i,\dots,\Omega_h)$ gives the inequality
$$\lambda_1(\Omega_i)+\alpha|\Omega_i|\le \lambda_1(\omega_i)+\alpha|\omega_i|,$$
i.e. $\Omega_i$ is a \emph{subsolution} for the functional $\lambda_1+\alpha|\cdot|$. Thus using the argument from the Alt-Caffarelli non-degeneracy lemma (see \cite[Lemma 3.4]{altcaf} and also \cite[Section 3]{buve}), we have the following result.

\begin{lemma}\label{nondeglem}
Suppose that $(\Omega_1,\dots,\Omega_h)$ is optimal for \eqref{optsum}. Then there are constants $C_{nd}$ and $r_0>0$ such that for all the first eigenfunctions $u_i$, every $0<r\le r_0$ and every $x_0\in\R^2$ we have the following implication
\begin{equation}\label{nondeg}
\Big(\ B_{r/2}(x_0)\cap\Omega_i\neq\emptyset \ \Big)\ \Rightarrow \ \Big(\ \frac1r\mean{B_r(x_0)}{u_i\,dx}\ge C_{nd}\ \Big).
\end{equation}
\end{lemma}

\begin{rem}
Together with the estimate \eqref{quasish}, Lemma \ref{nondeglem} gives that there is $r_0>0$ such that
\begin{equation}\label{intest}
\|u_i\|_{L^\infty(B_{r/2}(x_0))}\le 5\mean{B_r(x_0)}{u_i\,dx},\qquad\forall r\le r_0\ \ \hbox{such that}\ \ B_{r/2}(x_0)\cap\Omega_i\neq\emptyset.
\end{equation}
\end{rem}

On the common boundary of two optimal sets the non-degeneracy \eqref{nondeg} of the mean $\meantext{B_r(x_0)}{u_i\,dx}$ gives a bound from below for the gradient $\meantext{B_r(x_0)}{|\nabla u_i|^2\,dx}$. This fact follows by the elementary lemma proved below.

\begin{lemma}\label{mainineqlem}
Let $R>0$, $B_{R}(x_0)\subset\R^2$ and $U\in H^1(B_R(x_0))$ be a Sobolev function such that for almost every $r\in(0,R)$ the set $\{U=0\}\cap \partial B_r(x_0)$ is non-empty. Then we have
\begin{equation}\label{mainineq}
\frac1R\mean{B_R(x_0)}{U\,d\HH^1}\le2\left(\mean{B_R(x_0)}{|\nabla U|^2\,dx}\right)^{1/2}.
\end{equation}
\end{lemma} 
\begin{proof}
Without loss of generality we suppose that $x_0=0$. We first note that for almost every $r\in(0,R)$ the restriction $U|_{\partial B_r}$ is Sobolev. If, moreover, $\{U=0\}\cap\partial B_r\neq\emptyset$, then we have 
$$\int_{\partial B_r}U^2\,d\HH^1\le {4r^2}\int_{\partial B_r}|\nabla U|^2\,d\HH^1.$$
Applying the Cauchy-Schwartz inequality and integrating for $r\in (0,R)$, we get 
$$\Big(\frac{1}{R}\mean{B_R}{U\,dx}\Big)^2\le \frac{1}{R^2}\mean{B_R}{U^2\,dx}\le 4\mean{B_R}{|\nabla U|^2\,dx}.$$
\end{proof}

\begin{cor}\label{nondegradcor}
Suppose that $(\Omega_1,\dots,\Omega_h)$ is optimal for \eqref{optsum}. Then there is a constant $r_0>0$ such that for every $x_0\in\partial\Omega_i\cap\partial\Omega_j$, for some $i\neq j$ we have 
\begin{equation}\label{nondegrad}
\mean{B_r(x_0)}{|\nabla u_i|^2\,dx}\ge 4C_{nd}^2,\forall r\in(0,r_0),
\end{equation}
where $C_{nd}>0$ is the non-degeneracy constant from Lemma \ref{nondeglem}.
\end{cor}
\begin{proof}
Since $x_0\in\partial\Omega_i\cap\partial\Omega_j$, we have that for every $r>0$  $\Omega_i\cap B_r(x_0)\neq \emptyset$ and $\Omega_j\cap B_r(x_0)\neq \emptyset$. In view of Lemma \ref{nondeglem}, it is sufficient to check that $\Omega_i\cap\partial B_r(x_0)\neq\emptyset$ and $\Omega_j\cap\partial B_r(x_0)\neq0$, for almost every $r\in(0,r_0)$. Indeed, suppose that this is not the case and that $\Omega_i\cap\partial B_r(x_0)=\emptyset$. Since $\Omega_i$ is connected, we have that $\Omega_i\subset B_r(x_0)$, which gives $\lambda_1(\Omega_i)\ge \lambda_1(B_{r_0})$, which is impossible if we choose $r_0$ small enough.
\end{proof}

{\bf Growth estimate of the eigenfunctions on the boundary}. We now prove the two key estimates of the growth of $u_i$ close to the boundary $\partial\Omega_i$. We consider two kinds of estimates, one holds around the points, where two phases $\Omega_i$ and $\Omega_j$ are close to each other, and is reported in Lemma \ref{C2lemma}. The other estimate concerns the one-phase points,  i.e. the points on one boundary, say $\partial\Omega_i$, which are far away from all other sets $\Omega_j$.

\begin{lemma}\label{C2lemma}
Suppose that $(\Omega_1,\dots,\Omega_h)$ is optimal for \eqref{optsum}. Then there are constants $C_2$ and $r_0>0$ such that if $x_0\in\partial\Omega_i$ is such that
$\Omega_j\cap B_r(x_0)\neq\emptyset$, for some $j\neq i$ and $r\le r_0$, then 
\begin{equation}\label{C2est}
\|u_i\|_{L^\infty(B_r(x_0))}\le C_2r.
\end{equation}
\end{lemma}
\begin{proof}
Without loss of generality we suppose that $0=x_0\in\partial\Omega_i$. Let now $0<r\le r_0$ be such that $\Omega_j\cap B_r\neq\emptyset$. Choosing $r_0$ small enough we may apply Lemma \ref{nondeg} obtaining that 
$$\mean{B_{3r}}{u_j\,dx}\ge 3C_{nd}\,r.$$
Again by choosing $r_0$ small enough we may suppose that for every $r\in(0,r_0)$ we have $\partial B_{3r}\cap\Omega_i\neq0$. Indeed, if this is not the case for some $r$, then the set $\Omega_i$ is entirely contained in $B_{3r}$ and so $\lambda_1(\Omega_i)\ge \lambda_1(B_{3r})\ge \lambda_1(B_{3r_0})$, contradicting the optimality of $\Omega_i$. Thus, we may apply the estimate \eqref{mainineq} for $u_j$ obtaining 
$$C_{nd}^2\le \Big(\frac1{3r}\mean{B_{3r}}{u_j\,dx}\Big)^2\le 4\mean{B_{3r}}{|\nabla u_j|^2\,dx}.$$
By the two-phase monotonicity formula applied for $u_i$ and $u_j$, we get that there is a constant $C>0$ such that 
$$\frac{4C}{C_{nd}^2}\ge \mean{B_{3r}}{|\nabla u_i|^2\,dx}.$$
Since $B_r\cap\Omega_j\neq\emptyset$, by choosing $r_0$ small enough an reasoning as above we may suppose that for every $\tilde r\in(r,3r)$ $\partial B_{\tilde r}\cap\Omega_j\neq 0$. Thus, reasoning as in Lemma \ref{mainineqlem}, we get that 
$$4(3r)^2\int_{B_{3r}\setminus B_{2r}}|\nabla u_i|^2\,dx\ge \int_{B_{3r}\setminus B_{2r}}u_i^2\,dx\ge \frac{1}{5\pi r^2}\Big(\int_{B_{3r}\setminus B_{2r}}u_i\,dx\Big)^2.$$

By the mean value formula, there is $R\in(2r,3r)$ such that
\begin{equation}\label{C2est1}
\int_{\partial B_{R}}u_i\,dx\le \frac{1}{r}\int_{2r}^{3r}\Big(\int_{\partial B_s}u_i\,d\HH^1\Big)\,ds\le 27r\Big(\int_{B_{3r}}|\nabla u_i|^2\,dx\Big)^{1/2}
\end{equation} 
We now note that by \eqref{subh} the function $v(x)=u_i(x)-\lambda_1(\Omega_i)\|u_i\|_{L^\infty}(R^2-|x|^2)$ is subharmonic. Then, for every $x\in B_r$, we use the Poisson formula 
\begin{equation}\label{C2est2}
\begin{array}{ll}
\ds u_i(x)-\lambda_1(\Omega_i)\|u_i\|_{L^\infty}(3r)^2&\ds\le \frac{R^2-|x|^2}{2\pi R}\int_{\partial B_{R}}\frac{u_i(y)}{|y-x|^2}\,d\HH^{1}(y)\le 9\mean{\partial B_{R}}{u_i\,d\HH^1}.
\end{array}
\end{equation}
Using the non-degeneracy of $u_i$ (Lemma \ref{nondeglem}) and combining the estimates from \eqref{C2est1} and \eqref{C2est2} we get
\begin{equation}\label{C2est3}
\|u_i\|_{L^\infty(B_r)}\le 3^6 r\Big(\int_{B_{3r}}|\nabla u_i|^2\,dx\Big)^{1/2}\le \frac{2 \sqrt{C} 3^6}{C_{nd}}r.
\end{equation}
\end{proof}

The following Lemma is similar to \cite[Lemma 3.2]{altcaf} and \cite[Lemma 3.1]{brla}. We sketch the proof below for the sake of completeness.

\begin{lemma}\label{C1lemma}
Suppose that $(\Omega_1,\dots,\Omega_h)$ is optimal for \eqref{optsum}. Then there are constants $C_1>0$ and $r_0>0$ such that if $x_0\in\partial\Omega_i$ and $0<r\le r_0$ are such that
$\Omega_j\cap B_{2r}(x_0)=\emptyset$, for every $j\neq i$, then 
\begin{equation}\label{C1est}
\|u_i\|_{L^\infty(B_r(x_0))}\le C_1r.
\end{equation}
\end{lemma}
\begin{proof}
Without loss of generality we may suppose that $x_0=0$. Since $\Omega_j\cap B_{2r}=\emptyset$, for every $j\neq i$, we may use the $h$-uple $(\Omega_1,\dots,\Omega_i\cap B_{2r},\dots,\Omega_h)$ to test the optimality of $(\Omega_1,\dots,\Omega_i,\dots,\Omega_h)$. Thus we have
\begin{equation}\label{C1est1}
\begin{array}{ll}
\ds\int_{\R^2}|\nabla u_i|^2\,dx+\alpha|\Omega_i|&\ds=\lambda_1(\Omega_i)+\alpha|\Omega_i|\le \lambda_1(\Omega_i\cap B_{2r})+\alpha|\Omega_i\cap B_r|\\
&\ds\le \frac{\int_{\R^2}|\nabla \widetilde u_i|^2\,dx}{\int_{\R^2}\widetilde u_i^2\,dx}+\alpha|\Omega_i\cap B_{2r}|\le \int_{\R^2}|\nabla \widetilde u_i|^2\,dx+\alpha|\Omega_i\cap B_{2r}|,
\end{array}
\end{equation}
where we used the test function $\widetilde u_i\in H^1_0(\Omega_i\cap B_{2r})$ defined as $\widetilde u_i=v_i\ind_{B_{2r}}+u_i\ind_{B_{2r}^c}$ and $v_i\in H^1(B_{2r})$ is the solution of the obstacle problem
\begin{equation}
\min\Big\{\int_{B_{2r}}|\nabla v|^2\,dx:\ v\in H^1(B_{2r}),\ v-u_i\in H^1_0(B_{2r}),\ v\ge u_i\Big\}.
\end{equation}
By \eqref{C1est1} an the fact that $v_i$ is harmonic on the set $\{v_i>u_i\}$, we get
\begin{equation}\label{C1est2}
\int_{B_{2r}}|\nabla (u_i-v_i)|^2\,dx=\int_{B_{2r}}\big(|\nabla u_i|^2-|\nabla v_i|^2\big)\,dx\le \alpha|B_{2r}\setminus \Omega_i|.
\end{equation}
Now, reasoning as in \cite[Lemma 3.2]{altcaf} (see also \cite[Lemma 4.3.20]{tesi} and \cite{buve}), there is a constant $C>0$ such that 
\begin{equation}\label{C1est3}
\big|\{u_i=0\}\cap B_{2r}\big|\left(\frac1{2r}\mean{\partial B_{2r}}{u_i\,d\HH^1}\right)^2\le C\int_{B_{2r}}|\nabla (u_i-v_i)|^2\,dx.
\end{equation}
Now we note that by the optimality of $\Omega_i$, we have $\Omega_i=\{u_i>0\}$ and $|B_{2r}\cap\{u_i=0\}|>0$ (if $|B_{2r}\cap\{u_i=0\}|=0$, then by the optimality $v_i=u_i$ in $B_{2r}$; thus $u_i$ is superharmonic in $B_{2r}$ and so $u_i>0$ in $B_{2r}$, which contradicts the assumption $0\in\partial \Omega_i$).
Now \eqref{C1est2} and \eqref{C1est3} give  
\begin{equation}\label{C1est31/2}
\frac1{2r}\mean{\partial B_{2r}}{u_i\,d\HH^1}\le\sqrt{C/\alpha}.
\end{equation}
Since the function $\Big\{x\mapsto \Big(u_i(x)-\lambda_1(\Omega_i)\|u_i\|_{L^\infty}(4r^2-|x|^2)\Big)\Big\}$ is subharmonic, we can use the Poisson formula for every $x\in B_{r}$
\begin{equation}\label{C1est4}
\begin{array}{ll}
\ds u_i(x)-4\lambda_1(\Omega_i)\|u_i\|_{L^\infty}r^2&\ds\le \frac{(2r)^2-|x|^2}{4\pi r}\int_{\partial B_{2r}}\frac{u_i(y)}{|y-x|^2}\,d\HH^{1}(y)\le 4\mean{\partial B_{2r}}{u_i\,d\HH^1}.
\end{array}
\end{equation}

By the non-degeneracy of $u_i$ (Lemma \ref{nondeglem}) and \eqref{C1est4}, we have that for $r_0$ small enough 
\begin{equation*}
\frac{\|u_i\|_{L^\infty(B_{r})}}{r}\le \frac5{2r}\mean{\partial B_{2r}}{u_i\,d\HH^1}\le5\sqrt{C/\alpha},
\end{equation*}
which gives the claim.
\end{proof}

We combine the estimates from Lemma \ref{C1lemma} and Lemma \ref{C2lemma}, obtaining the following  

\begin{prop}\label{C12prop}
Suppose that $(\Omega_1,\dots,\Omega_h)$ is optimal for \eqref{optsum}. Then there are constants $r_0>0$ and $C_{12}>0$ such that for every $i\in\{1,\dots,h\}$ we have 
\begin{equation}\label{C12est}
\|u_i\|_{L^\infty(B_r(x_0))}\le C_{12}\, r, \qquad \forall r\in(0,r_0).
\end{equation}
\end{prop} 

{\bf Conclusion of the proof of the Lipschitz continuity of the eigenfunctions.} We now use the estimate from Proposition \ref{C12prop} to deduce the Lipschitz continuity of $u_i$. The argument is standard and we recall it briefly for the sake of completeness. It is based on the following classical lemma.
\begin{lemma}\label{gradestlem}
Suppose that $B_r\subset\R^2$, $f\in L^\infty(B_r)$ and $u\in H^1(B_r)$ satisfies the equation 
$$-\Delta u=f\quad\hbox{weakly in}\quad [H^1_0(B_r)]'.$$ 
Then there is a dimensional constant $C>0$ such that the following estimate holds
\begin{equation}\label{gradest}
\|\nabla u_i\|_{L^\infty(B_{r/2})}\le C\left(\|f\|_{L^\infty(B_r)}+\frac{\|u\|_{L^\infty(B_r)}}{r}\right).
\end{equation} 
\end{lemma}

We prove the following result which implies Theorem \ref{mainr2} {\it (iii)} since if the bounded open set $\Dr\subset'\R^2$ has boundary of class $C^2$, then the function $w_\Dr$ defined below is Lipschitz continuous on $\overline\Dr$.
\begin{thm}
Let $\Dr\subset\R^2$ be a bounded open set. Let $(\Omega_1,\dots,\Omega_h)$ be optimal for \eqref{optsum}. Then the corresponding first eigenfunctions $u_1,\dots, u_h$ are locally Lipschitz continuous in $\Dr$. If, moreover, $\Dr$ is such that the weak solution $w_\Dr$ of the problem 
$$-\Delta w_{\Dr}=1,\qquad w_\Dr\in H^1_0(\Dr),$$
is Lipschitz continuous on $\R^2$, then the first eigenfunctions $u_1,\dots, u_h$ are globally Lipschitz continuous on $\R^2$.
\end{thm} 
\begin{proof}
Let $r_0>0$ be the constant from Proposition \ref{C12prop} and fix $r_1\le r_0/2$. Let $x_0\in\Omega_i$ be such that $\hbox{dist}(x_0,\partial\Dr)\ge r_1$. If $r:=\hbox{dist}(x_0,\partial\Omega_i)\ge r_1$, then by \eqref{gradest}, we have 
\begin{equation}\label{gradest1}
|\nabla u_i (x_0)|\le C\left(\lambda_1(\Omega_i)+r_1^{-1}\right)\|u_i\|_{L^\infty}.
\end{equation} 
If $r:=\hbox{dist}(x_0,\partial\Omega_i)<r_1$, then we set $y_0\in\partial\Omega_i$ to be such that $|x_0-y_0|=\hbox{dist}(x_0,\partial\Omega_i)$. Using Proposition \ref{C12prop} and again \eqref{gradest}, we have 
\begin{equation}\label{gradest2}
\begin{array}{ll}
\ds|\nabla u_i (x_0)|&\ds\le C\left(\lambda_1(\Omega_i)\|u_i\|_{L^\infty}+\frac{\|u_i\|_{L^\infty(B_r(x_0))}}{r}\right)\\
\\
&\ds\le C\left(\lambda_1(\Omega_i)\|u_i\|_{L^\infty}+\frac{\|u_i\|_{L^\infty(B_{2r}(y_0))}}{r}\right)\le C\Big(\lambda_1(\Omega_i)\|u_i\|_{L^\infty}+2C_{12}\Big),
\end{array}
\end{equation} 
which gives the local Lipschitz continuity of $u_i$. 

If the function $w_\Dr$ is Lipschitz continuous on $\R^d$, we consider for every point $x_0\in\Omega_i$ two possibilities for $r:=\hbox{dist}(x_0,\partial\Omega_i)$: if $3r\ge\hbox{dist}(x_0,\partial\Dr)$, then the maximum principle $u_i\le \lambda_1(\Omega_i)\|u_i\|_{L^\infty}w_\Dr$ and the gradient estimate \eqref{gradest} gives
\begin{equation}\label{gradest3}
\begin{array}{ll}
\ds|\nabla u_i (x_0)|&\ds\le C\left(\lambda_1(\Omega_i)\|u_i\|_{L^\infty}+\frac{\|u_i\|_{L^\infty(B_r(x_0))}}{r}\right)\\
\\
&\ds\le C\lambda_1(\Omega_i)\|u_i\|_{L^\infty}\left(1+\frac{\|\nabla w_\Dr\|_{L^\infty}}{r}\big(\hbox{dist}(x_0,\partial\Dr)+r\big)\right)\\
\\
&\ds\le C\lambda_1(\Omega_i)\|u_i\|_{L^\infty}\Big(1+4\|\nabla w_\Dr\|_{L^\infty}\Big).
\end{array}
\end{equation} 
If $3r\le\hbox{dist}(x_0,\partial\Dr)$ and $r\le r_0/2$, then the gradient estimate \eqref{gradest} gives again \eqref{gradest2}. If $r\ge r_0/2$, then we have \eqref{gradest1} with $r_1=r_0/2$ and this concludes the proof. 
\end{proof} 
 
\subsection{A monotonicity formula with decay} 
In order to prove the lack of double points on the boundary of $\Dr$ and the regularity of the auxiliary phase $\Omega_{h+1}$ we will need special type of a two phase monotonicity formula in which the supports of the eigenfunctions cannot invade certain prescribed zone. In this case the product of the two gradients $\ds \left(\frac{1}{r^{2}}\int_{B_r}|\nabla u^+|^2\,dx\right)\left(\frac{1}{r^{2}}\int_{B_r}|\nabla u^-|^2\,dx\right)$ decays as $r\to0$. The result is in the spirit of the three phase formula but the proof follows the idea of the proof of the two phase formula that was carried out in \cite{coteve03}. 
\begin{lemma}\label{lem2ph05}
Consider the unit ball $B_1\subset\R^2$. Let $u^+,u^-\in H^1(B_1)\cap L^\infty(B_1)$ be two non-negative functions with disjoint supports, i.e. such that $\int_{B_1}u^+u^-\,dx=0$, and let $\lambda_+,\lambda_-\ge 0$ be two real numbers such that 
$$\Delta u^++\lambda_+u^+\ge 0 \qquad \hbox{and} \qquad \Delta u^-+\lambda_-u^-\ge 0.$$ 
If, moreover, the set $\Omega:=B_1\cap\{u^+=0\}\cap \{u^-=0\}$ has positive density in $0$, in the sense that 
$$\liminf_{r\to 0}\frac{|\Omega\cap B_r|}{|B_r|}=c>0,$$
then there is some $\eps>0$, depending on $d$, $\lambda_+$, $\lambda_-$ and $c$ such that 
\begin{equation}\label{teo2phm2}
\left(\frac{1}{r^{2}}\int_{B_r}|\nabla u^+|^2\,dx\right)\left(\frac{1}{r^{2}}\int_{B_r}|\nabla u^-|^2\,dx\right)=o(r^\eps).
\end{equation} 
\end{lemma}

The proof of Lemma \ref{lem2ph05} is based on Lemma \ref{lemma2phm}, which involves the auxiliary functions $U_1$ and $U_2$ constructed below.
Let $\lambda:=\max\{\lambda_+,\lambda_-\}$ and let $r_0>0$ be small enough such that there is a positive radially symmetric function $\vf\in H^1(B_{r_0})$ satisfying 
\begin{equation}\label{2phm1}
-\Delta \vf=\lambda\vf\ \ \hbox{in}\ \  B_{r_0}, \qquad 0<a\le \vf\le b,
\end{equation}
for some constants $0<a\le b$ depending on $d$, $\lambda$ and $r_0$.
We now introduce the notation 
\begin{equation}\label{2phm2}
U_1:=\frac{u^+}{\vf} \qquad \hbox{and} \qquad U_2:=\frac{u^-}{\vf}.
\end{equation}        
\begin{rem}\label{oss2phm}
A direct computation of the gradient of $U_i$ on $B_{r_0}$ gives 
\begin{equation*}
\begin{array}{ll}
\nabla U_1=&\vf^{-1}\nabla u^+-\vf^2 u^+\nabla\vf.
\end{array}
\end{equation*}\end{rem}
We define the function $\Phi:[0,r_0]\to\R^+$ as
\begin{equation}\label{2phmPhi}
\Phi(r):=\left(\frac{1}{r^{2}}\int_{B_r}\vf^2|\nabla U_1|^2\,dx\right)\left(\frac{1}{r^{2}}\int_{B_r}\vf^2|\nabla U_2|^2\,dx\right).
\end{equation}

\begin{lemma}\label{lemma2phm}
Consider the unit ball $B_1\subset\R^2$. Let $u^+,u^-\in H^1(B_1)\cap L^\infty(B_1)$ be as in Lemma \ref{lem2ph05} and let $\Phi:[0,r_0]\to\R^+$ be given by \eqref{2phmPhi}. Then 
\begin{enumerate}[(a)]
\item $\Phi$ is increasing on the interval $(0,r_0)$;
\item If, moreover, the set $\Omega:=B_1\cap\{u^+=0\}\cap \{u^-=0\}$ has positive density in $0$, then there are constants $C>0$ and $\eps>0$ such that
$$\frac{1}{r^\eps}\Phi(r)\le \frac{C}{r_0^\eps}\Phi(r_0).$$ 
\end{enumerate}
\end{lemma}
\begin{proof}
We first estimate the derivative of $\Phi$, using the notations $\nabla_n u$ and $\nabla_\tau u$ respectively for the normal and the tangential part of the gradient $\nabla u$ on the boundary of $\partial B_r$.
\begin{align}
\frac{\Phi'(r)}{\Phi(r)}&=-\frac{4}{r}+\sum_{i=1,2}\frac{\int_{\partial B_r}\vf^2|\nabla U_i|^2\,d\HH^{1}}{\int_{B_r}\vf^2|\nabla\widetilde U_i|^2\,dx}\nonumber\\
&\ge-\frac{4}{r}+\sum_{i=1,2}\frac{\int_{\partial B_r}\vf^2\big(|\nabla_\tau U_i|^2+|\nabla_n U_i|^2\big)\,d\HH^{1}}{\int_{\partial B_r}\vf^2 U_i|\nabla_n U_i|\,d\HH^{1}}\label{longeq1}\\
&\ge-\frac{4}{r}+\sum_{i=1,2}\frac{2\left(\int_{\partial B_r}\vf^2|\nabla_n U_i|^2\,d\HH^{1}\right)^{1/2}\left(\int_{\partial B_r}\vf^2|\nabla_\tau U_i|^2\,d\HH^{1}\right)^{1/2}}{\left(\int_{\partial B_r}\vf^2 U_i^2\,d\HH^{1}\right)^{1/2}\left(\int_{\partial B_r}\vf^2 |\nabla_n U_i|^2\,d\HH^{1}\right)^{1/2}}\label{longeq2}\\
&=-\frac{4}{r}+2\sum_{i=1,2}\left(\frac{\int_{\partial B_r}|\nabla_\tau U_i|^2\,d\HH^{1}}{\int_{\partial B_r} U_i^2\,d\HH^{1}}\right)^{1/2}\label{longeq3}\\
&\ge-\frac{4}{r}+ 2\sum_{i=1,2}\sqrt{\lambda_1(\partial B_r\cap\{U_i>0\})}\nonumber\\
&\ge-\frac{4}{r}+\sum_{i=1,2}\frac{2\pi}{\HH^{1}(\partial B_r\cap\{U_i>0\})},\label{longeq4}
\end{align}
where \eqref{longeq1} follows by integration by parts and the inequality $-\hbox{div}(\varphi^2\nabla U_i)\leq 0$ obtained using Remark \ref{oss2phm}; \eqref{longeq2} is obtained by applying the mean quadratic-mean geometric inequality in the nominator and the Cauchy-Schwartz inequality in the denominator; \eqref{longeq3} is due to the fact that $\vf$ is constant on $\partial B_r$; \eqref{longeq4} follows by a standard symmetrization argument. Setting 
$$\theta(r):=\frac{\HH^1(\Omega\cap\partial B_r)}{\HH^1(\partial B_r)},$$
and applying the mean arithmetic-mean harmonic inequality to \eqref{longeq4}, we get
\begin{equation}\label{mainder}
\frac{\Phi'(r)}{\Phi(r)}\ge \frac{4}{r}\Big(-1+\frac{1}{1-\theta(r)}\Big)\ge \frac{4\theta(r)}{r},
\end{equation} 
which gives {\it (a)}. In order to prove $(b)$, we note that for $r_0>0$ small enough we have the density estimate 
$$|\Omega\cap B_r|\ge c| B_r|,\qquad \forall 0<r\le r_0.$$
Using the fact that $\frac{\partial}{\partial r}|\Omega \cap B_r|=\HH^1(\Omega\cap\partial B_r)=2\pi r\theta(r)$ we get
\begin{equation}\label{aptheta1}
\int_0^r 2\pi s(\theta(s)-c)\,ds\ge 0,\qquad \forall r\in(0,r_0).
\end{equation}
As a consequence we have that 
\begin{equation}\label{aptheta100}
\int_{rc/2}^r 2\pi s\left(\theta(s)-\frac{c}2\right)\,ds\ge 0,\qquad \forall r\in(0,r_0).
\end{equation}
Indeed, if this is not the case, then 
$$0\le \int_0^r 2\pi s(\theta(s)-c)\,ds\le \int_0^{cr/2} 2\pi s(1-c)\,ds-\int_{cr/2}^r 2\pi s\frac{c}2\,ds\le-\pi r^2c(1-c)^2,$$
which is in contradiction with \eqref{aptheta1}. By \eqref{aptheta100}, we get that there is a constant $c_0>0$ such that
\begin{equation}\label{aptheta101}
\int_{rc/2}^r\theta(s)\,ds\ge c_0 r,\qquad \forall r<r_0.
\end{equation}
By \eqref{mainder} we have 
\begin{align*}
\log\big(r^{-\eps}\Phi(r)\big)-\log\Big(\big(rc/2\big)^{-\eps}\Phi(rc/2)\Big)&=\int_{rc/2}^r\left(-\frac\eps{s}+
\frac{\Phi'(s)}{\Phi(s)}\right)\,ds\\
&\ge \int_{rc/2}^r\frac{4}{s}\left(-\frac{\eps}{4}+\theta(s)\right)\,ds\ge \eps\log(c/2)+4c_0,
\end{align*}
which is positive for $\eps>0$ small enough. Thus, we obtain that the sequence
$$a_n:=r_n^{-\eps}\Phi(r_n),\qquad\hbox{where}\quad r_n=(c/2)^nr_0,$$
is decreasing and so, by rescaling we obtain {\it (b)}.
\end{proof}

\begin{proof}[Proof of Lemma \ref{lem2ph05}]
We first note that as a consequence of Remark \ref{oss2phm}, we have the estimates:
\begin{equation}\label{teo2phm3}
\begin{array}{rl}
\ds\int_{B_r}\frac{|\nabla u^{\pm}|^2}{|x|^{d-2}}\,dx&\ds\le\ 2\int_{B_r}\vf^2\frac{|\nabla U_{12}|^2}{|x|^{d-2}}\,dx+2\|\vf^{-1}\nabla\vf\|^2_{L^\infty(B_{r_0})}\int_{B_r}\frac{u^2}{|x|^{d-2}}\,dx,\\
\\
\ds\int_{B_r}\vf^2\frac{|\nabla U_{12}|^2}{|x|^{d-2}}\,dx&\ds\le\  2\int_{B_r}\frac{|\nabla u^{\pm}|^2}{|x|^{d-2}}\,dx+2\|\vf^{-1}\nabla\vf\|^2_{L^\infty(B_{r_0})}\int_{B_r}\frac{u^2}{|x|^{d-2}}\,dx.
\end{array}
\end{equation}
Taking in consideration the inequality 
\begin{equation}
\int_{B_{r_0}}\frac{|\nabla u^{\pm}|^2}{|x|^{d-2}}\,dx\le C\left(1+\int_{B_{2r_0}}|u^\pm|^2\,dx\right), 
\end{equation}
proved in \cite{cajeke}, we obtain the claim by Lemma \ref{lemma2phm} and simple arithmetic.
\end{proof}

\subsection{Multiphase points and regularity of the free boundary}
This subsection is dedicated to the proof of {\it (i)}, {\it (ii)} and {\it (iv)} of Theorem \ref{mainr2}.
 
{\bf Lack of triple points.} The lack of triple points was proved in \cite{buve} in the more general case of partitions concerning general functionals depending on the spectrum of the Dirichlet Laplacian. The original proof uses the notion of an energy subsolution. In the present case the lack of triple points follows directly. In fact if there are three phases $\Omega_i$, $\Omega_j$, $\Omega_k$ such that the intersection of their boundaries contains a point $x_0$, then by the non-degeneracy of the gradient (Corollary \ref{nondegradcor}) we have that the product $\ds \prod_{i=1}^3\left(\frac{1}{r^{2}}\int_{B_r}|\nabla u_i|^2\,dx\right)$ remains bounded from below by a strictly positive constant, which is in contradiction with the three-phase monotonicity formula (Theorem \ref{teo3phm}).

{\bf Lack of two-phase points on the boundary of the box.}
Our first numerical simulations showed the lack of double points (i.e. points on the boundary of two distinct sets) on the boundary of the box $\Dr$. We first notice that there is a quick argument that proves the above claim in the case when the boundary $\partial\Dr$ is smooth. Indeed, if this is the case and if $x_0\in\partial\Dr$, then there is a ball $B\subset\Dr^c$ such that $x_0\in\partial B$. Since the gradient of the first eigenfunction $u$ on $B$ satisfies the non-degeneracy inequality \eqref{nondegrad}, we can use the three-phase monotonicity formula to conclude the proof. 

If the boundary $\partial\Dr$ is only Lipschitz we need to use Lemma \ref{lem2ph05}. Suppose, by absurd, that there is a point $x_0\in \partial\Omega_i\cap\partial\Omega_j\cap\partial\Dr$. If $u_i$ and $u_j$ are the first eigenfunctions on $\Omega_i$ and $\Omega_j$, by Corollary \ref{nondegrad} we have 
\begin{equation}\label{double1}
\mean{B_r(x_0)}{|\nabla u_i|^2\,dx}\ge C_{nd}\qquad\hbox{and}\qquad\mean{B_r(x_0)}{|\nabla u_j|^2\,dx}\ge C_{nd},
\end{equation}
for small enough $r>0$ and some non-degeneracy constant $C_{nd}>0$. Since $\partial\Dr$ is Lipschitz, we have the density estimate $\ \ds\liminf_{r\to0}\frac{|\Dr^c\cap B_r(x_0) |}{|B_r|}>0\ $ and so, we can apply Lemma \ref{lem2ph05}, obtaining a contradiction.

{\bf Regularity of the auxiliary set $\ds\Omega_{h+1}=\Dr\setminus\Big(\bigcup_{i=1}^h\Omega_i\Big)$.}  We first notice that since each of the sets $\Omega_1,\dots, \Omega_h$ is a shape subsolution for $\lambda_1+a|\cdot|$, we have that each of these sets has finite perimeter by Theorem \ref{subth}. As a consequence $\Omega_{h+1}$ also has finite perimeter. Suppose that $x_0\in \Dr\cap\partial^\ast\Omega_{h+1}$. 

Suppose that $x_0$ is on the boundary of at most one phase, i.e. that there is ball $B_r(x_0)$ and an index $i\in\{1,\dots,h\}$ such that $B_r(x_0)=\big(B_r(x_0)\cap\Omega_i\big)\cup\big(B_r(x_0)\cap\Omega_{h+1}\big)$. Then the set $\Omega_i$ is a solution of 
$$\min\Big\{\lambda_1(\Omega)+\int_{\Omega}W_i(x)\,dx\ :\ \Omega\subset \Dr_i\cap B_r(x_0),\ \Omega\ \text{open}\Big\},$$
where the set $D_i$ is given by Theorem \ref{subth}. By the regularity result of Brian\c con and Lamboley Theorem \ref{regth} we have that $\partial^\ast\Omega_{h+1}=\partial\Omega_{h+1}$ in $B_r(x_0)$ and is locally a graph of a smooth function.

Thus in order to conclude it is sufficient to prove that $x_0$ belonging to the boundary of just one of the phases is the only possible case. Indeed, suppose that there is $j\neq i$ such that for every ball $B_r(x_0)$ the sets $B_r(x_0)\cap\Omega_i$ and $B_r(x_0)\cap\Omega_j$ are both non-empty. By the non-degeneracy of the gradients of the eigenfunctions $u_i$ and $u_j$ we have that $\ds\int_{B_r(x_0)}|\nabla u_{i}|^2\,dx\le C_{nd} r^2$ and  $\ds\int_{B_r(x_0)}|\nabla u_{j}|^2\,dx\le C_{nd} r^2$. On the other hand, since $x_0$ is in the reduced boundary of $\Omega_{h+1}$ we have that 
$$\lim_{r\to 0}\frac{|B_r(x_0)\cap \Omega_{h+1}|}{|B_r(x_0)|}=\frac12.$$
Thus by the decay monotonicity formula Lemma \ref{lem2ph05} we get
$$\lim_{r\to0}\left(\frac{1}{r^2}\int_{B_r(x_0)}|\nabla u_{i}|^2\,dx\right) \left(\frac{1}{r^2}\int_{B_r(x_0)}|\nabla u_{j}|^2\,dx\right) =0,$$
which is a contradiction. Thus every point of the reduced boundary belongs to at most one phase and $\partial^\ast\Omega_{h+1}$ is smooth.

\section{Further remarks and open questions}
This section is dedicated to some further developments around Theorem \ref{mainr2}. In particular, using the decay monotonicity formula from Lemma \ref{lem2ph05} and the same argument as in Theorem \ref{mainr2} {\it (iv)} we prove that the optimal set for the second eigenfunction has smooth reduced boundary. We also discuss the extension of Theorem \ref{mainr2} to smooth surfaces and the analogous result in this case.  
\subsection{On the regularity of the optimal set for $\lambda_2$}
Consider the shape optimization problem 
\begin{equation}\label{lb2opt}
\min\Big\{\lambda_2(\Omega)+\alpha|\Omega|\ :\ \Omega\ \text{open},\ \Omega\subset\Dr\Big\},
\end{equation}
where $\Dr\subset\R^2$ is a bounded open set and $\alpha>0$. By the Buttazzo-Dal Maso Theorem this problem admits a solution in the class of quasi-open sets. The question of regularity of the solutions is quite involved and no progress was made for almost two decades until in \cite{bucur-mink} it was proved that every solution has finite perimeter and in \cite{buve} it was proved that there is an open solution characterized through a multiphase problem. In the Theorem below we answer the question of the regularity of the reduced boundary $\partial^\ast \Omega$ of the solutions of \eqref{lb2opt}.
\begin{thm}
Let $\Omega$ be a solution of \eqref{lb2opt}. Then the reduced boundary $\Dr\cap\partial^\ast\Omega$ is smooth.
\end{thm}
\begin{proof}
We first notice that it was proved in \cite{buve} that for every solution $\Omega$ of the problem \eqref{lb2opt} there are disjoint open sets $\omega_1,\omega_2\subset\Omega$ of the same measure as $\Omega$, i.e. $|\Omega\setminus(\omega_1\cup\omega_2)|=\emptyset$ such that the set $\omega_1\cup\omega_2$ is still a solution of \eqref{lb2opt} and such that the couple $(\omega_1,\omega_2)$ is a solution to the multiphase problem 
\begin{equation}\label{lb2opt2}
\min\Big\{\max\{\lambda_1(\omega_1),\lambda_1(\omega_2)\}+\alpha|\omega_1|+\alpha|\omega_2|\ :\ \omega_1,\omega_2\ \text{open},\ \omega_{1,2}\subset\Dr,\ \omega_1\cap\omega_2=\emptyset\Big\}.
\end{equation}
We notice that necessarily $\omega_1$ and $\omega_2$ are both connected and $\lambda_1(\omega_1)=\lambda_1(\omega_2)$, otherwise it would be possible to construct a better competitor for \eqref{lb2opt2}. Thus, by confronting the couple $\omega_1,\omega_2$ with a couple $\tilde\omega_1,\omega_2$ where $\widetilde\omega_1\subset\omega_1$ we get that $\omega_1$ is a shape subsolution for the functional $\lambda_1+\alpha|\cdot|$ and analogously $\omega_2$ is a shape subsolution for the same functional. In particular, all the conclusions of Theorem \ref{subth} are valid. Let now $x_0\in\partial^\ast\Omega$. Using the non-degeneracy of the gradient of the first eigenfunctions $u_1\in H^1_0(\omega_1)$ and  $u_2\in H^1_0(\omega_2)$ in $x_0$ and the decay monotonicity formula Lemma \ref{lem2ph05}, and reasoning as in the proof of Theorem \ref{mainr2} {\it (iv)} we get that there is a ball $B_r(x_0)$ that does not intersect one of the sets $\omega_1$ and $\omega_2$. Without loss of generality $B_r(x_0)\cap \omega_2=\emptyset$. Now by the regularity result of Brian\c con and Lamboley \cite{brla} and the fact that $\partial^\ast\omega_1=\partial^\ast\Omega$ in $B_r(x_0)$ we get that $\partial^\ast\Omega$ is regular in a neighbourhood of $x_0$.  
\end{proof}
\begin{rem}
We notice that an estimate on the Hausdorff dimension of the set $\partial\Omega\setminus\partial^\ast\Omega$ is not available at the moment. 
\end{rem}
\subsection{Multiphase shape optimization problems on smooth manifolds}
We notice that all the arguments that we use are local and Theorem \ref{mainr2} can easily be extended to the case where the box $(\Dr,g)$ is a riemannian manifold with or without boundary. In fact the existence of an optimal partition follows by the analogous of the Buttazzo - Dal Maso Theorem proved in \cite{buve}. The Laplace-Beltrami operator $\Delta_g$ in local coordinates satisfies $\eps \Delta\le \Delta_g\le \eps^{-1}\Delta$ as an operator, where $\eps>0$ depends on $\Dr$ and $g$, and analogously the gradient satisfies $\eps |\nabla u|\le |\nabla_g u|\le \eps^{-1}|\nabla u|$, for any function $u\in H^1(\Dr)$ expressed in local coordinates. Thus, the two and three-phase monotonicity formulas are still valid as well as the non-degeneracy of the gradient, the lack of triple points inside $\Dr$ and the lack of double points on the boundary of $\Dr$. We regroup the results that are still valid in the following Theorem.  
 \begin{thm}\label{mainsurf}
Suppose that $\Dr$ is a compact riemannian surface. Let $0<a\le A$ be two positive real numbers and $W_i:\Dr\to[a,A]$, $i=1,\dots,h$ be given $C^2$ functions. Then there are disjoint open sets $\Omega_1,\dots,\Omega_h\subset\Dr$  solving the multiphase optimization problem \eqref{optsum} in $\Dr$. Moreover, any solution to \eqref{optsum} satisfies the conditions {\it (i)}, {\it (ii)} and {\it (iii)} of Theorem \ref{mainr2}.
\end{thm}

\section{Numerical eigenvalue computation on a fixed grid}
\label{eigcomp}

There are multiple ways of computing numerically the low Dirichlet-Laplace eigenvalues of a shape $\Omega$, most of them requiring a good description of the boundary (for example finite elements, or fundamental solutions). In our case it is necessary to compute the first eigenvalue of a number of shapes, for which it is difficult to keep track of their boundaries. Thus, having a method which allows us to work on a fixed domain $D$ containing the shape, greatly simplifies the treatment of the problem. Methods of this kind were used in \cite{buboou},\cite{disk-partitions} in the study of spectral minimal partitions. In our study we use the method presented in \cite{buboou}. We did not found any other works in the literature which study the numerical error associated to this method. In this section we present the discretization algorithm, as well as the errors obtained for a few simple shapes.

This eigenvalue computation method is inspired from penalized problems of the form
\begin{equation}
-\Delta u + \mu u = \lambda_k(\mu) u ,\qquad u \in H^1(D)\cap L^2(D,\mu),
\label{penalized}
\end{equation}
where $D$ is a bounded open set in $\Bbb{R}^2$, and $\mu$ is a measure such that $\mu(A)= 0$ whenever $A$ has capacity zero. The case where $\lambda_k$ corresponds to a Dirichlet Laplace eigenvalue of a set $\Omega \subset D$ is included in the formulation \eqref{penalized}. Indeed, if $\infty_{\Omega^c}$ is defined as follows:
\[ \infty_{\Omega^c}(A) = \begin{cases} 0 & \text{ if } \cp(A\cap \Omega) = 0 \\
\infty & \text{ otherwise} \end{cases},\]
then $\lambda_k(\infty_{\Omega^c})=\lambda_k(\Omega)$. We have denoted $\cp(A)$ the capacity of the set $A$. For further details about the penalized formulation \eqref{penalized}, we refer to \cite[Chapter 6]{bucurbuttazzo}.

This formulation suggests the following numerical method: we choose $\mu = (1-\ind_\Omega)C dx$, where $\ind_\Omega$ is the characteristic function of $\Omega$, and $C$ is large. In \cite{buboou} it is proved that as $C \to \infty$ we have $\lambda_k(C(1-\ind_\Omega)dx) \to \lambda_k(\Omega)$. In the following we propose to study the behavior of this eigenvalue computation method with respect to the discretization parameter $N$ and with respect to the choice of $C$. We compare these values with the ones provided by the MpsPack software \cite{mpspack}, which is quite precise.

We consider the domain $D = [-1.5,1.5]^2$ and on it we take a $N \times N$ uniform grid. We discretize a function $u : D \to \Bbb{R}$ by considering its values on this regular grid. For sets $\Omega \subset D$ we consider the approximation of problem \eqref{penalized} defined as
\begin{equation}
-\Delta u + C(1-\ind_\Omega) u = \lambda_k(C\ind_{\Omega^c}dx) u. 
\label{penalizedC}
\end{equation}
This leads us to the discretized matrix problem
\[ (A+C \text{diag}(1-\ind_\Omega))u = \lambda u,\]
where $A$ is the finite difference discretization of the Laplace operator.

We present below the relative error, compared to MpsPack, in function of the measure parameter $C$ and the discretization parameter $N$. In tables \ref{test-disk},\ref{test-square},\ref{weird-shape} and \ref{weird-shape2} we present the maximal relative error $|\lambda_k^{\text{approx}}-
\lambda_k|/\lambda_k
$ (with $1 \leq k \leq 10$) for the unit disk, for the square of side length $2$ and for the shapes presented in Figure \ref{testing-penalized}. Here $\lambda_k$ stands for the analytical value (when available) or the value computed with MpsPack.

In our experiments we observed that for a fixed discretization parameter $N$, the relative error stabilizes itself when $C$ is large enough. This numerical effect seems to be due to the fact that $\Omega$ is approximated using a rectangular grid, so at a given $N$, for large $C$ we only compute the eigenvalue of this discrete approximation of $\Omega$.

\begin{table}
\begin{tabular}{|c|c|c|c|c|c|c|c|}
\hline 
$N$ & $C = 10^3$ & $C = 10^4$ & $C = 10^5$ & $C = 10^6$ & $C = 10^7$ & $C = 10^8$ & $C = 10^9$\\
 \hline
 $ 100$ &$5.2\cdot 10^{-2}$ & $1.8\cdot 10^{-2}$ &$1.2\cdot 10^{-2}$ & $1.1\cdot 10^{-2}$ & $1.1\cdot 10^{-2}$ & $1.1\cdot 10^{-2}$ & $1.1\cdot 10^{-2}$\\
 \hline
 $ 200$ & $5.1\cdot 10^{-2}$ & $1.1\cdot 10^{-2}$ & $1.5\cdot 10^{-3}$ & $5.5 \cdot 10^{-4}$ & $7 \cdot 10^{-4}$ & $7.2 \cdot 10^{-4}$ & $7.2 \cdot 10^{-4}$ \\
 \hline
 $ 300$ & $5.6\cdot 10^{-2}$ & $1.6\cdot 10^{-2}$& $4.7\cdot 10^{-3}$ & $2.8 \cdot 10^{-3}$ & $2.6\cdot 10^{-3}$ &
$2.6\cdot 10^{-3}$ & $2.6\cdot 10^{-3}$ \\
 \hline
 $ 400$ &$5.7\cdot 10^{-2}$ & $1.6\cdot 10^{-2}$  & $3.9\cdot 10^{-3}$ & $1.6\cdot 10^{-3}$  & $1.3\cdot 10^{-3}$ & $1.3\cdot 10^{-3}$  & $1.3\cdot 10^{-3}$\\
 \hline
 $ 500$ & $5.7\cdot 10^{-2}$ & $1.6\cdot 10^{-2}$  & $3.8\cdot 10^{-3}$ & $1.1\cdot 10^{-3}$ & $7.9\cdot 10^{-4}$ & $7.6\cdot 10^{-4}$ & $7.5 \cdot 10^{-4}$ \\
 \hline
\end{tabular}
\caption{Relative errors for the unit disk}
\label{test-disk}
\end{table}

\begin{table}
\begin{tabular}{|c|c|c|c|c|c|c|c|}
\hline 
$N$ & $C = 10^3$ & $C = 10^4$ & $C = 10^5$ & $C = 10^6$ & $C = 10^7$ & $C = 10^8$ & $C = 10^9$\\
 \hline
 $100$ &$5.1\cdot 10^{-2}$ & $1.9\cdot 10^{-2}$ &$1.3\cdot 10^{-2}$ & $1.3\cdot 10^{-2}$ & $1.3\cdot 10^{-2}$ & $1.3\cdot 10^{-2}$ & $1.3\cdot 10^{-2}$\\
 \hline
 $200$ & $4.4\cdot 10^{-2}$ & $5.7\cdot 10^{-3}$ & $3.7\cdot 10^{-3}$ & $4.8 \cdot 10^{-3}$ & $5 \cdot 10^{-3}$ & $5 \cdot 10^{-3}$ & $5 \cdot 10^{-3}$ \\
 \hline
 $ 300$ & $6.2\cdot 10^{-2}$ & $2.2\cdot 10^{-2}$& $1.2\cdot 10^{-2}$ & $  10^{-2}$ & $ 10^{-2}$ &
$ 10^{-2}$ & $10^{-2}$ \\
 \hline
 $400$ &$5.7\cdot 10^{-2}$ & $1.6\cdot 10^{-2}$  & $5\cdot 10^{-3}$ & $2.9\cdot 10^{-3}$  & $2.7\cdot 10^{-3}$ & $2.7\cdot 10^{-3}$  & $2.7\cdot 10^{-3}$\\
 \hline
 $ 500$ & $5.4\cdot 10^{-2}$ & $1.3\cdot 10^{-2}$  & $8.3\cdot 10^{-4}$ & $1.7\cdot 10^{-3}$ & $2\cdot 10^{-3}$ & $2\cdot 10^{-3}$ & $2\cdot 10^{-3}$ \\
 \hline
\end{tabular}
\caption{Relative errors for the square of side length $2$}
\label{test-square}
\end{table}

\begin{table}
\begin{tabular}{|c|c|c|c|c|c|c|c|}
\hline 
$N$ & $C = 10^3$ & $C = 10^4$ & $C = 10^5$ & $C = 10^6$ & $C = 10^7$ & $C = 10^8$ & $C = 10^9$\\
 \hline
 $ 100$ &$6\cdot 10^{-2}$ & $1.8\cdot 10^{-2}$ &$1\cdot 10^{-2}$ & $9.5\cdot 10^{-3}$ & $9.4\cdot 10^{-3}$ & $9.4\cdot 10^{-3}$ & $9.4\cdot 10^{-3}$\\
 \hline
 $ 200$ & $6.5\cdot 10^{-2}$ & $1.8\cdot 10^{-2}$ & $6.1\cdot 10^{-3}$ & $4.4 \cdot 10^{-3}$ & $4.2 \cdot 10^{-3}$ & $4.2 \cdot 10^{-3}$ & $4.2 \cdot 10^{-3}$ \\
 \hline
 $ 300$ & $6.7\cdot 10^{-2}$ & $1.9\cdot 10^{-2}$& $4.9\cdot 10^{-3}$ & $  2.5\cdot 10^{-3}$ & $ 2.2\cdot 10^{-3}$ &
$2.2\cdot 10^{-3}$ & $2.2\cdot 10^{-3}$ \\
 \hline
 $ 400$ &$6.8\cdot 10^{-2}$ & $1.9\cdot 10^{-2}$  & $4.7\cdot 10^{-3}$ & $1.8\cdot 10^{-3}$  & $1.4\cdot 10^{-3}$ & $1.4\cdot 10^{-3}$  & $1.4\cdot 10^{-3}$\\
 \hline
 $ 500$ & $6.9\cdot 10^{-2}$ & $2\cdot 10^{-2}$  & $5.3\cdot 10^{-3}$ & $1.9\cdot 10^{-3}$ & $1.4\cdot 10^{-3}$ & $1.4\cdot 10^{-3}$ & $1.4\cdot 10^{-3}$ \\
 \hline
\end{tabular}
\caption{Relative errors for the shape presented in Figure \ref{testing-penalized} (left)}
\label{weird-shape}
\end{table}

\begin{table}
\begin{tabular}{|c|c|c|c|c|c|c|c|}
\hline 
$N$ & $C = 10^3$ & $C = 10^4$ & $C = 10^5$ & $C = 10^6$ & $C = 10^7$ & $C = 10^8$ & $C = 10^9$\\
 \hline
 $ 100$ &$6.9\cdot 10^{-2}$ & $2.2\cdot 10^{-2}$ &$1.4\cdot 10^{-2}$ & $1.3\cdot 10^{-2}$ & $1.3\cdot 10^{-2}$ & $1.3\cdot 10^{-2}$ & $1.3\cdot 10^{-2}$\\
 \hline
 $ 200$ & $7.2\cdot 10^{-2}$ & $2\cdot 10^{-2}$ & $6.8\cdot 10^{-3}$ & $4.8 \cdot 10^{-3}$ & $4.6 \cdot 10^{-3}$ & $4.6 \cdot 10^{-3}$ & $4.6 \cdot 10^{-3}$ \\
 \hline
 $ 300$ & $7.4\cdot 10^{-2}$ & $2.1\cdot 10^{-2}$& $5.9\cdot 10^{-3}$ & $  3.3\cdot 10^{-3}$ & $ 3\cdot 10^{-3}$ &
$2.9\cdot 10^{-3}$ & $2.9\cdot 10^{-3}$ \\
 \hline
 $ 400$ &$7.6\cdot 10^{-2}$ & $2.2\cdot 10^{-2}$  & $6.1\cdot 10^{-3}$ & $2.8\cdot 10^{-3}$  & $2.4\cdot 10^{-3}$ & $2.4\cdot 10^{-3}$  & $2.4\cdot 10^{-3}$\\
 \hline
 $ 500$ & $7.6\cdot 10^{-2}$ & $2.3\cdot 10^{-2}$  & $5.6\cdot 10^{-3}$ & $1.8\cdot 10^{-3}$ & $1.3\cdot 10^{-3}$ & $1.3\cdot 10^{-3}$ & $1.3\cdot 10^{-3}$ \\
 \hline
\end{tabular}
\caption{Relative errors for the shape presented in Figure \ref{testing-penalized} (right)}
\label{weird-shape2}
\end{table}

\begin{figure}
\begin{center}
\includegraphics[width = 0.4 \textwidth]{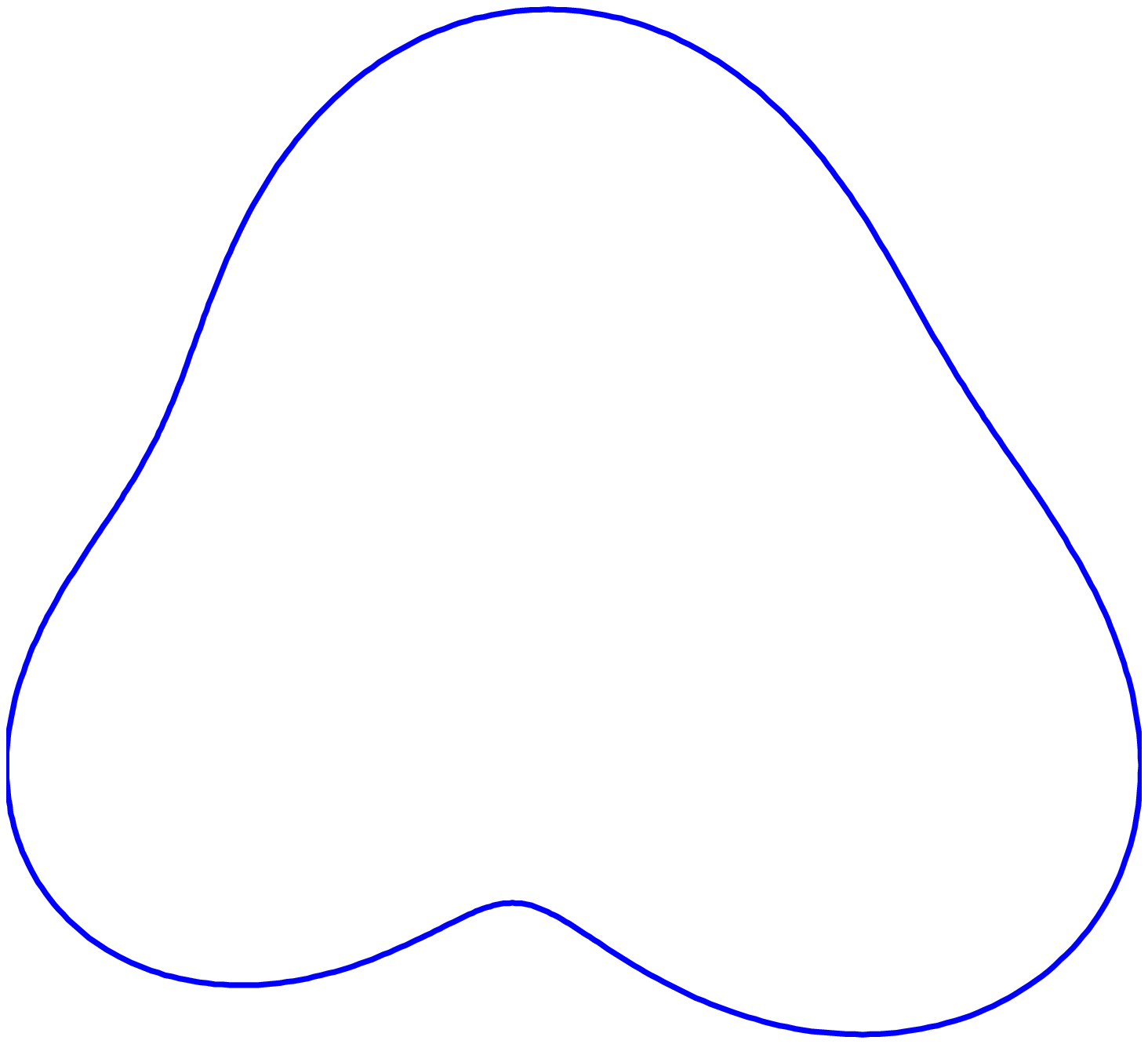} \hspace{0.05\textwidth}
\includegraphics[width = 0.4 \textwidth]{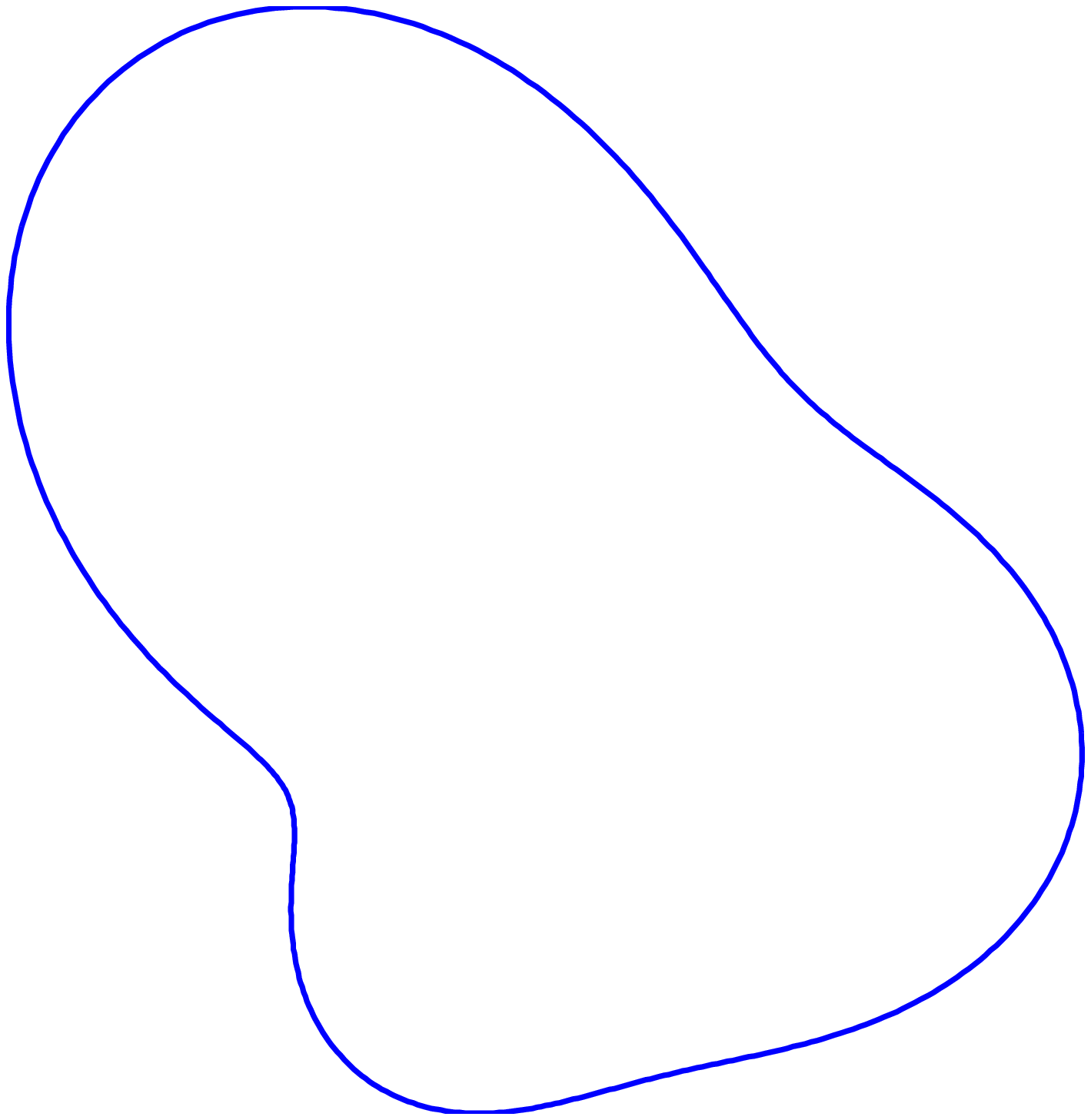}
\end{center}
\caption{Shapes for which we test the method in Table \ref{weird-shape} (left) and Table \ref{weird-shape2} (right)}

\label{testing-penalized}
\end{figure}

\section{Proof of Theorem \ref{error-thm}}
\label{error-estimate}
In this section, we give a theoretical estimate of the relative error obtained when working with the penalized method. We study the difference between the eigenvalue $\lambda_k(C\ind_{\Omega^c}dx)$, given by \eqref{penalizedC}, and $\lambda_k(\Omega)$. We fix $\Omega \subset D$ to be an open set with boundary of class $C^2$. In the following, we denote $\mu_C = C\ind_{\Omega^c}dx$.

 We consider the functions $w,w_C$ defined as follows
\[ -\Delta w = 1 \text{ in }\Omega,\ w \in H_0^1(\Omega),\]
\[ -\Delta w_C+C\ind_{\Omega^c} w_C = 1 \text{ in }D, w_C \in H_0^1(D).\]
Note that the standard maximum principle implies that $w_C \geq w$ on $D$. Using the terminology defined in \cite{potentials-bbv} we note that $\mu_C dx \prec \infty_{\Omega^c} $ so cf. \cite[Lemma 4.3]{potentials-bbv} and \cite[Lemma 4.1]{bucur-mink}  the following estimate holds
\[ \|\mathcal{R}_{\mu_C}-
\mathcal{R}_{\infty_{\Omega^c}}\|_{\mathcal{L}(L^2)} \leq C_{N,\Omega}\|w_C-w\|_{L^1}.\]
In general, we denote $R_\mu$ the resolvent operator associated to the problem
\[ -\Delta u + u\mu = f,\ u \in H_0^1(D)\cap L^2(D,\mu).\]
Using \cite[Corollary 6.1.8]{bucurbuttazzo} we obtain the estimate
\[ \left| \frac{1}{\lambda_k(\mu_C)}-\frac{1}{\lambda_k(\Omega)}\right| \leq \|\mathcal{R}_{\mu_C}-
\mathcal{R}_{\infty_{\Omega^c}}\|_{\mathcal{L}(L^2)} \leq C_{N,\Omega} \|w_C-w\|_{L^1}.\]
Thus, we have
\[ \frac{|\lambda_k(\Omega)-\lambda_k(\mu_C)|}{\lambda_k(\Omega)} \leq \lambda_k(\mu_C)C_{N,\Omega} \int_D w_C-w. \]
The monotonicity property stated in \cite[Proposition 6.1.5]{bucurbuttazzo} shows that $\lambda_k(\mu_C) \leq \lambda_k(\Omega)$. In order to finish the proof, it suffices to give an upper bound for $\ds \int_D w_C-w$. 

We clearly have 
\[ \int_D |\nabla w_C|^2+C \int_{\Omega^c} w_C^2 = \int_D w_C.\]
When $C\to \infty$ we have $w_c \wconv w$ in $H_0^1(D)$, and as a consequence $\ds \lim_{C \to \infty} C\int_{\Omega^c} w_C^2 = 0$. This proves that for $C$ large enough there exists a constant $M$ such that
\[ \left( \int_{\Omega^c} w_C\right)^2 \leq |\Omega^c|\int_{\Omega^c} w_C^2   \leq \frac{M}{C}.\]
Thus 
\[ \int_{\Omega^c}w_C-w \leq \frac{M}{C^{1/2}}.\]
For the estimate of $\ds\int_\Omega w_C-w$ we use the fact that $w_C -w$ is harmonic in $\Omega$, so
\[ \int_\Omega w_C-w \leq \sup_{\partial \Omega} w_C |\Omega|.\]
It remains to estimate $\sup_{\partial \Omega} w_C$.

Assume that $B_{x_0,r_0} \subset \Omega^c$. Then 
\[ -\Delta w_C \leq 1 \text{ in }D,\]
so $\ds w_C +\frac{|x-x_0|^2}{2N}$ is subharmonic in $D$. This implies
\[ w_C(x_0) \leq \frac{1}{\omega_N r_0^N} \int_{B_{x_0,r_0}} (w_c +\frac{|x-x_0|^2}{2N})\leq \frac{r_0^2}{2N}+\frac{1}{\omega_N^{1/2} r_0^{N/2}}\left( \int_{B_{x_0,r_0}} w_C^2\right)^{1/2},\]
where we used the fact that $|x-x_0|\leq r$ and we applied the Cauchy Schwarz inequality.
Thus, for $C$ large enough we have
\[ w_C(x_0) \leq \frac{r_0^2}{2N} +\frac{M}{\omega_N^{1/2} r_0^{N/2}C^{1/2}}.\]
Next, we choose $r_0$ of the form $C^{-\alpha}$, which gives us
\[ w_C(x_0) \leq \frac{1}{2N C^{2\alpha}} +\frac{M}{\omega_N^{1/2} C^{(1-N\alpha)/2}}.\]
We choose $\alpha = 1/(N+4)$, which gives the same exponent for $C$ in the two terms of the above sum. Thus
\[ w_C(x_0) \leq \left(\frac{1}{2N}+\frac{M}{\omega_N^{1/2}}\right) C^{-2/(N+4)}.\]
Clearly, as $C \to \infty$, $x_0$ can be chosen closer and closer to $\partial \Omega$. The fact that $\Omega$ is of class $C^2$ implies that $\Omega$ satisfies an exterior ball condition $B_\rho$. If $d(x_0,\partial \Omega)<\rho$ then we can apply the previous estimate.

To go from $x_0$ to the boundary $\partial \Omega$ we note that the Minkowski sum $\Omega+B_{C^{-\alpha}}$ satisfies an interior and exterior ball condition, if $C$ is large enough. For simplicity, we denote $\Omega' = \Omega+ B_{C^{-\alpha}}$ in the sequel. Thus $\Omega'$ is of class $C^{1,\alpha}$ and $\nabla w_{\Omega'}$ is well defined on $\partial \Omega'$.
 Furthermore, consider $B_{\rho'}$ an exterior ball tangent to $\Omega'$ and another concentric ball $B_R$ such that $B_R$ contains $\Omega'$. The annulus $A$ determined by $B_{\rho'},B_R$ contains $\Omega'$, and thus $w_{\Omega'} \leq w_A$ in $\Omega'$ and $|\nabla w_{\Omega'}|\leq |\nabla w_A|$ on $\partial \Omega'$ . It is well known that $w_A$ is Lipschitz, with a Lipschitz constant depending on $\rho'$ and the diameter of $\Omega'$. Thus, on $\partial\Omega'$ we have that $|\nabla w_{\Omega'}|$ is bounded, and since $|\nabla w_{\Omega'}|$ is maximal on the boundary, it follows that $w_{\Omega'}$ is Lipschitz.

The function $w_C-w_{\Omega'}$ is subharmonic on $\partial \Omega'$. As a consequence, we have 
\[ w_{\Omega + B_{C^{-\alpha}}} \geq \left( w_C - \left(\frac{1}{2N}+\frac{M}{\omega_N^{1/2}}\right) C^{-2/(N+4)}\right)^+,\]
which together with the Lipschitz continuity of $w_{\Omega+ B_{C^{-\alpha}}}$ gives us that
\[ w_C |_{\partial \Omega} \leq \left(\frac{1}{2N}+\frac{M}{\omega_N^{1/2}}\right) C^{-2/(N+4)}+M_2  C^{-\alpha},\]
where $M_2$ is the constant in the Lipschitz continuity result.

Thus \[w_C|_{\partial \Omega} \leq \left(\frac{1}{2N}+\frac{M}{\omega_N^{1/2}}\right) C^{-2/(N+4)}+M_2 \ell^{1/2}C^{-1/(N+4)}.\]
Consequently, there exists a constant $M_3$, depending on $\varepsilon,N,D$, such that
\[ \int_D w_c-w \leq M_3C^{-1/(N+4)}.\]

In conclusion, for $C$ large enough, there exists a constant $K$ such that
\[ \frac{|\lambda_k(\Omega)-\lambda_k(\mu_C)|}{\lambda_k(\Omega)} \leq K C^{-1/(N+4)}.\]
\hfill $\square$

\begin{rem}
Using techniques similar to \cite[Lemma 3.4.11]{henrot-pierre} we are able to prove that there is an upper bound of the form $KC^{-\delta}$ (with $K,\delta>0$) for the relative error even in the more general case when $\Omega$ satisfies a $\varepsilon$-cone condition (equivalently, a uniform Lipschitz condition). The drawback is that we do not have an explicit formula for $\delta$, like in the case presented above.
\end{rem}

We remark that in the case $N=2$, studied numerically in the previous section, the relative error is bounded theoretically by a term of order $C^{-1/6}$. If we look at the numerical errors, we see that from $C=10^3$ to $C=10^9$ the errors roughly decrease by one order of magnitude. This is in good correspondence with the theoretical result which predicts a decrease of the relative error by approximately one order of magnitude when $C$ is multiplied by $10^6$. This correspondence shows that this theoretical error bound is close to being sharp in two dimensions. 

\section{Numerical setting and optimization algorithm}
\label{optim-algo}

In order to compute numerically the shape and the position of the optimal sets, we use the procedure described in Section \ref{eigcomp}. This technique has been introduced in \cite{buboou} for the study of the case $\alpha=0$. We recall that the problem we study has the form 
\begin{equation}\label{moplbk}
\min\Big\{\sum_{i=1}^h\lambda_k(\Omega_i)+\alpha|\Omega_i|:\ \Omega_i\subset D\ \hbox{quasi-open},\ \Omega_i\cap\Omega_j=\emptyset\Big\}.
\end{equation}
where with $\lambda_k(\Omega)$ we denote the $k$-th eigenvalue of the Dirichlet Laplacian on $\Omega\subset D$.

For a given measurable function $\vf:\Omega\in[0,1]$ and constant $C>0$, we consider the spectrum of the operator $-\Delta+C(1-\vf)$ on $ D $, consisting on the eigenvalues with variational characterization 
$$\lambda_k(\vf, C):=\min_{S_k\subset H^1_0(\Omega)} \max_{u\in S_k}\frac{\int_{\Omega}|\nabla u|^2+C(1-\vf)u^2\,dx}{\int_\Omega u^2\,dx},$$
where the minimum is over all $k$-dimensional subspaces $S_k$ of $H^1_0( D )$. The corresponding $k$-th eigenfunction satisfies the equation
\begin{equation}
-\Delta u_k+C(1-\varphi)u_k=\lambda_k(\varphi,C)u_k,\qquad u_k\in H_0^1(D),\quad \int_{ D }u_k^2\,dx=1.
\label{eigen}
\end{equation}

 By the general existence theorem of Buttazzo and Dal Maso \cite{budm93}, there is a solution $\big(\varphi_1^C,\dots,\varphi_h^C\big)$ of the problem 
\begin{equation}\label{mopphi}
\min\Big\{\sum_{i=1}^h\Big(\lambda_k(\vf_i, C)+\alpha\int_{ D }\vf_i\,dx\Big):\ \vf_i: D \to [0,1]\ \hbox{measurable},\ \sum_{i=1}^h\vf_i\le1\Big\}.
\end{equation}
Moreover, by the approximation result \cite[Theorem 2.4]{buboou}, or the result given by Theorem \ref{error-thm}, we have that, for every $i=1\dots h$,
$$\lim_{C\to+\infty}\lambda_1(\vf_i^C,C)=\lambda_1(\Omega_i)\qquad\hbox{and}\qquad \lim_{C\to+\infty}\vf_i^C=\ind_{\Omega_i},$$
where the second limit is strong in $L^1( D )$ and the $h$-uple $(\Omega_1,\dots,\Omega_h)$ is optimal for \eqref{moplbk}. 

We were not able to prove that for $k \geq 2$ the functions $\varphi_i^C$ converge to characteristic functions as $C\to \infty$. In \cite{buboou} a concavity argument was used to prove the result, and this argument does not extend to the case $k \geq 2$. In the description of the algorithm we keep $k$ general, but the numerical results presented are for $k=1$. Although we don't have a theoretical justification of the convergence in the case $k=2$, the algorithm behaves well and produces the expected results. For $k \geq 3$ we did not manage to obtain conclusive results.

%
%
%
%

Note that for $\alpha>0$ solutions of problem \ref{moplbk} do not consist of partitions of $D$. Therefore the functions $\varphi_l$ satisfy the non-overlapping constraint $\sum_{i=1}^h \varphi_i \leq 1$. This inequality constraint is not easy to treat numerically, so we choose to add an additional phase, representing the empty space.
 Define $\vf_{h+1}:=1-\sum_{i=1}^h\vf_i$, the empty phase associated to the multiphase problem. Thus \eqref{mopphi}  is equivalent to 
\begin{equation}
\min\Big\{\sum_{i=1}^h\lambda_k(\vf_i, C)-\alpha\int_{ D }\vf_{h+1}\,dx:\ \vf_i: D \to [0,1]\ \hbox{measurable},\ \sum_{i=1}^{h+1}\vf_i=1\Big\},
\end{equation}
which is more suitable for numerical implementation.  In this way \eqref{moplbk} is reformulated as an optimal partitioning problem
\[ \min\Big\{ \sum_{i=1}^h \lambda_k(\Omega_i)-\alpha|\Omega_{h+1}| : \ \Omega_i\subset\R^d\ \hbox{quasi-open},\ \Omega_i\cap\Omega_j=\emptyset,\ \hbox{for}\  i,j=1,\dots,h+1 \Big\}. \] 
In this setting the numerical cost computation of the above problem involves the discrete approximation of the measure of $\Omega_{h+1}$ given by
\[ |\Omega_{h+1}| \simeq \frac{1}{N^2}\sum_{i,j=1}^{N^2} \varphi_{i,j}^{h+1}.\]


In order to use an optimization algorithm we approximate the derivative of the eigenvalues $\lambda_k(\vf_l,C)$ as a function of the values of the phases $\varphi_l$ on the grid points. The precise expression of this derivative was given in \cite{buboou} and has the form
\begin{equation} \partial_{i,j} \lambda_k(\varphi_l,C)= -C(U_{i,j}^l)^2,
\label{discderiv}
\end{equation}
where $U^l$ is the $l$-th normalized eigenvector solution of the corresponding discrete equation.
The discrete derivative of the volume is given by
\[ \partial_{i,j}|\Omega_{h+1}|=1/N^2. \]
 
In order to perform the optimization under the constraint $\sum_{l=1}^{h+1}\varphi_l=1$ we use the projection operator on the simplex 
\[ \Bbb{S}^h = \Big\{ X=(X_1,..,X_{h+1}) \in [0,1]^{h+1} : \sum_{l=1}^{h+1} X_l=1 \Big\}, \] defined by
\[ \Big(\Pi_{\Bbb{S}^h}\varphi^l\Big)_{i,j}=\frac{|\varphi^l_{i,j}|}{\sum_{l=1}^{h+1}|\varphi_{i,j}^l|}. \] 
More details about the justification of the choice of this non orthogonal projection operator can be found in \cite{buboou}. We did not manage to improve this projection procedure. We observed that both aspects: the condition that the sum is equal to $1$ and that the functions $\varphi^l$ take values in $[0,1]$, are essential in the optimization process, and this projection operator preserves them both. 

The optimization procedure proposed in \cite{buboou} was based on a steepest descent algorithm with an adapting step length. We improve the descent algorithm by introducing a linesearch procedure in order to determine the step length. A description of the procedure can be found in Algorithm \ref{linesearch}. The number of iterations is significantly reduced, but each iteration needs multiple function evaluations. 

\begin{algorithm}
\caption{Linesearch algorithm}
\label{linesearch}
\begin{algorithmic}[1]
\Require $\gamma_0$, $\omega>1$, $x$, $d$ (descent direction)
\State $\gamma = \gamma_0$
\State Evaluate the cost $c$ corresponding to $x$
\State $c_0 = c$ (variable which keeps previous cost)
\Repeat 
\State $x_t = x+\gamma d$
\State $x_p = \Pi_{\Bbb{S}^h} (x_t)$ (projection on the constraint)
\State Evaluate the cost $c_p$ corresponding to $x_p$
\If {$c_p<c_0$}
\State $\gamma \gets \omega \gamma$
\Else
\State {\bf break}
\EndIf  
\State $c_0 = c_p$
\Until

\Return $\gamma$
\end{algorithmic}
\end{algorithm}

\begin{algorithm}
\caption{General form of the optimization algorithm}
\begin{algorithmic}[1]
\Require $k,\ \alpha,\ h,\ \varepsilon,\ \gamma_0,\ p_{max}$
\State $p=1$
\State Choose random initial densities $(\varphi^l)$ and project them on the constraint
\Repeat
\State Compute $c = F(\varphi^l)$ (the cost functional)
\State Choose descent direction $d = -\nabla F(\varphi^l)$
\State Find step length $\gamma$ using the linesearch algorithm
\State Update $\varphi^l \gets \varphi^l - \gamma d$
\State $\varphi^l \gets \Pi_{\Bbb{S}^h}(\varphi^l)$ (project on the constraint)

\State $p \gets p+1$
\Until $p=p_{max}$ or $\gamma \|\nabla F(\varphi^l)\|_{\ell^\infty}<\varepsilon$
\end{algorithmic}
\end{algorithm}

\begin{figure}
\includegraphics[width=0.8\textwidth]{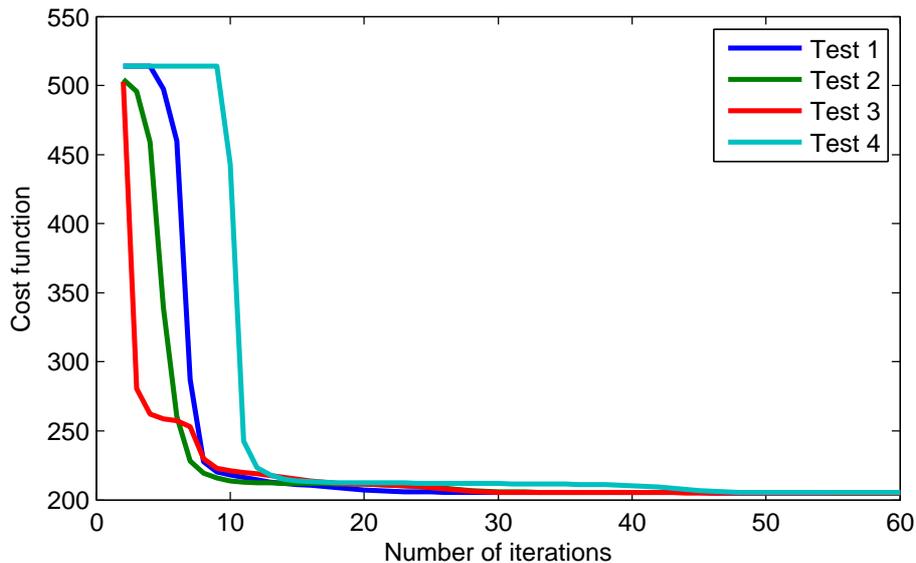}
\caption{Cost evolution in the four cases presented in Figure \ref{regular6}}
\label{costplot}
\end{figure}

In order to test the stability of our modified algorithm, we took a rectangular box which can be paved with regular hexagons, with one edge oriented horizontally, in a periodic setting. One possibility is to choose the edges of the rectangle having a ratio of $\sqrt{3}$, in the case of $6$ cells, or $2/\sqrt{3}$ in the case of $12$ cells. In each case we performed the optimization starting from random densities with sum $1$. We observe that the resulting partitions are equivalent, and the corresponding costs are close. Results can be seen in Figure \ref{regular6} (the case of $6$ cells) and Figure \ref{regular12} (the case of $12$ cells). The cost evolution, in the case of $6$ cells, is plotted in Figure \ref{costplot}.

\begin{figure}
\centering
  \includegraphics[width=0.24\textwidth]{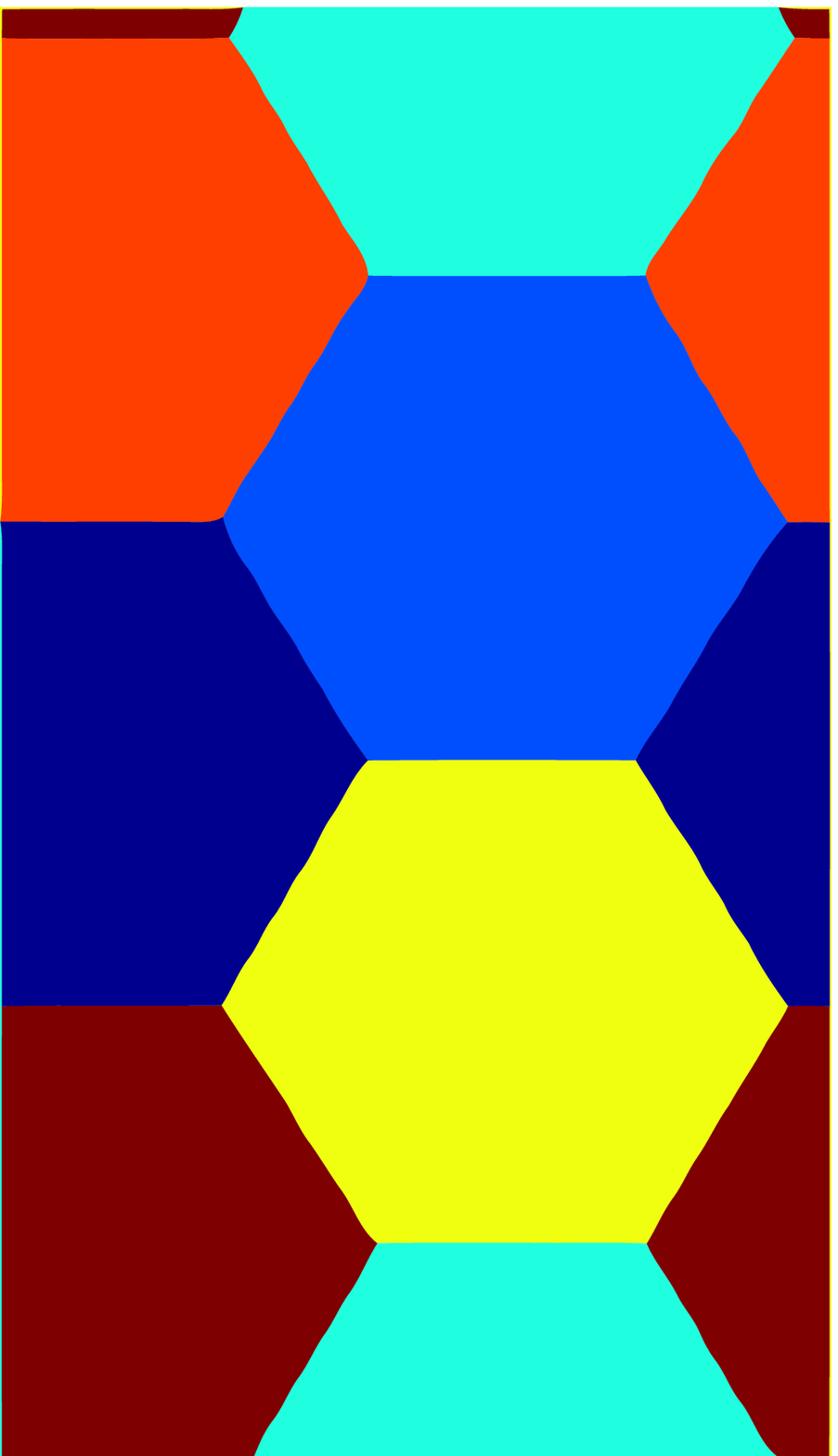}
  ~
  \includegraphics[width=0.24\textwidth]{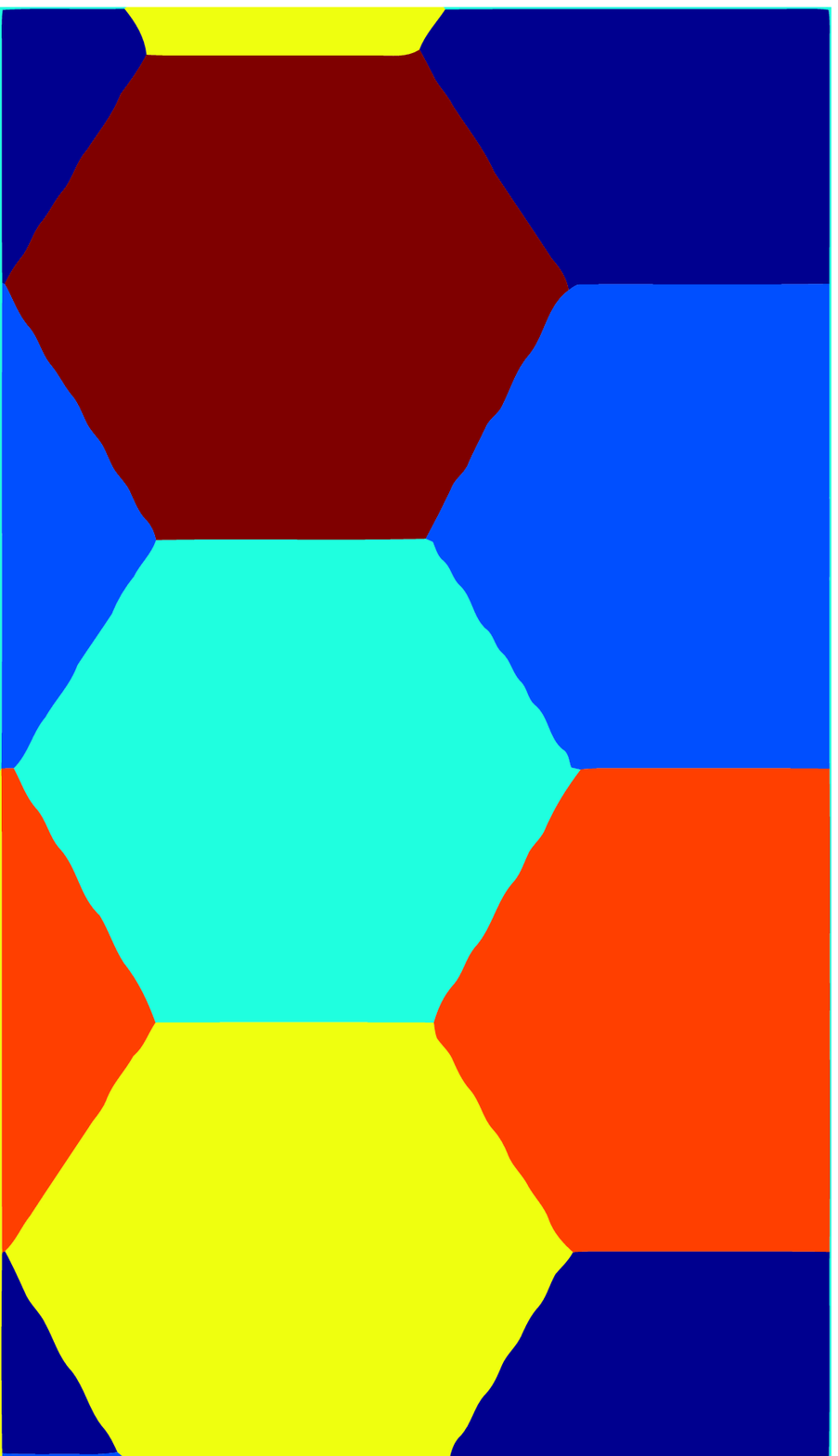} 
  ~
  \includegraphics[width=0.24\textwidth]{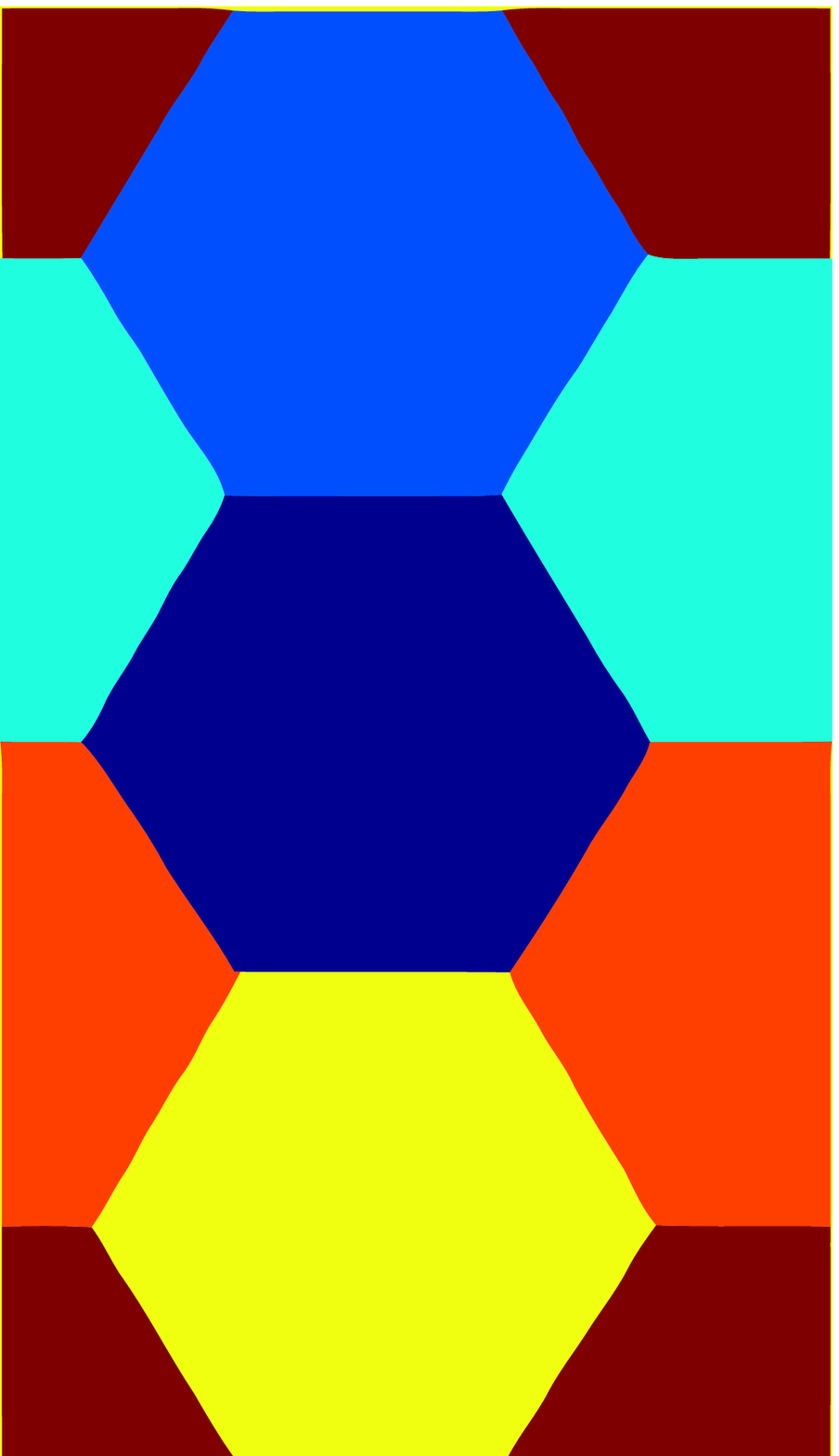}
  ~
  \includegraphics[width=0.24\textwidth]{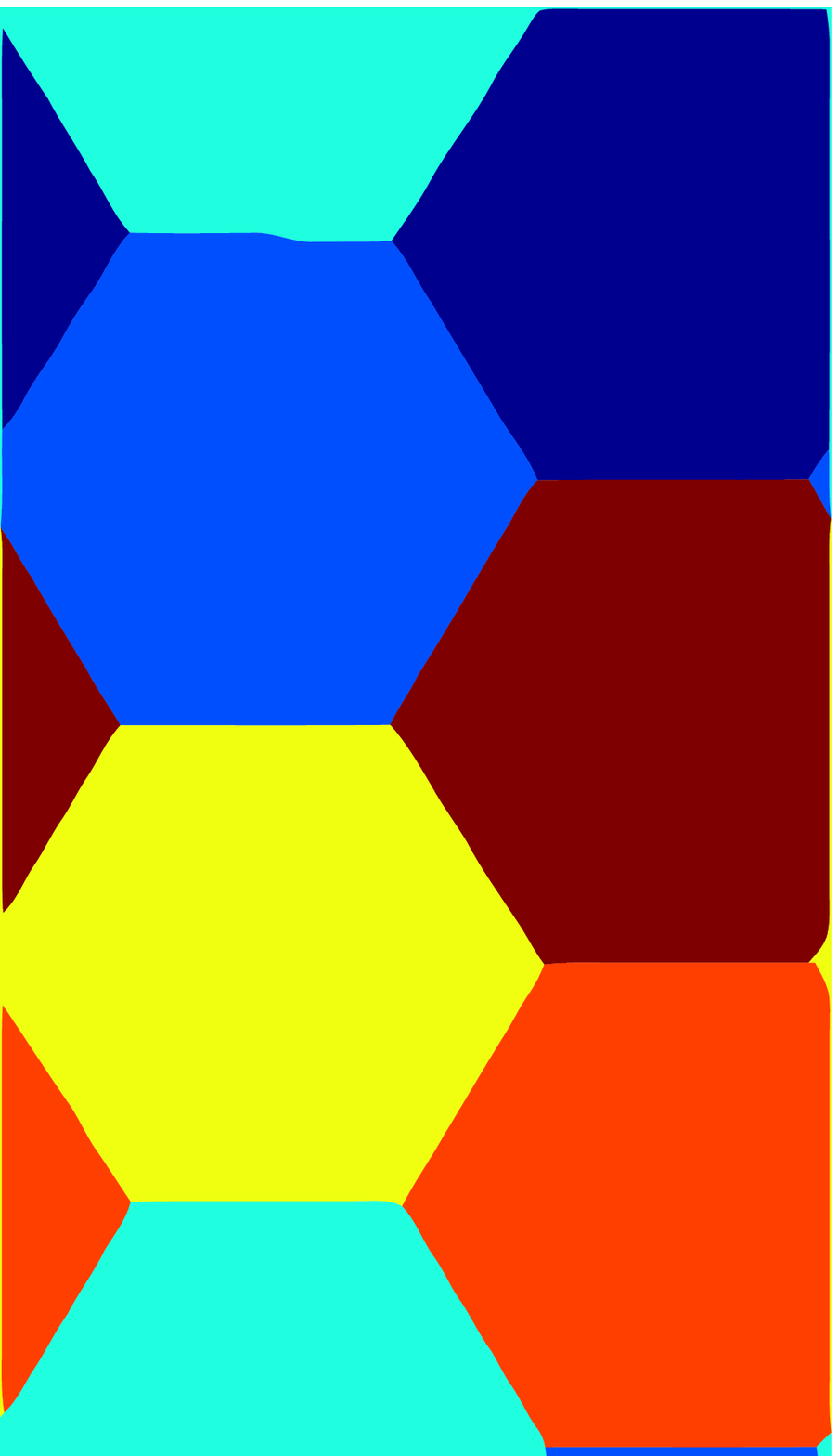}
  \caption{Optimal results - $6$ cells on a periodic domain, starting each time from random densities. Optimal numerical value (left to right): $205.21,205.23,205.22,205.22$}
\label{regular6}
\end{figure}

\begin{figure}
\centering
  \includegraphics[width=0.24\textwidth]{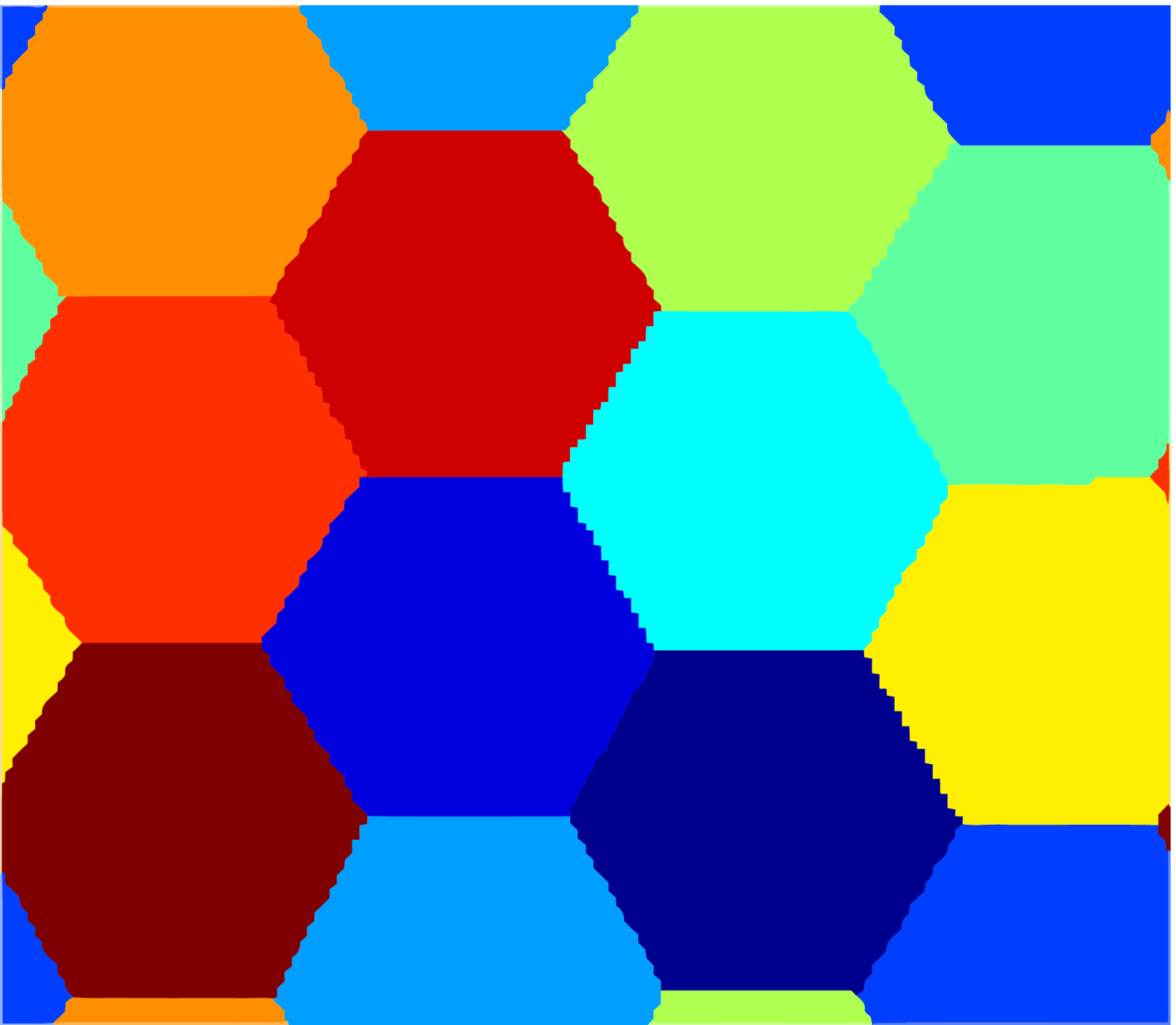}
  ~
  \includegraphics[width=0.24\textwidth]{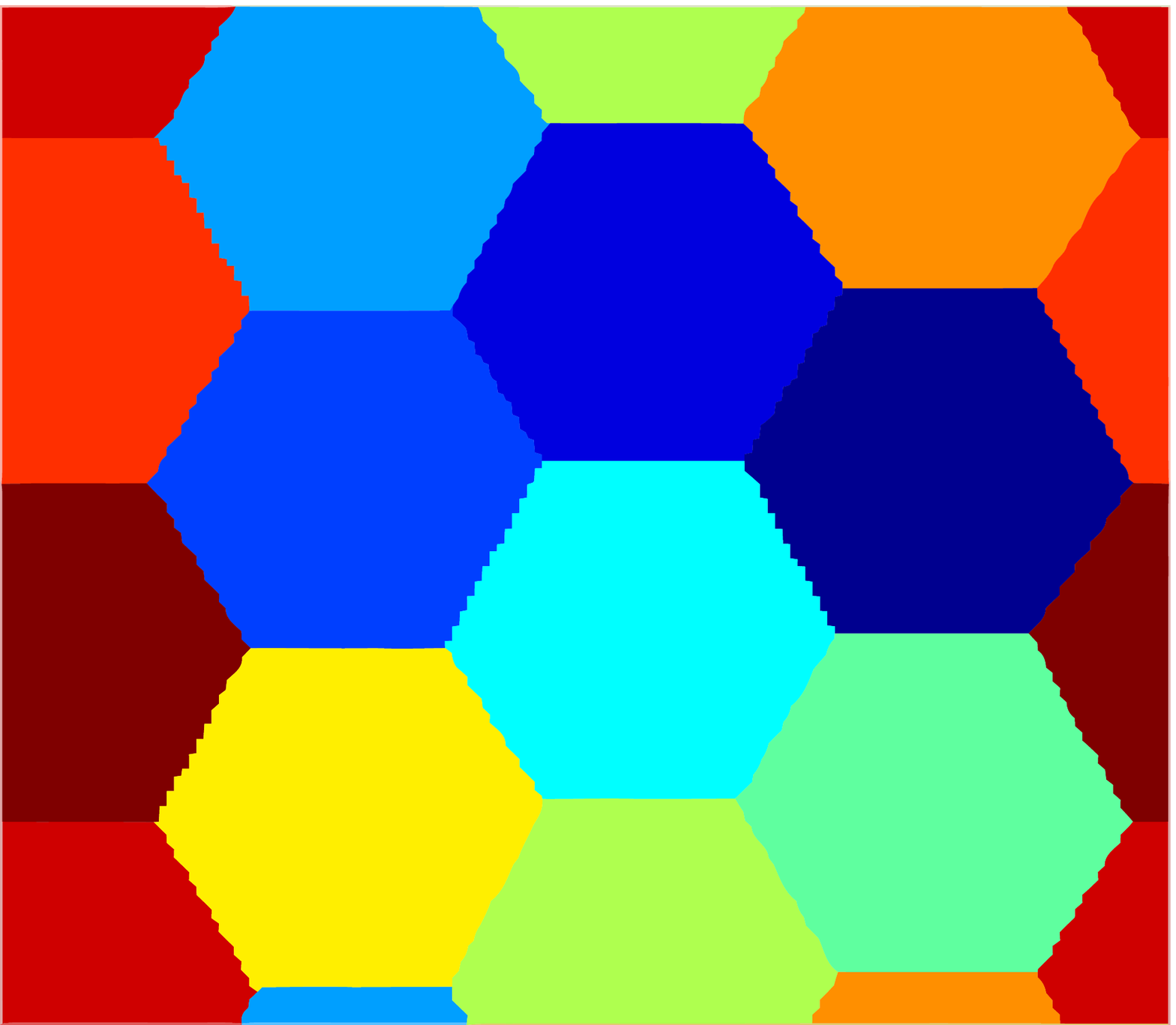} 
  ~
  \includegraphics[width=0.24\textwidth]{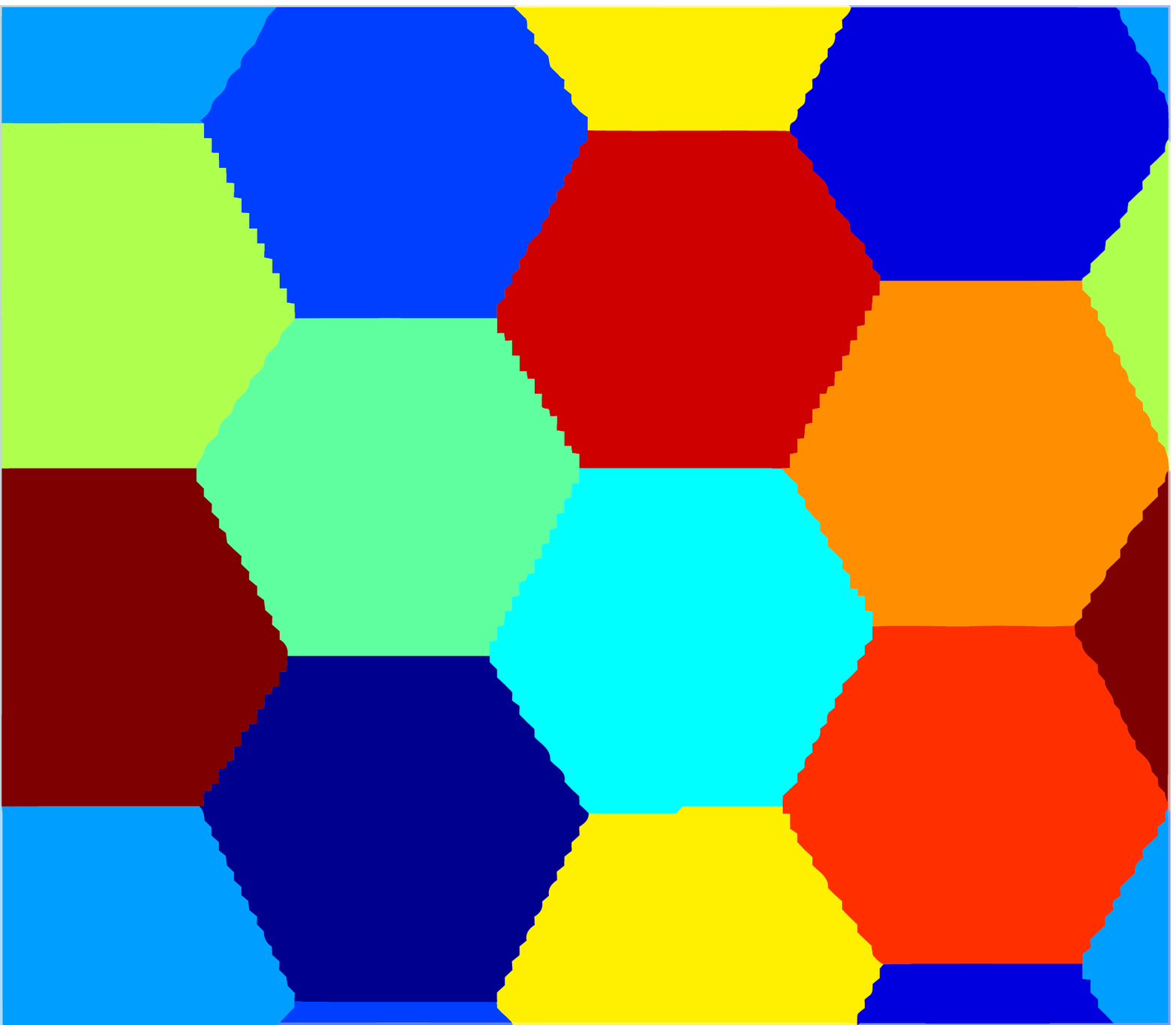}
  ~
  \includegraphics[width=0.24\textwidth]{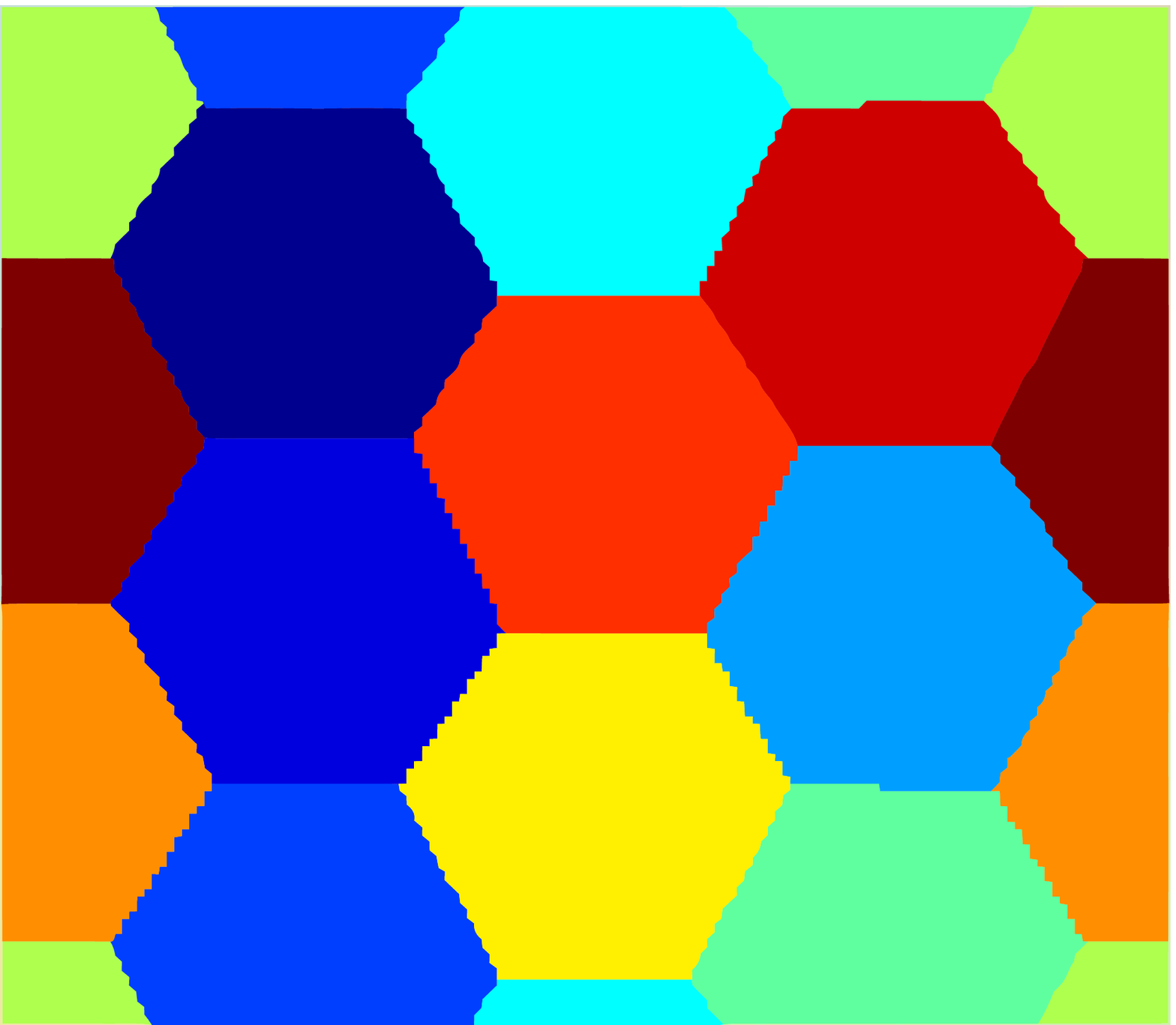}
  \caption{Optimal results - $12$ cells on a periodic domain, starting each time from random densities. Optimal numerical value (left to right): $1512.85,1512.83,1513.12,1513.26$}
\label{regular12}
\end{figure}

In order to be able to study the minimizers of problem \ref{moplbk} in the case where $D$ is not rectangular, we use a finite elements approach. We find a triangulation of $D$ using the software Distmesh \cite{distmesh}, or by specifying a regular triangulation directly (when possible). We compute the associated rigidity and mass matrices $K$ and $M$, respectively. Then, if $\varphi$ is a vector containing the values of the discretization of $\Omega$, we are left to solve the problem
\[ \int_D \nabla u \nabla v+ \int_D C(1-\varphi) uv = \lambda \int_D uv, \]
which has the discrete form
\[ v^TKu + Cv^T \text{diag}(1-\varphi)M u = \lambda v^T M u.\]
Since this is true for each $v$, we are left with the generalized eigenvalue problem
\[ (K+C\text{diag}(1-\varphi)M)u = \lambda M u. \]
In this way, we are able to find numerical minimizers for problem \ref{moplbk} even when $D$ is not rectangular (see Figures \ref{monotonic2},\ref{non-rectangular}). The drawback is that finding generalized eigenvalues is more time consuming than finding eigenvalues. When working on a rectangular domain, using finite differences, we can easily handle discretizations of $200\times 200$ ($40000$ points) on a single machine\footnote{Processor: i7 quad-core 2.2Gh, 6GB of RAM}. For the finite elements case we use triangulations with roughly $5000$ points. 

\section{Discussion of the numerical results}
\label{numerical-obs}

In this section we present some numerical simulations that confirm the theoretical results stated in Theorem \ref{mainr2} and the article \cite{buve}. Furthermore, the numerical simulations in the periodic case, indicate that as $\alpha$ decreases, the cells of the multiphase configuration are monotonically increasing. This was also observed in the case of non-periodic conditions, when the domain has a certain symmetry, which allows a well behaved circle packing. Note that, when the size of the box is well chosen, there exists an optimal parameter $\alpha$, such that the optimal configuration consists of the hexagonal circle packing configuration. If the observed shape monotonicity property is true, then the actual spectral partitioning problem ($\alpha = 0$) can be solved, and the optimal partition is formed of regular hexagons. We note that this result concerning the case $\alpha = 0$ is still an open problem, while results of \cite{buboou} confirm numerically this conjecture (see Figures \ref{monotonic1},\ref{monotonic2}, as well as Figure \ref{nonperiodic}).

\begin{figure}
  \centering
  \includegraphics[width=0.32\textwidth]{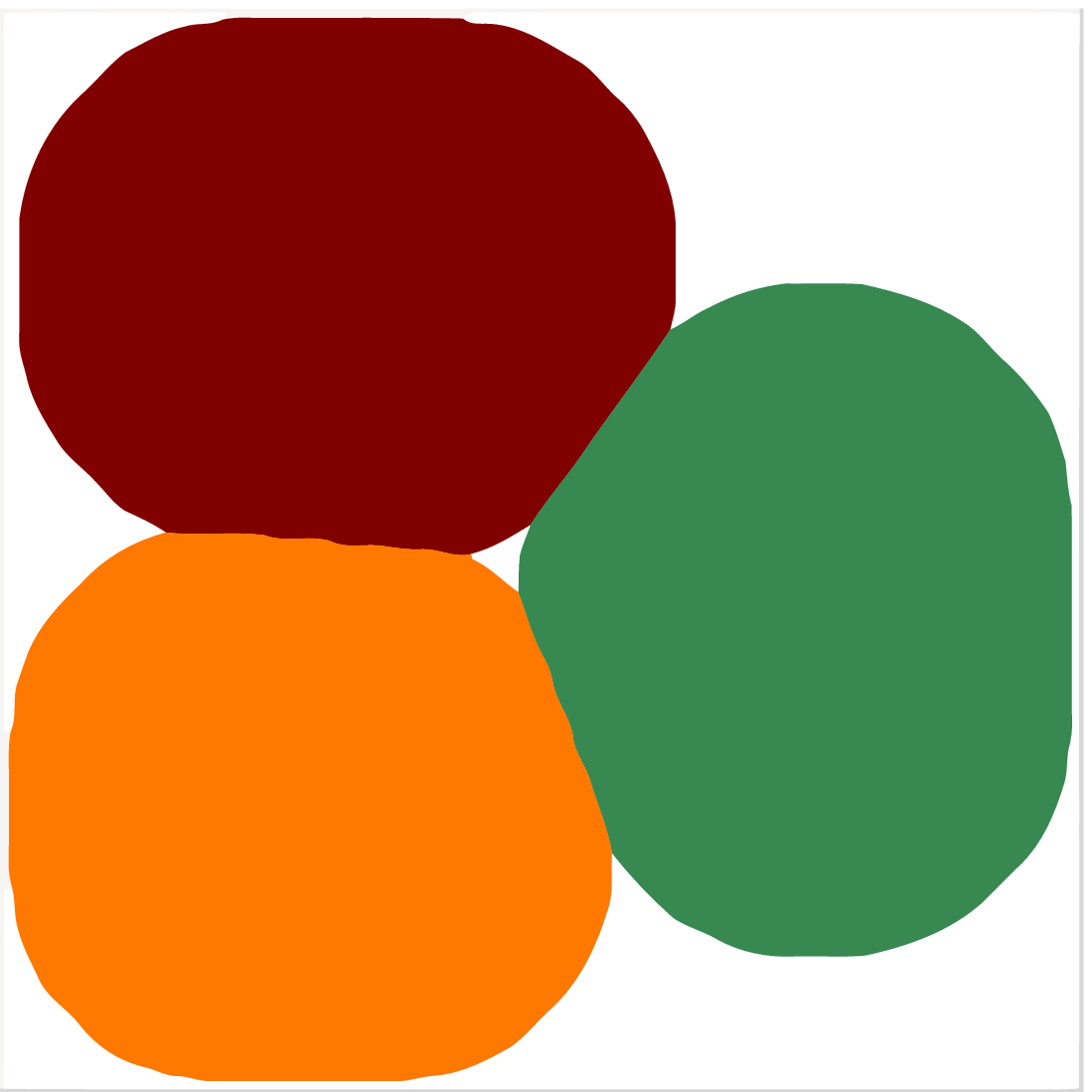}
  ~
  \includegraphics[width=0.32\textwidth]{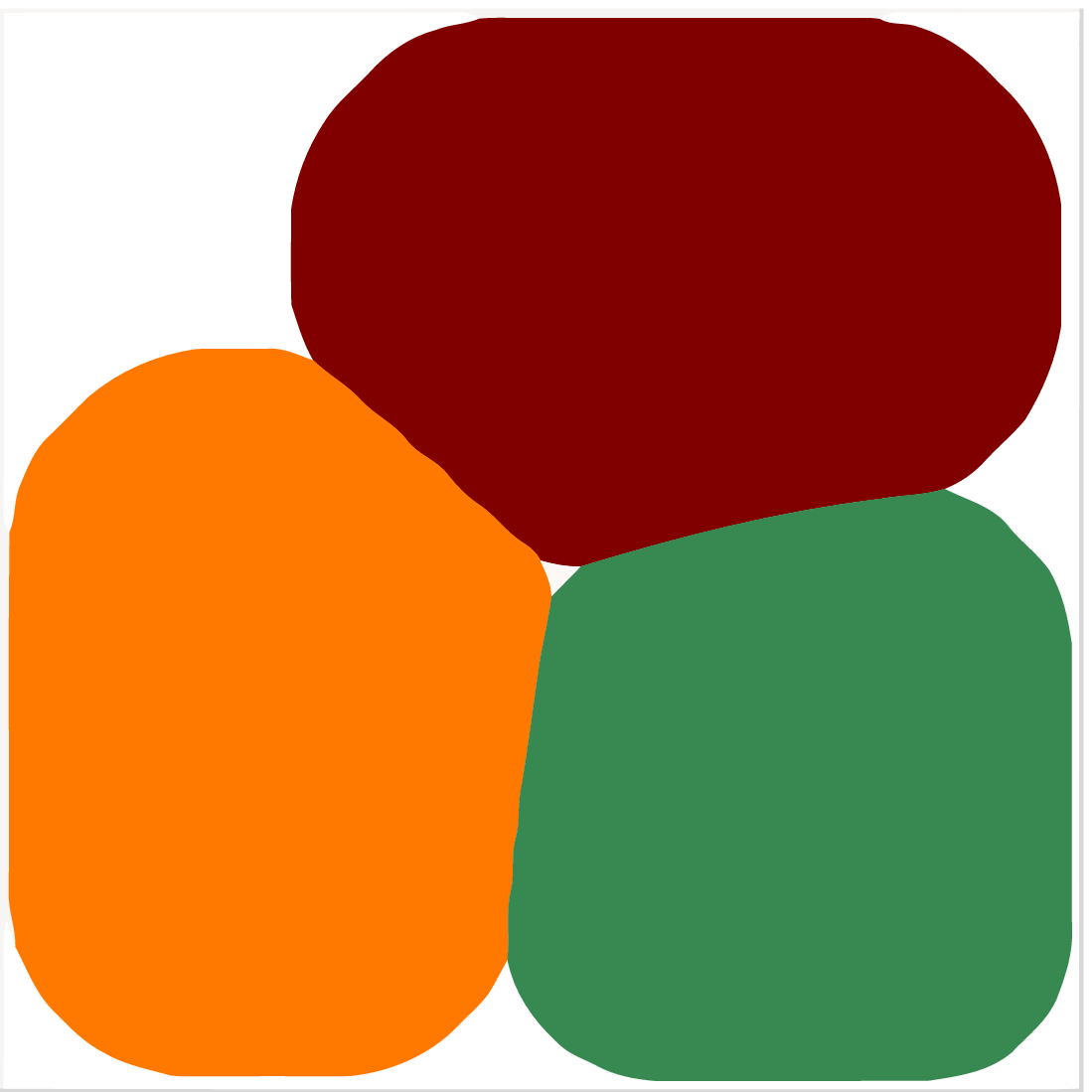} 
  ~
  \includegraphics[width=0.32\textwidth]{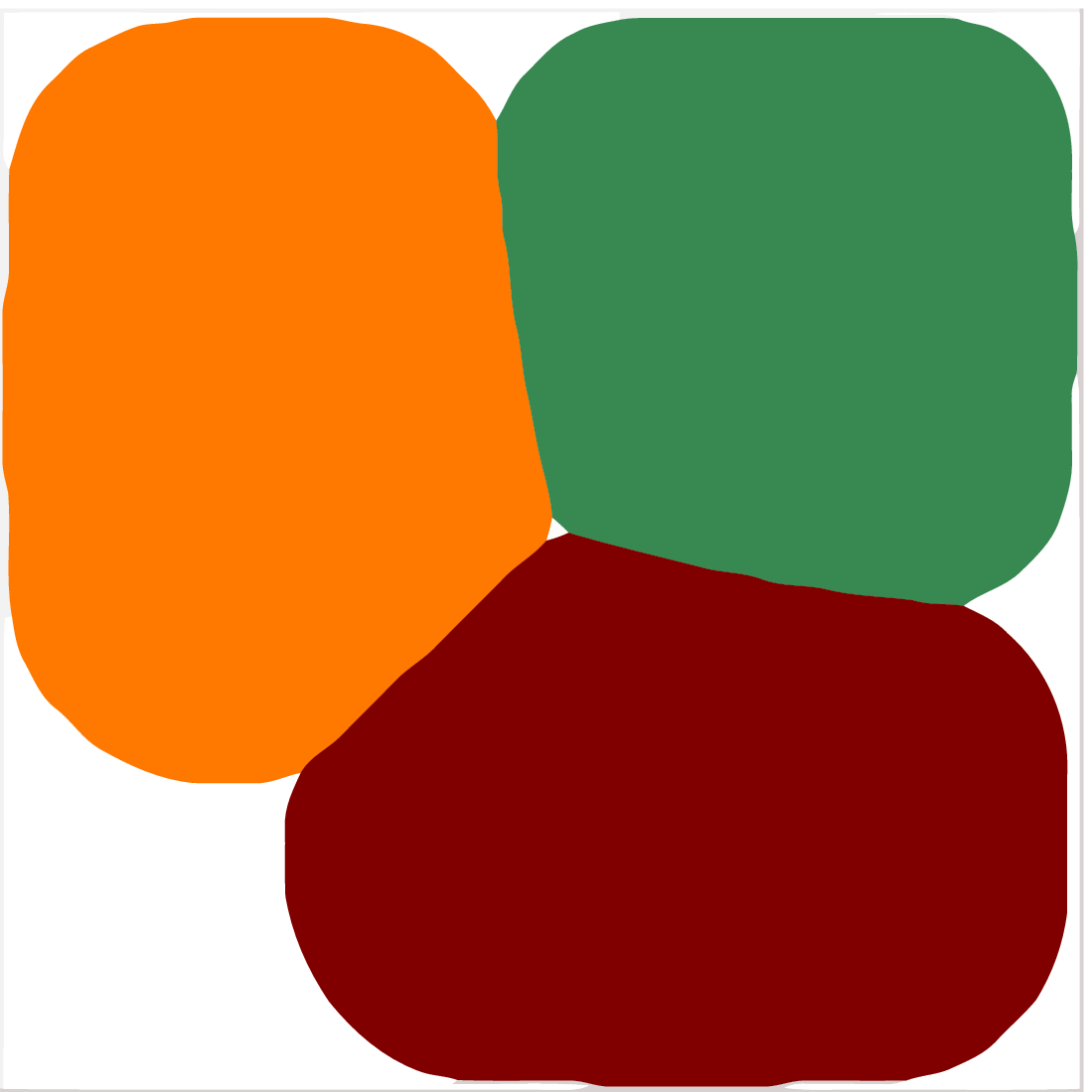}\\
  \includegraphics[width=0.32\textwidth]{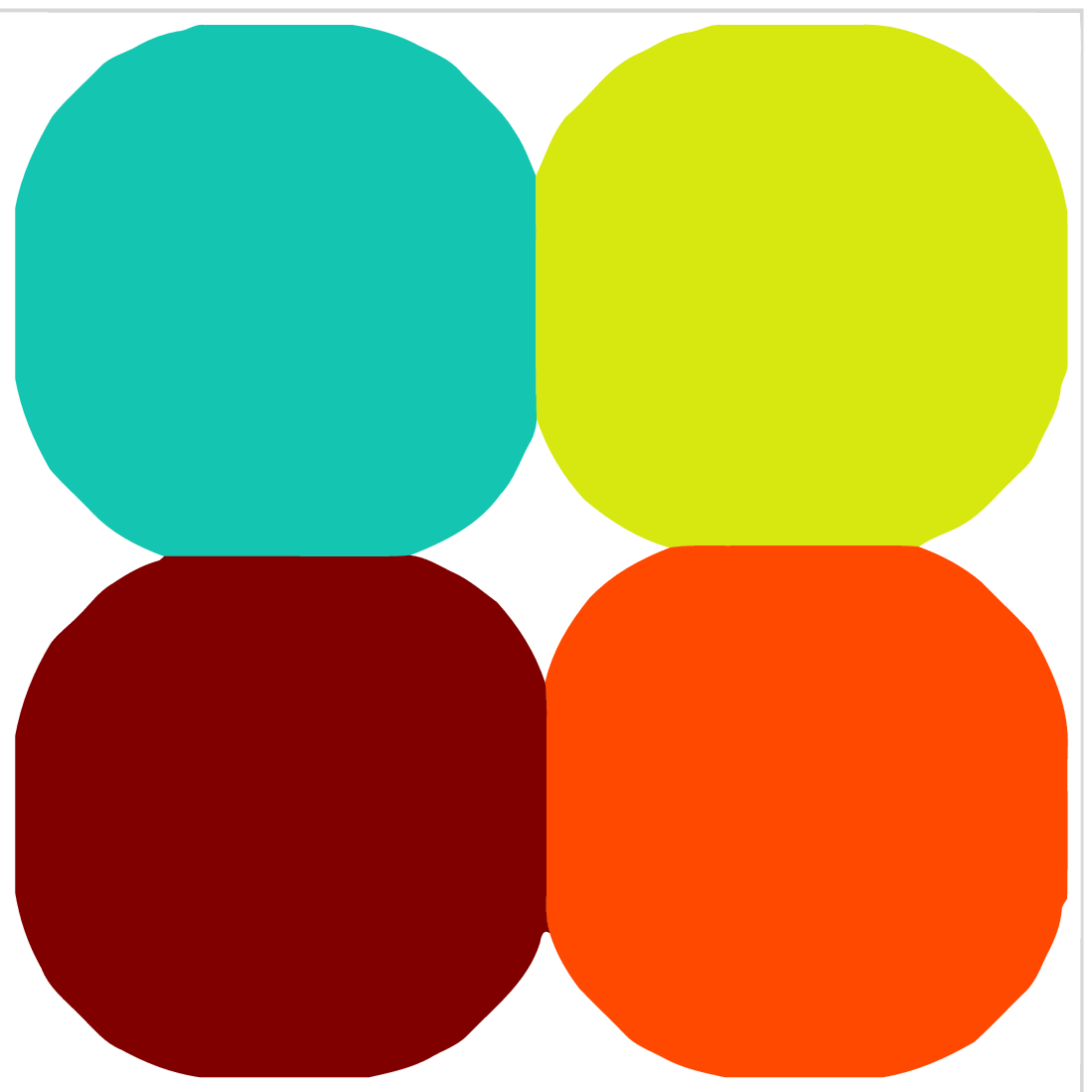}
  ~
  \includegraphics[width=0.32\textwidth]{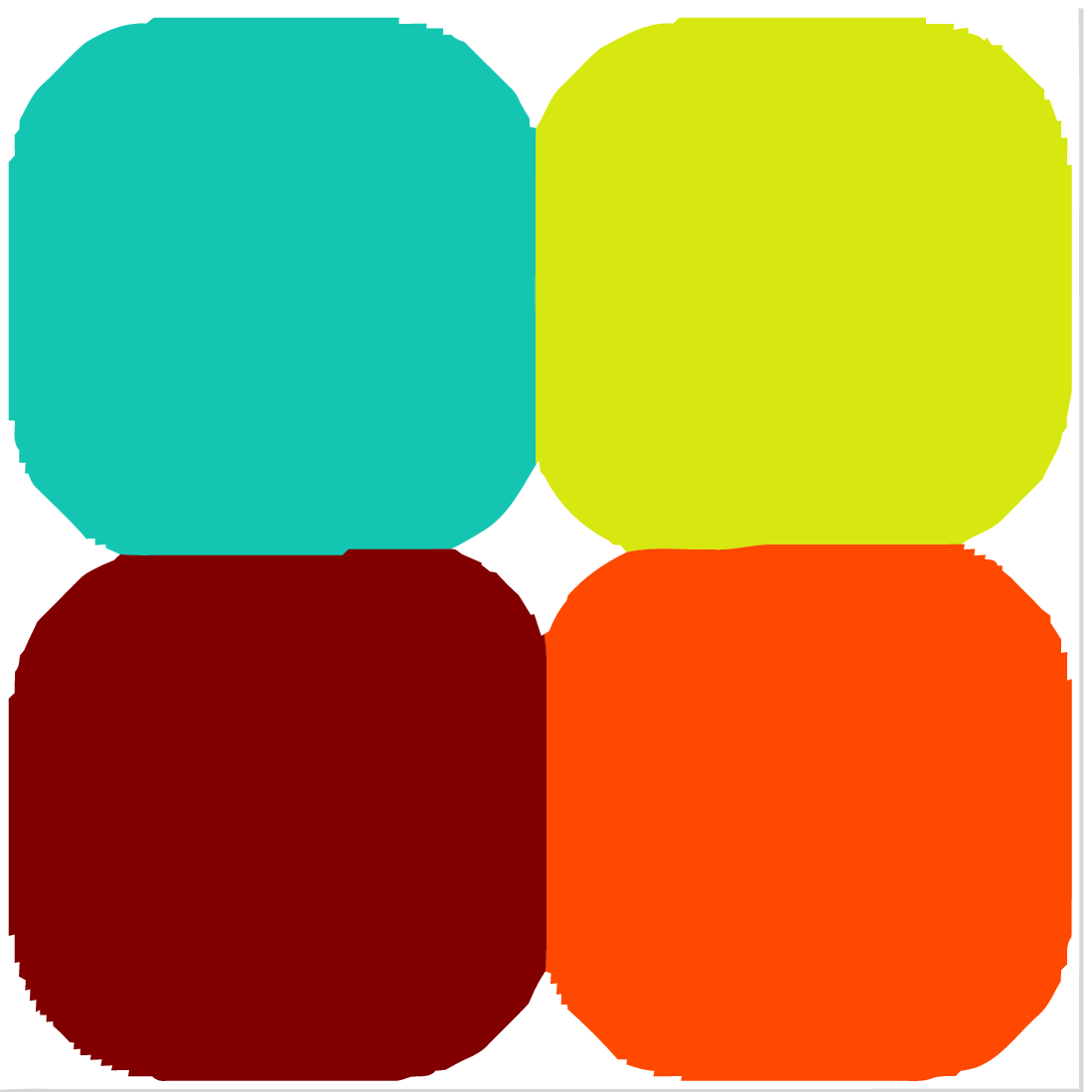}
  ~
  \includegraphics[width=0.32\textwidth]{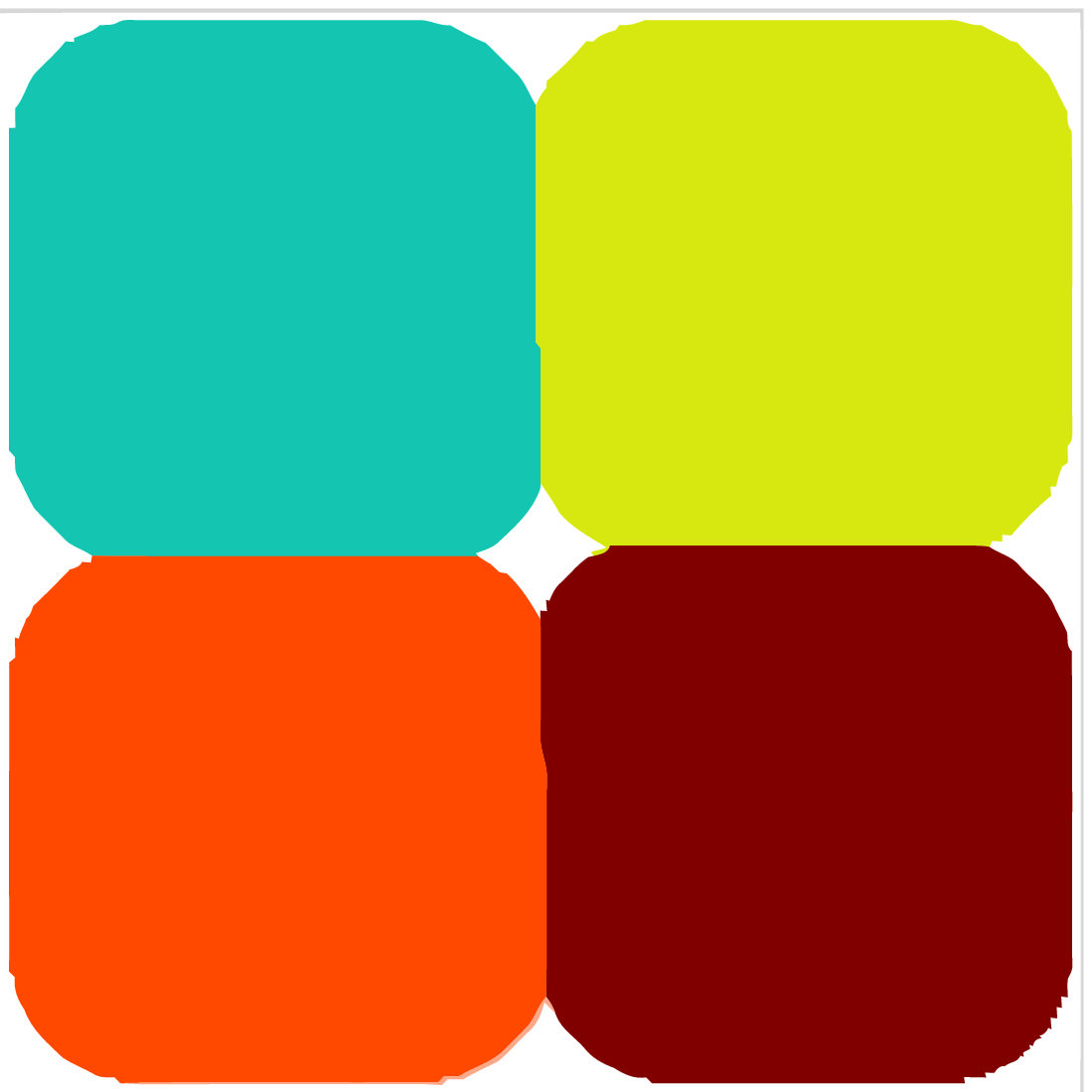}
  \caption{$k=1$, $200 \times 200$ non-periodic grid, $3$ phases ($\alpha=170,100,80$) and $4$ phases ($\alpha=250,150,100$)}
  \label{nonperiodic}
\end{figure}

In all the cases the lack of triple junction points, proved in \cite{buve}, is clearly observed, provided that the parameter $\alpha>0$ is large enough.
The lack of double points on the boundary of the square proved in Theorem \ref{mainr2} can also be noticed in Figures \ref{nonperiodic},\ref{non-rectangular}. Another phenomenon that can be observed is that the sets $\Omega_i$ near the corners of $D$ do not fill the corner. This is a fact that can be easily proved by adding a ball $B$ (i.e. subsolution for the functional $\lambda_1+\alpha|\cdot|$) outside $\Dr$, for which the corner of the square lies on the sphere $\partial B$. Now the claim can be deduced by the monotonicity Theorem \ref{teo2phm} (B), as in Theorem \ref{mainr2}.

\begin{figure}  
  \includegraphics[width=0.2\textwidth]{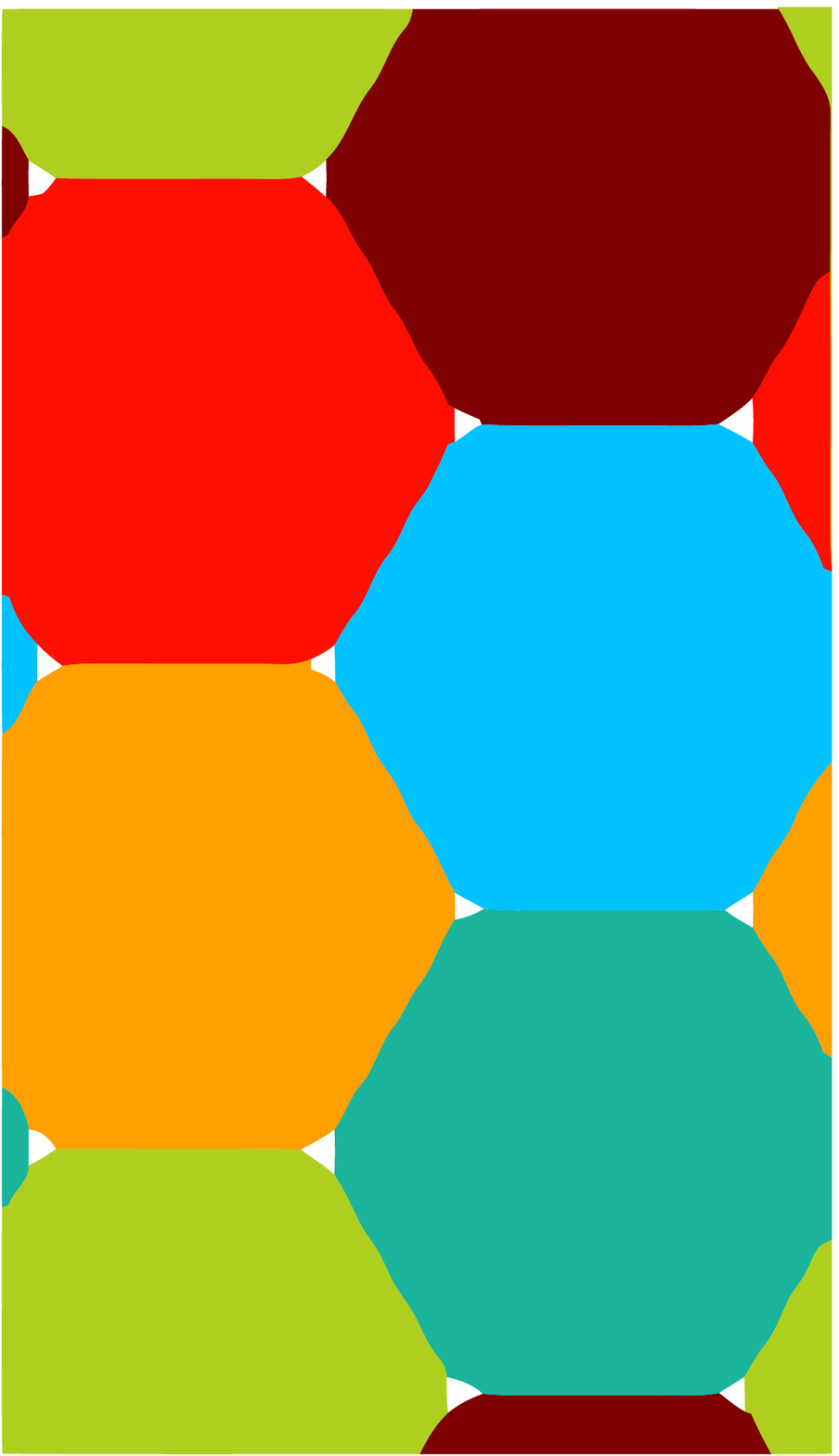}
  ~
  \includegraphics[width=0.2\textwidth]{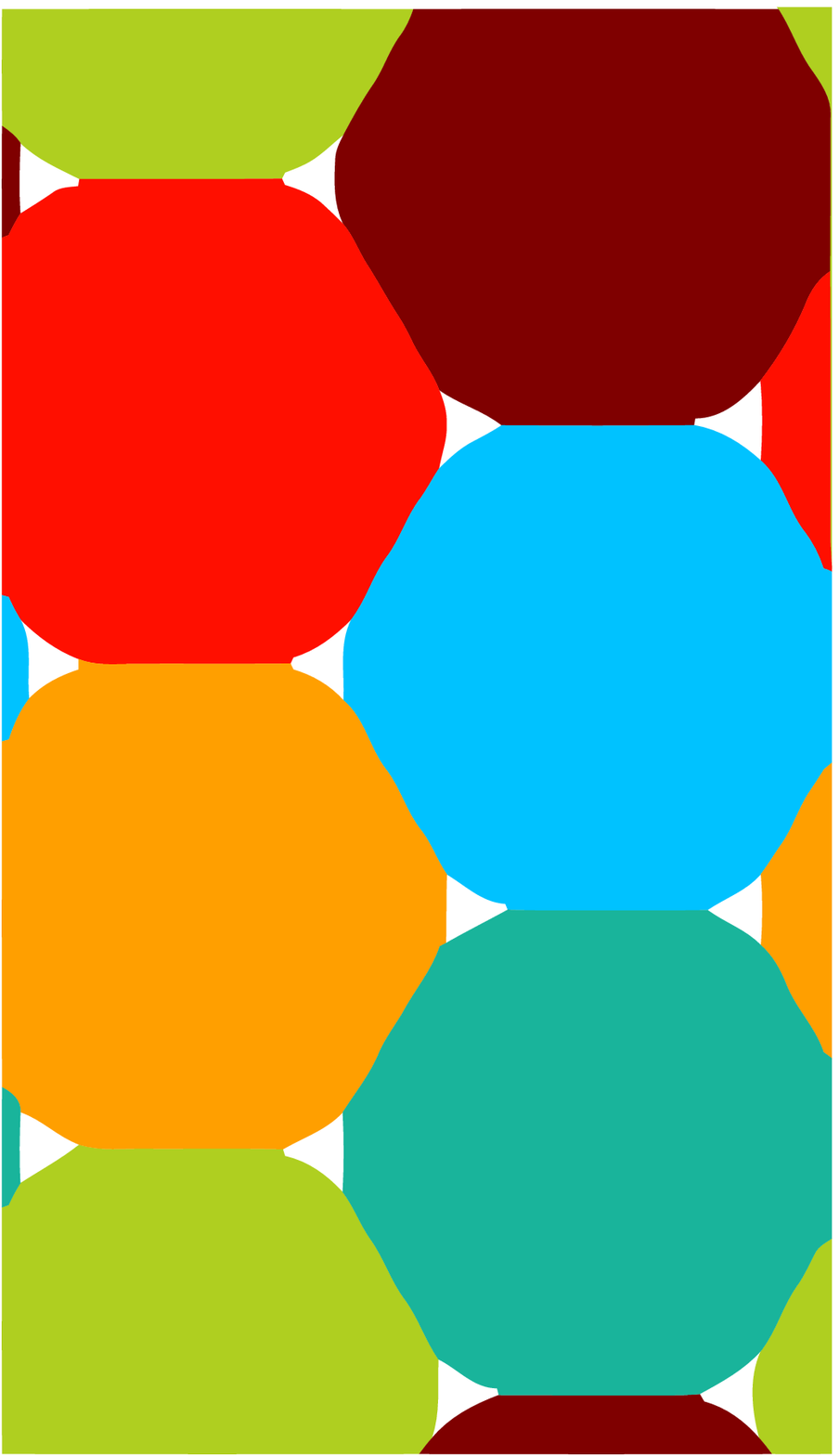}
  ~
  \includegraphics[width=0.2\textwidth]{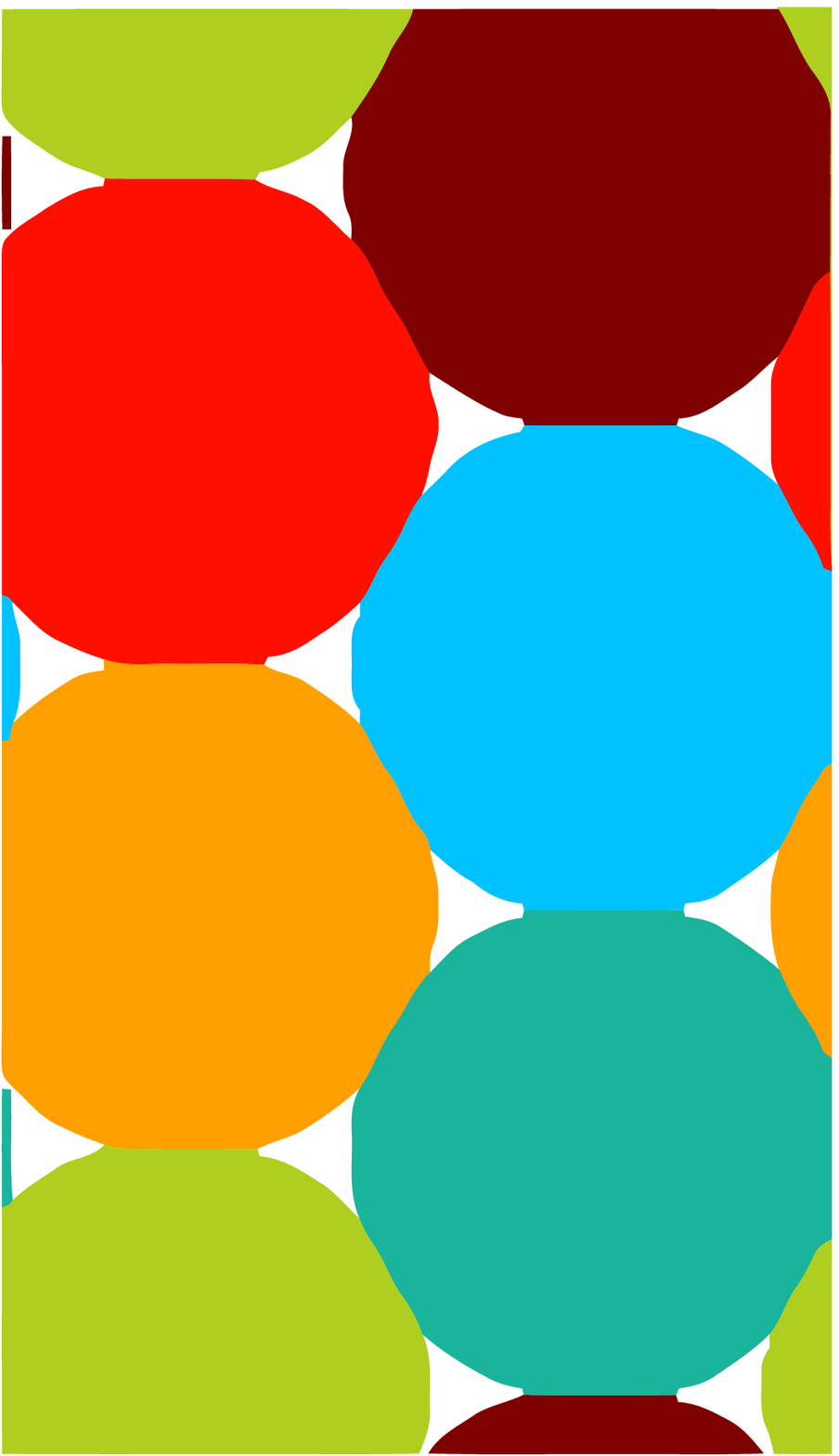}
  ~
  \includegraphics[width=0.2\textwidth]{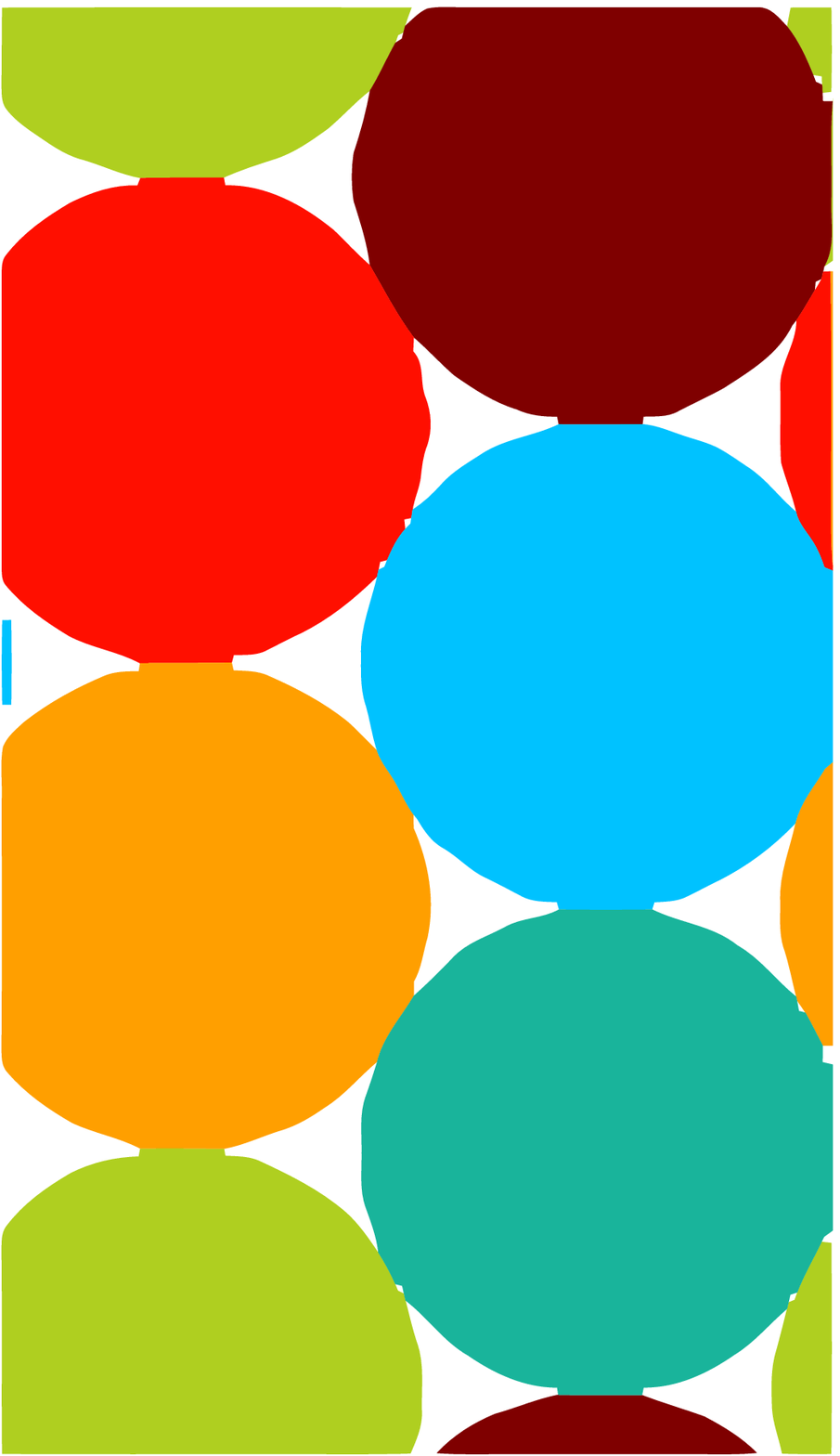}
  \caption{$k=1$, illustration of the monotonicity property. Values of $\alpha$: $150, 200, 250, 300$ (left to right)}
  \label{monotonic1} 
\end{figure}

\begin{figure}  
  \includegraphics[width=0.3\textwidth]{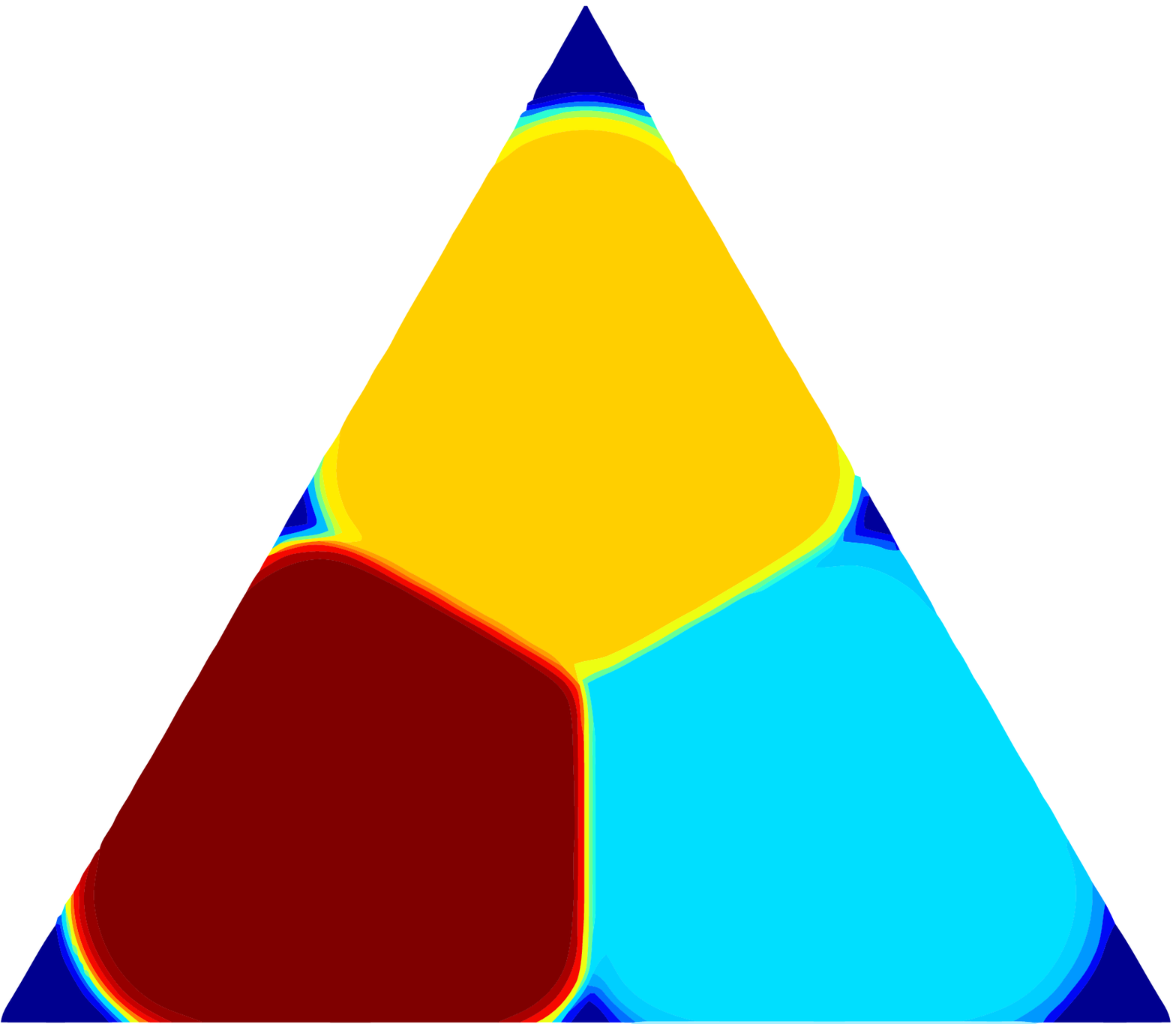}
  ~
  \includegraphics[width=0.3\textwidth]{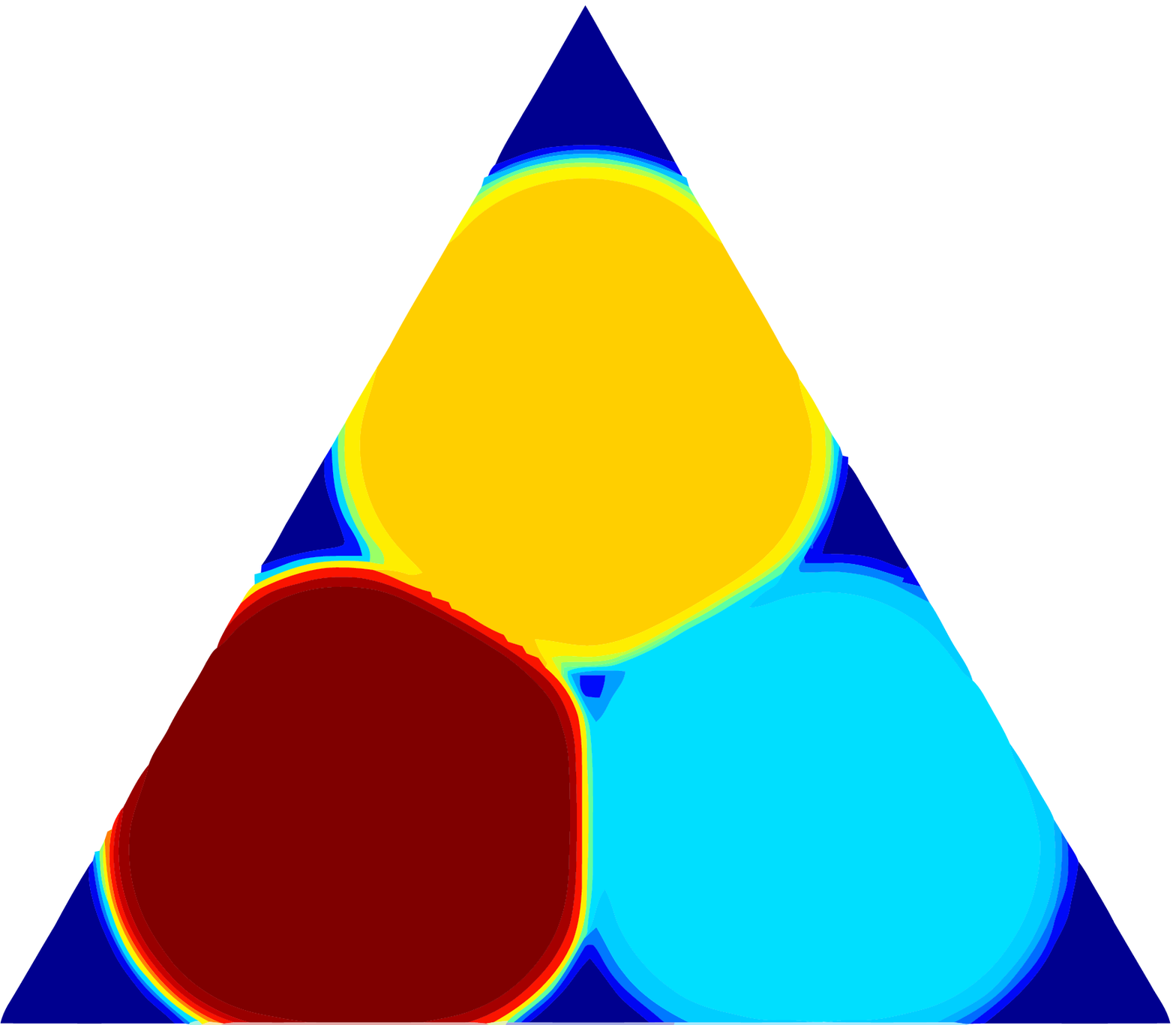}
  ~
  \includegraphics[width=0.3\textwidth]{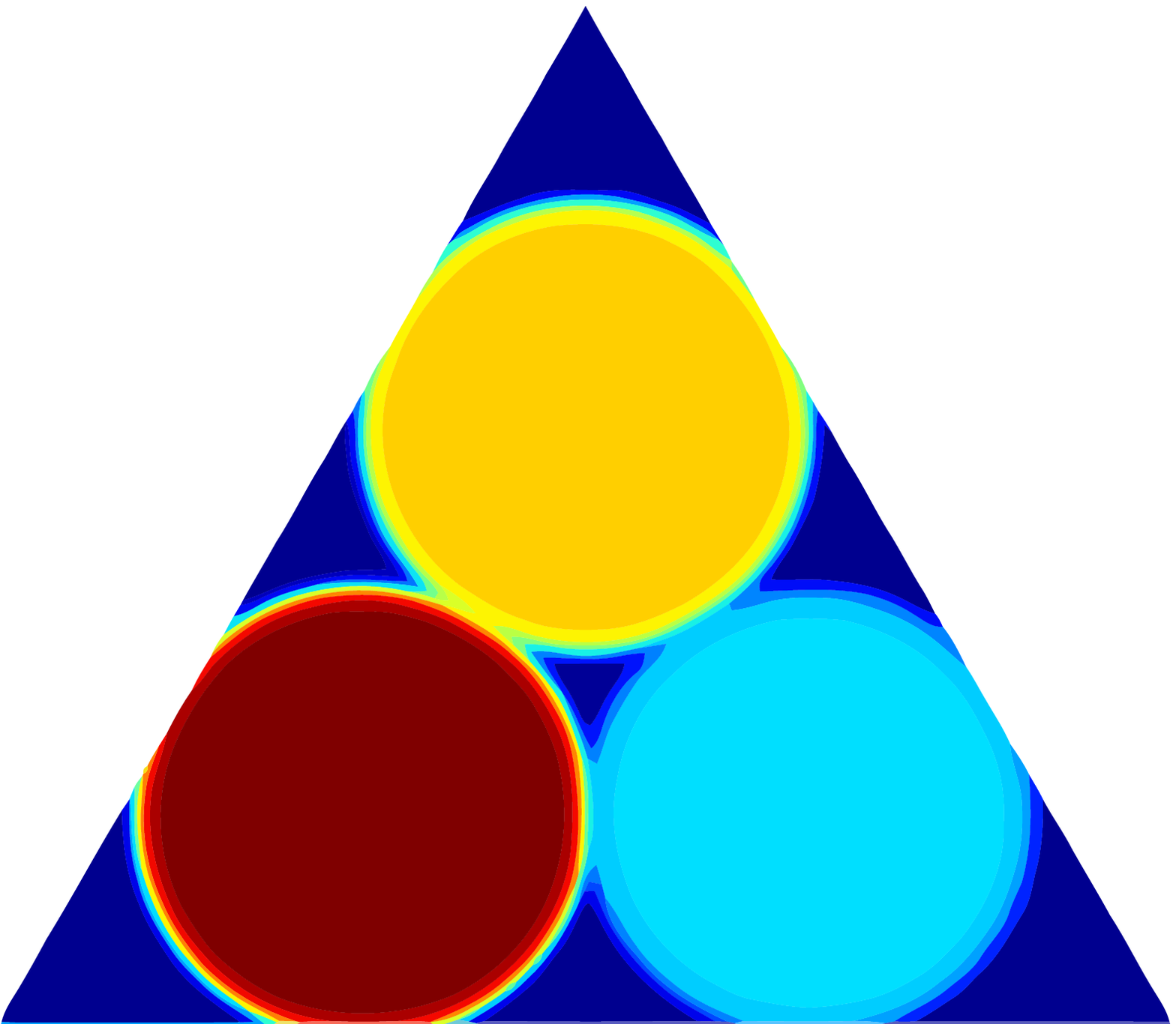}
  \caption{$k=1$, illustration of the monotonicity property in the case of an equilateral triangle. Values of $\alpha$: $10, 25,50$ (left to right)}
  \label{monotonic2} 
\end{figure}

Some fine qualitative properties of the optimal configurations $(\Omega_1,\dots,\Omega_h,\Omega_{h+1})$, which are still open questions, were observed during the numerical simulations.

\begin{itemize}
\item The set of one-phase points $\partial\Omega_j\cap\partial\Omega_{h+1}$ on the boundary of the $j$th optimal cell $\Omega_j$ is locally a graph of a \emph{convex function}.
\item For each pair of distinct indices $i,j\in\{1,\dots,h\}$, there are exactly two boundary two-phase points on the common boundary $\partial\Omega_i\cap\partial\Omega_j$, i.e.
$$\HH^0\big(\partial\Omega_i\cap\partial\Omega_j\cap\partial\Omega_{h+1}\big)=2.$$ 
\item If $x_0\in\partial\Omega_i\cap\partial\Omega_j\cap\partial\Omega_{h+1}$ is a boundary two-phase point, then the set $\Omega_i\cap\Omega_j$ has a cusp in $x_0$. More precisely, for $r>0$ small enough, the free boundaries $\partial\Omega_i\cap\partial\Omega_{h+1}\cap B_r(x_0)$ and $\partial\Omega_j\cap\partial\Omega_{h+1}\cap B_r(x_0)$ are graphs of convex functions meeting tangentially in the origin $x_0$.
\end{itemize} 

Finally, we considered the periodic version of the problem \eqref{optsum} on the square $[0,1]\times[0,1]$ and in other rectangular domains, in attempt to simulate a "partition" of the whole space $\R^2$ (see Figure \ref{periodic}, Figure \ref{monotonic1}). For small enough constants $\alpha>0$ we obtain a configuration with touching hexagons with rounded corners, in support of the numerical results in \cite{buboou}. 

Most of the tests we made were in the case $k=1$, but the algorithm works for $k=2$ as well. The main issue in the case of higher eigenvalues concerns the differentiability of the eigenvalues with respect to perturbations, which is well known to be closely related to their multiplicity. Secondly, we were not able to prove that for $k \geq 2$ the relaxed formulation converges to the actual problem when $C \to +\infty$. Nevertheless, we were able to obtain some interesting numerical results also in the case $k=2$ and one example can be seen in Figure \ref{periodic}. 

As stated in Theorem \ref{mainsurf} the theoretical results also extend to the case of the Laplace-Beltrami fundamental eigenvalues on surfaces. Using the same finite elements procedure as in the case of non-rectangular domains, we were able to compute numerically some optimal configurations on the sphere, observing the same behavior as in the plane: the lack of triple points and monotonicity with respect to $\alpha$. (see Figure \ref{sphere}) We notice that in the cases $h \in \{3,4,6,12\}$ the optimal configurations converge to the corresponding regular tiling of the sphere (Y partition, regular tetrahedron, cube, dodecahedron) as $\alpha \to 0$.

\begin{figure}  
  \includegraphics[width=0.32\textwidth]{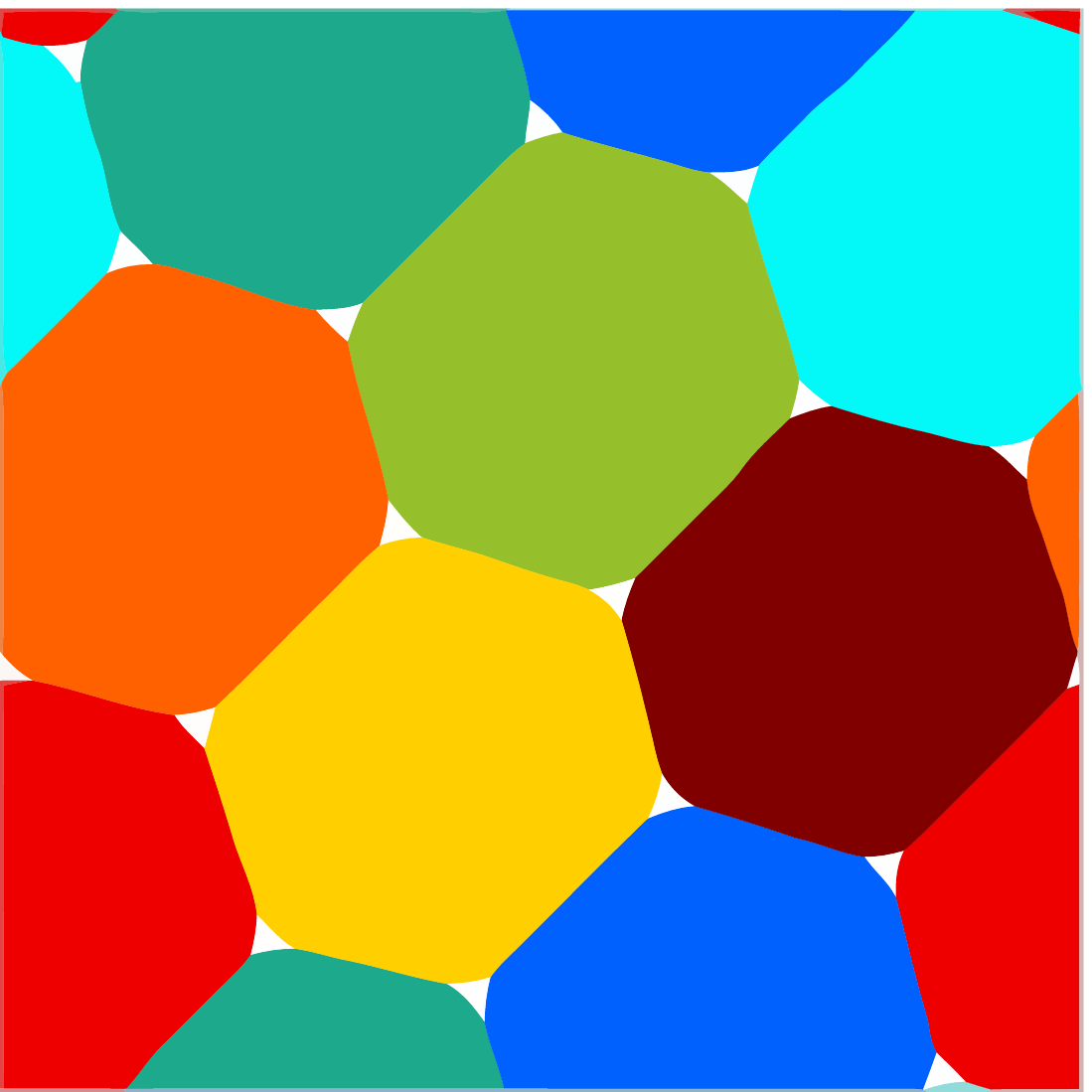}
  ~
  \includegraphics[width=0.32\textwidth]{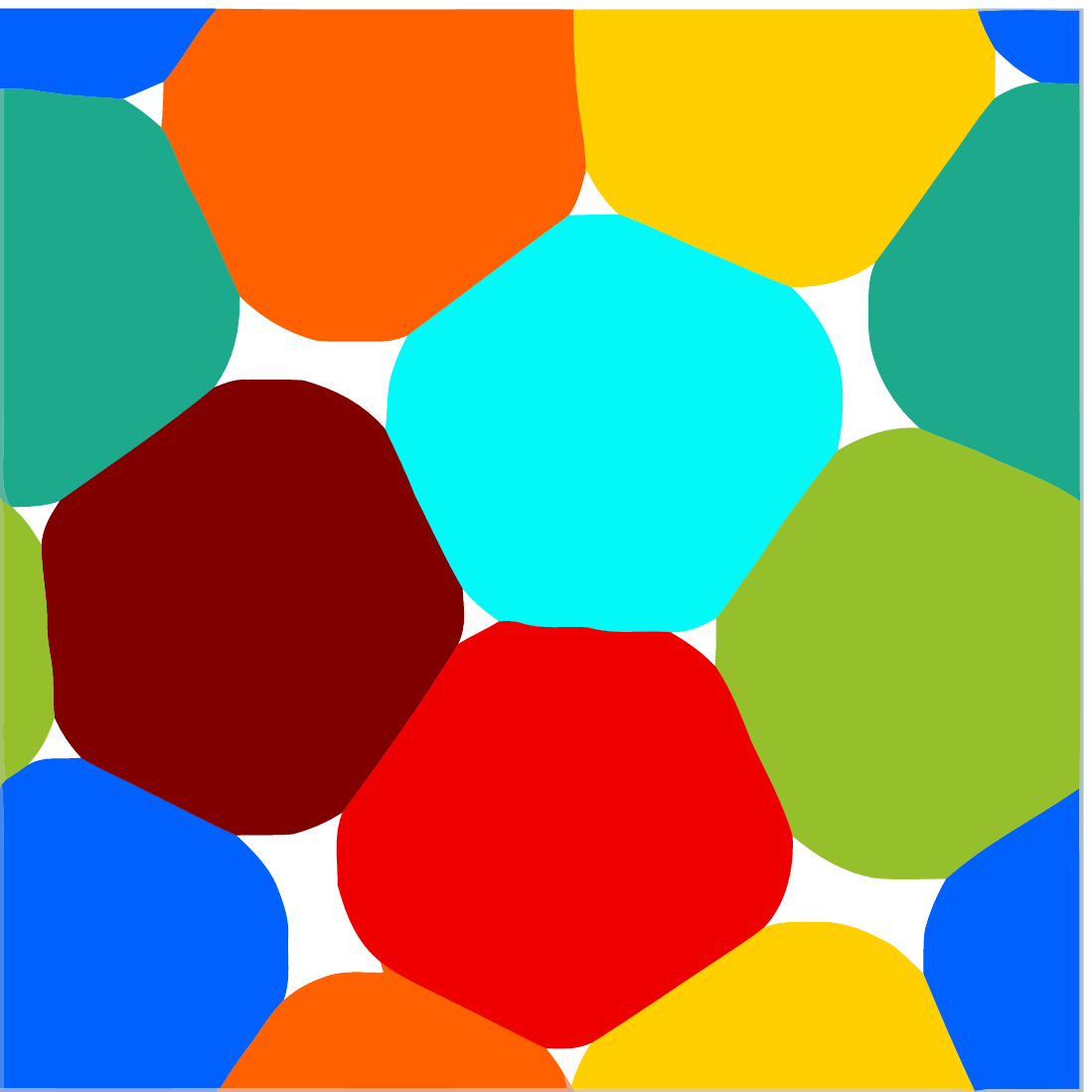}
  ~
  \includegraphics[width=0.32\textwidth]{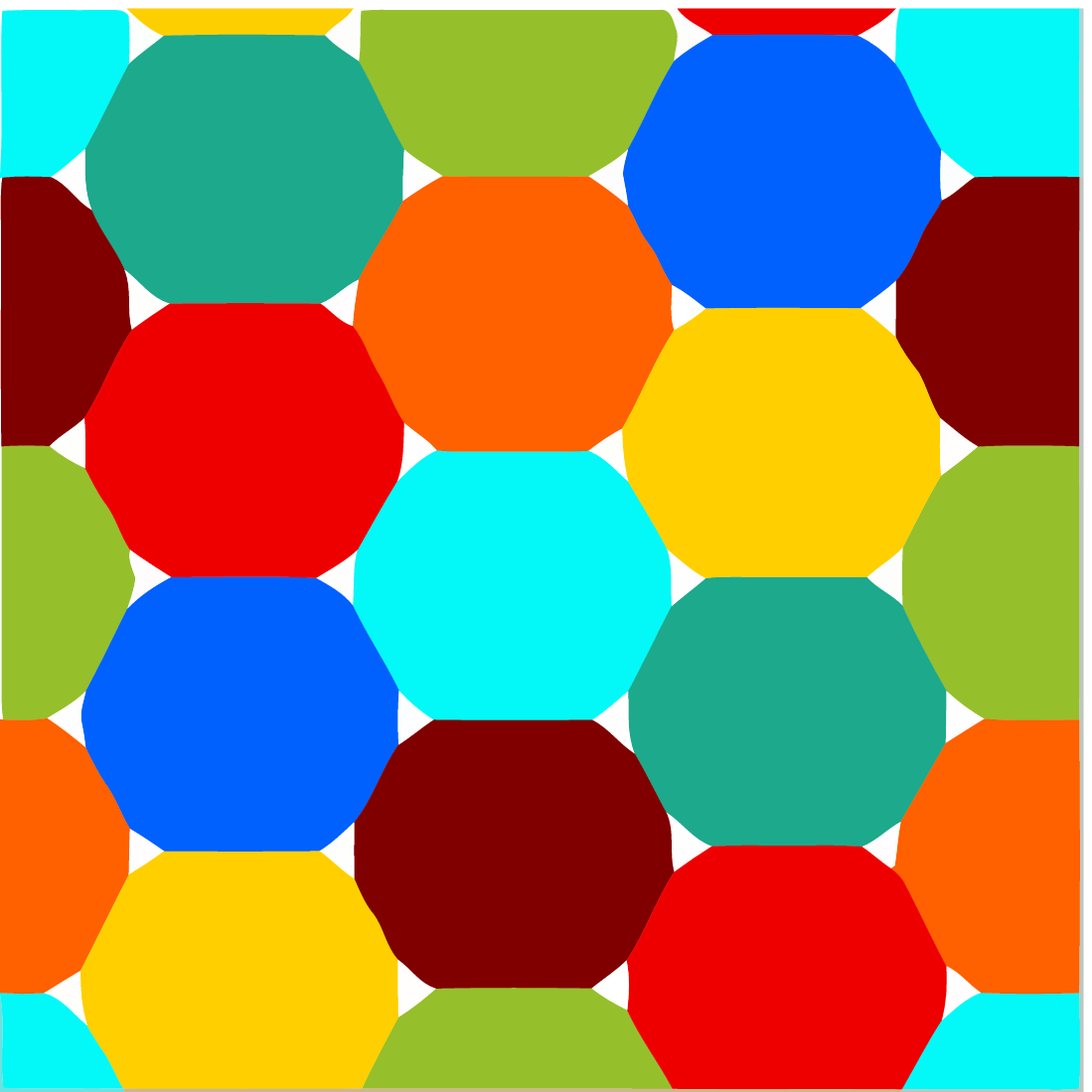}
  \caption{$k=1$, $200 \times 200$ periodic grid, $8$ phases, $\alpha=500,580$ and $k=2$, $8$ phases, $\alpha=270$}
  \label{periodic} 
\end{figure}

\begin{figure}  
  \includegraphics[width=0.32\textwidth]{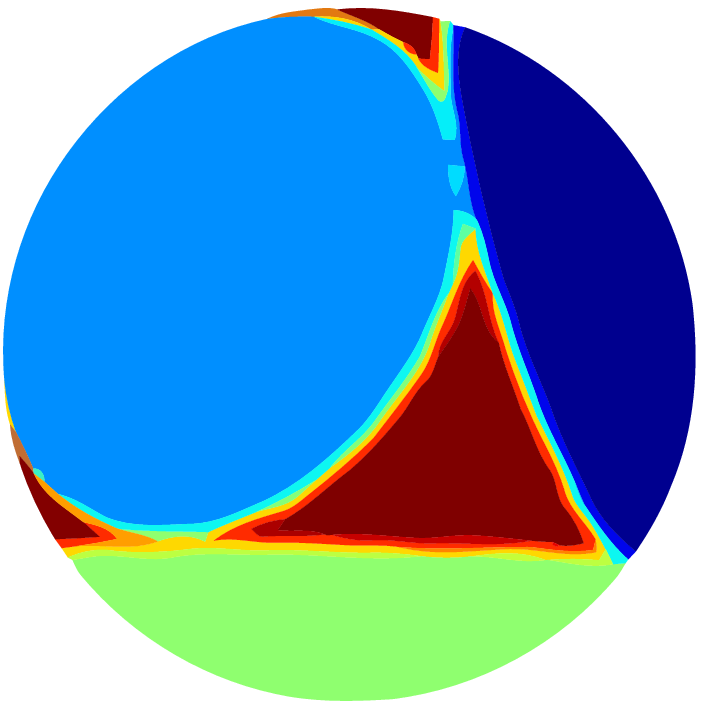}
  ~
  \includegraphics[width=0.32\textwidth]{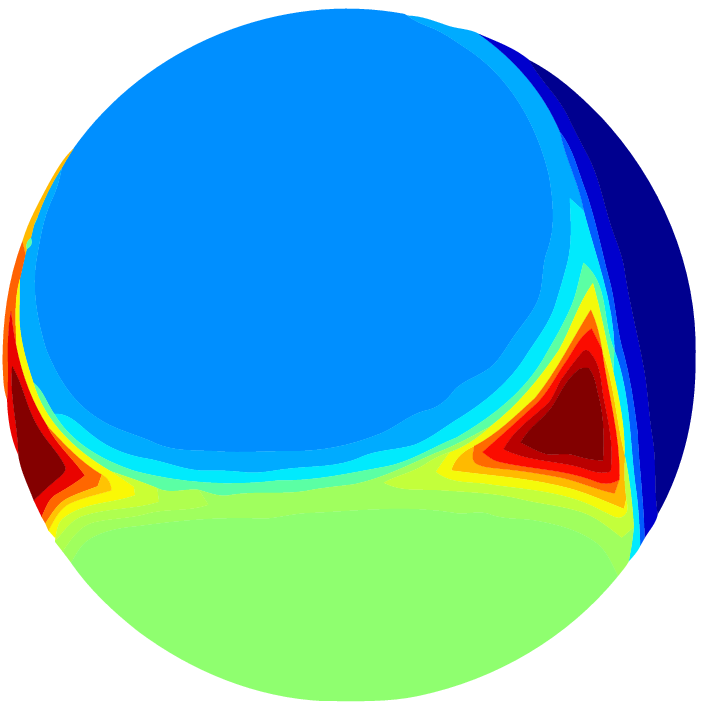}
  ~
  \includegraphics[width=0.32\textwidth]{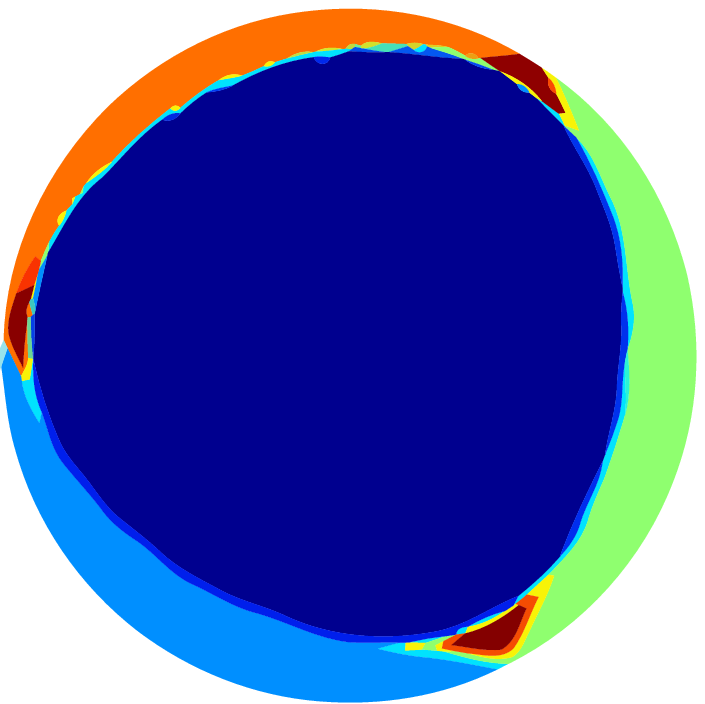}
  \caption{Optimal configurations on the sphere in the case of four phases, for decreasing values of $\alpha$}
  \label{sphere} 
\end{figure}

\begin{figure}
\centering
  \includegraphics[width=0.15\textwidth]{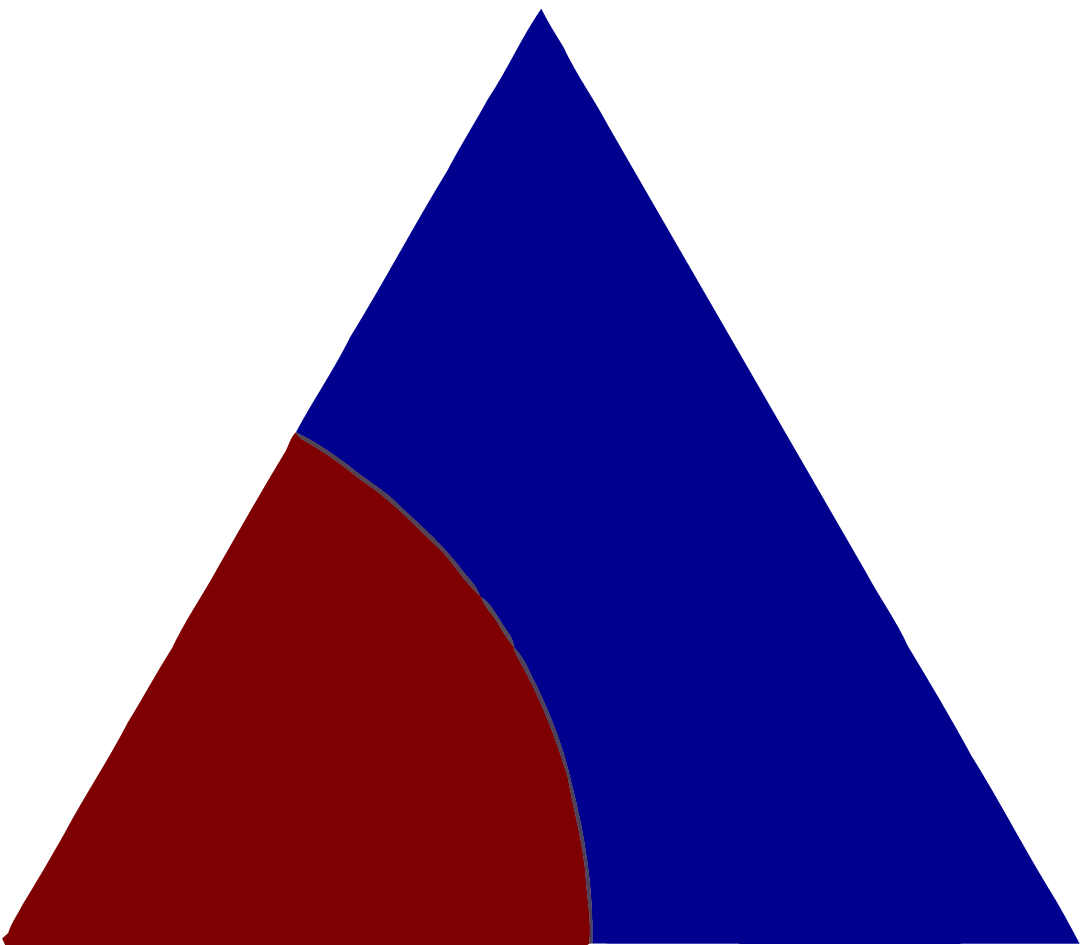}
  ~
  \includegraphics[width=0.15\textwidth]{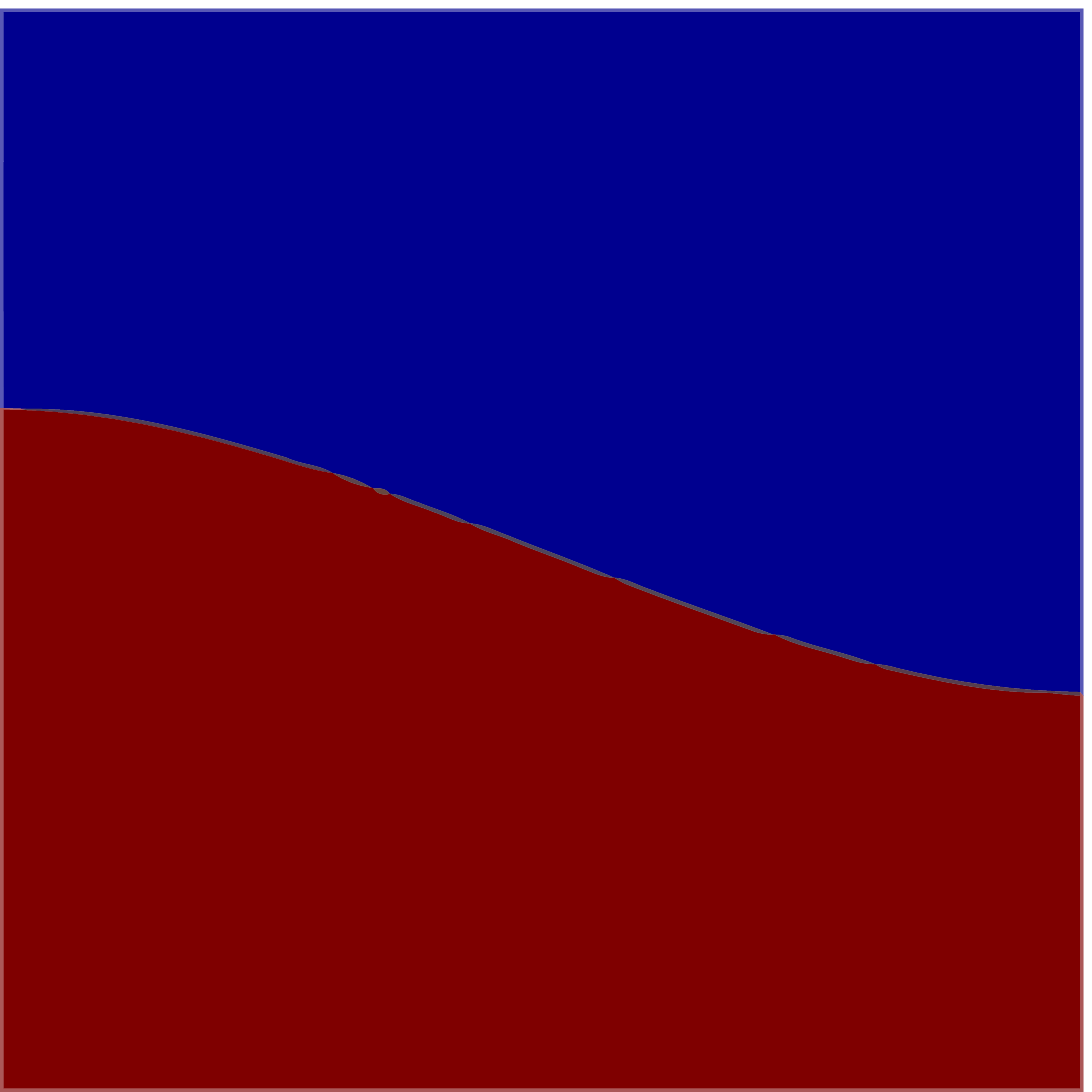} 
  ~
  \includegraphics[width=0.15\textwidth, angle = 90,origin=c]{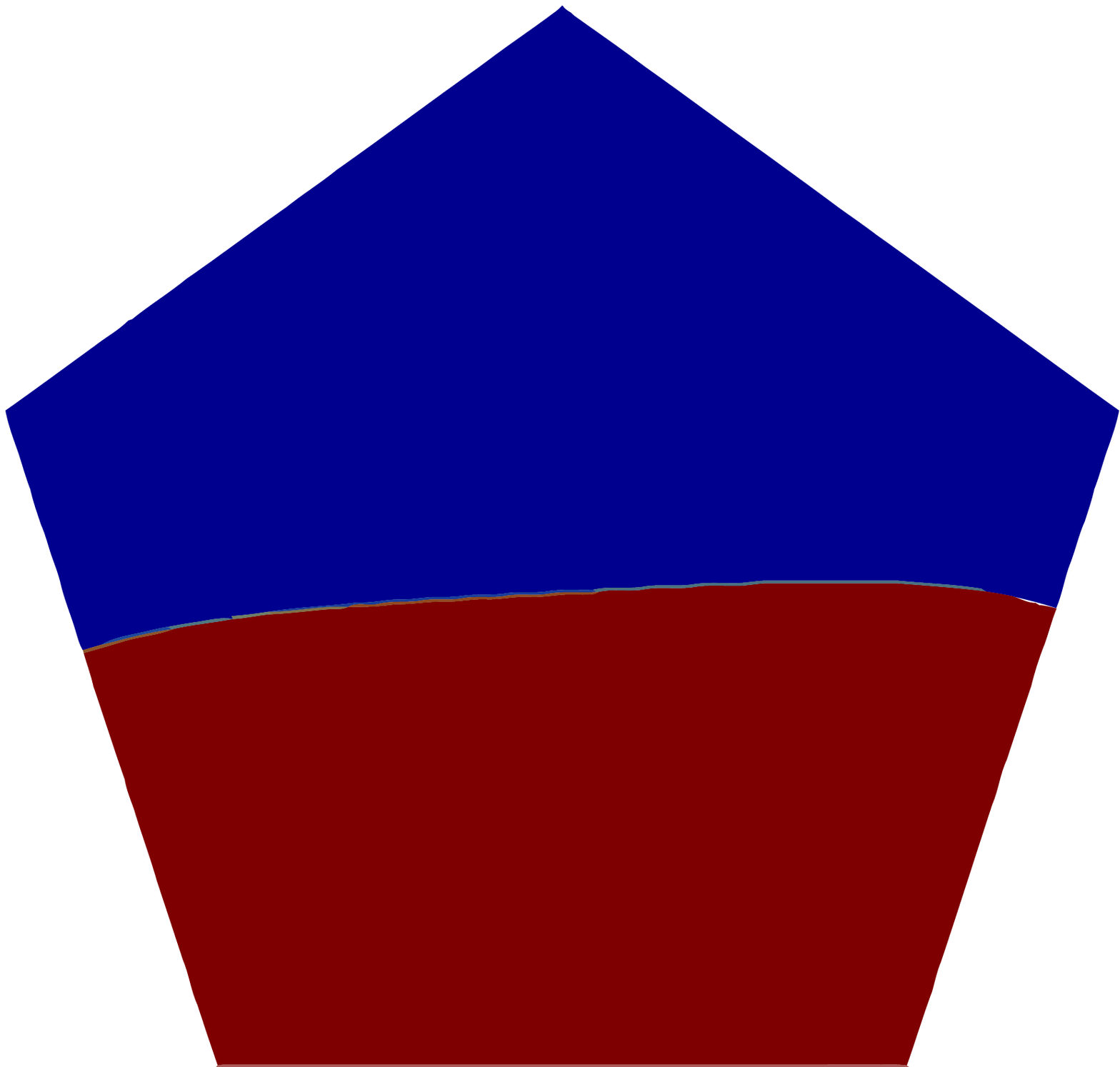}
  ~
  \includegraphics[width=0.15\textwidth]{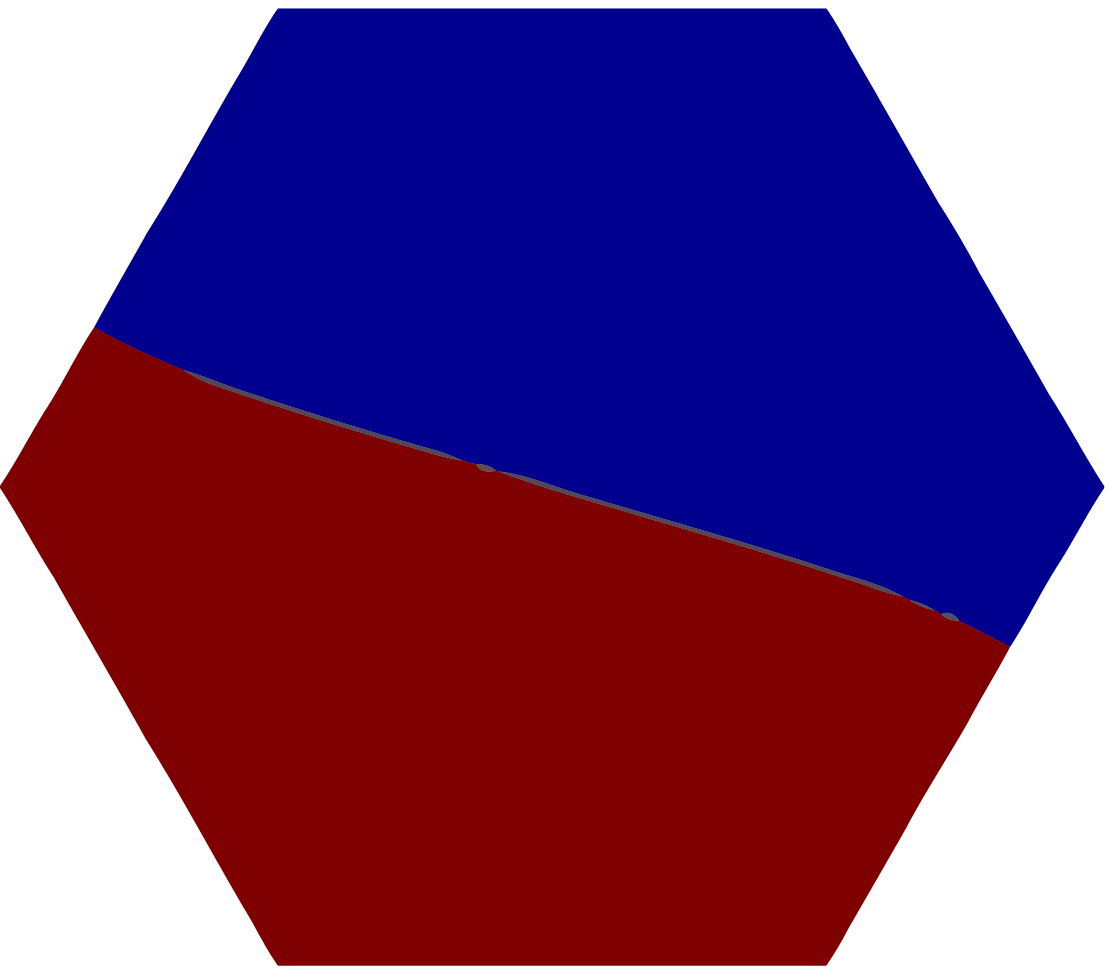}
  ~
  \includegraphics[width=0.15\textwidth]{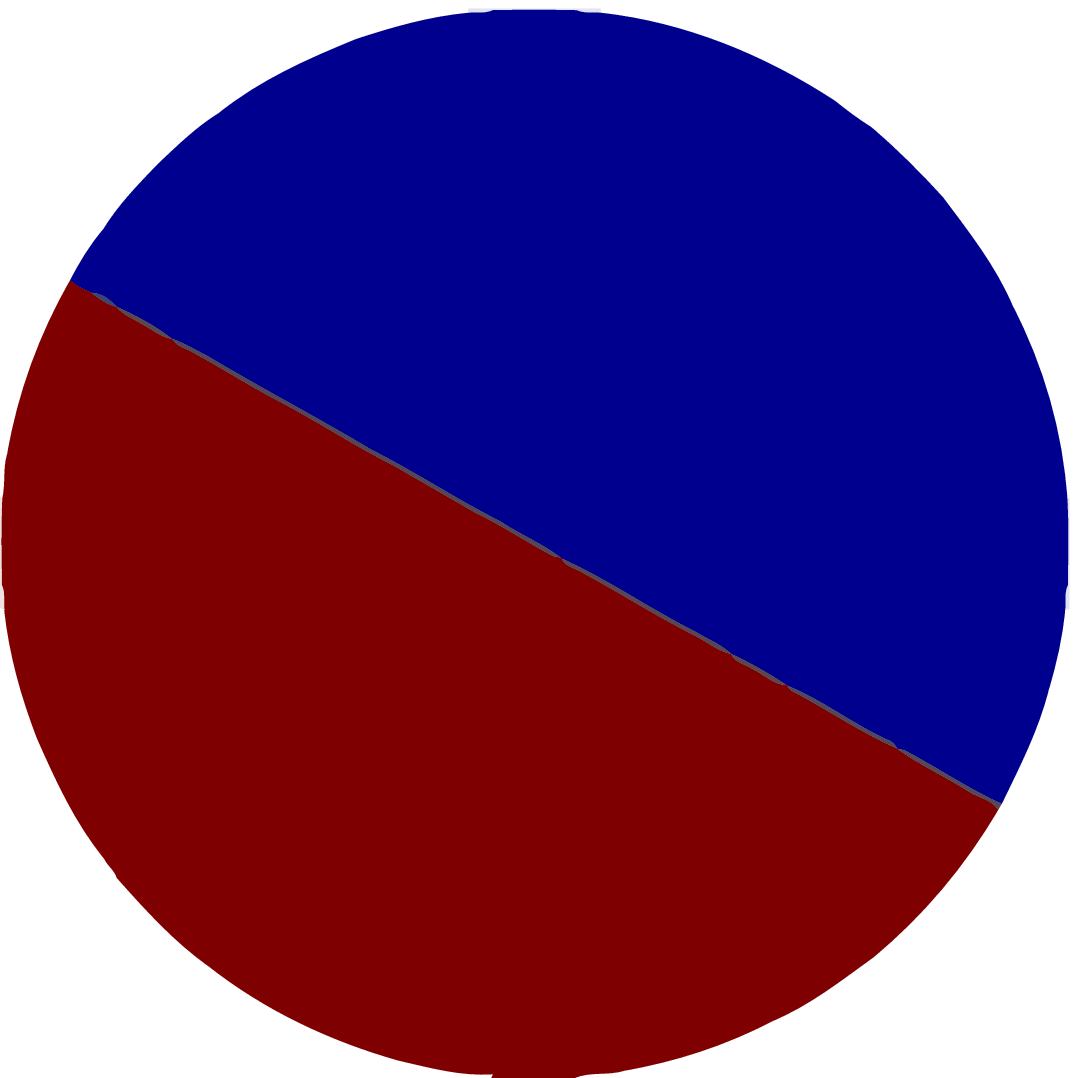}\\
  \includegraphics[width=0.15\textwidth]{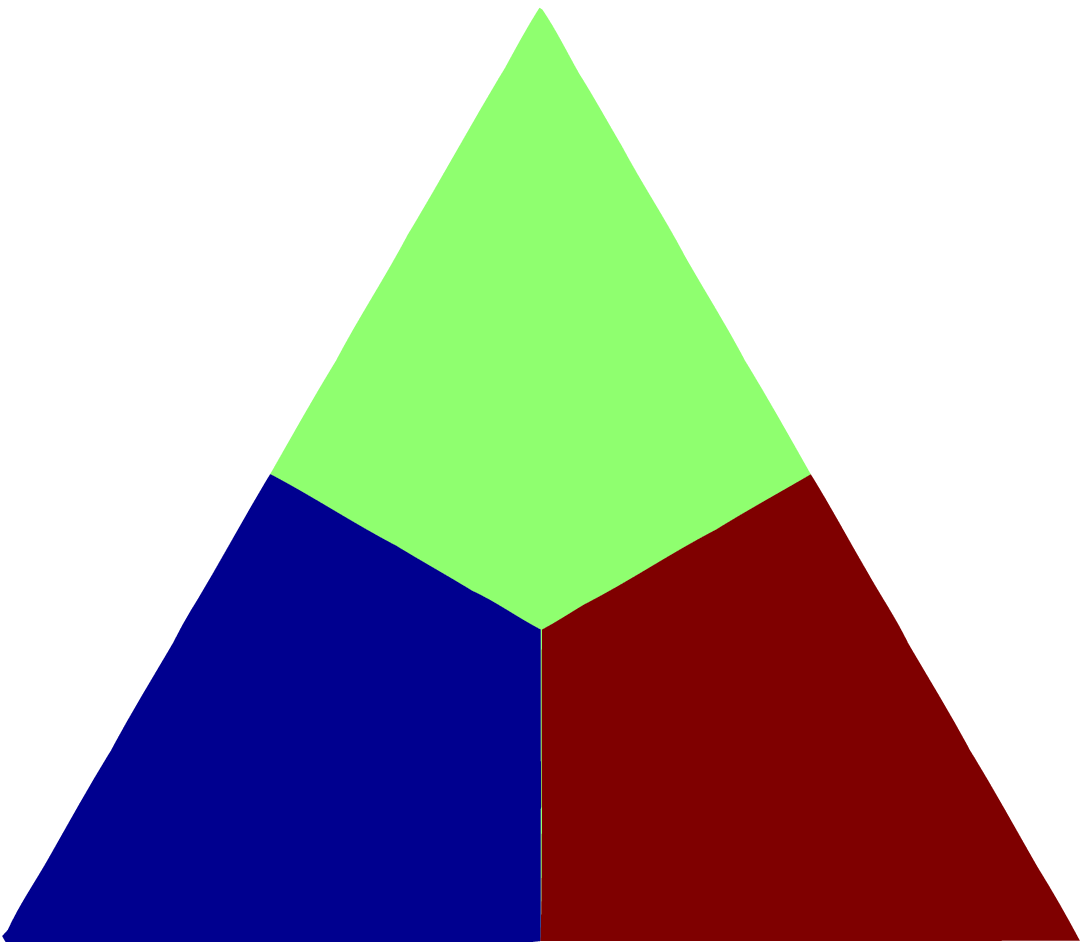}
  ~
  \includegraphics[width=0.15\textwidth]{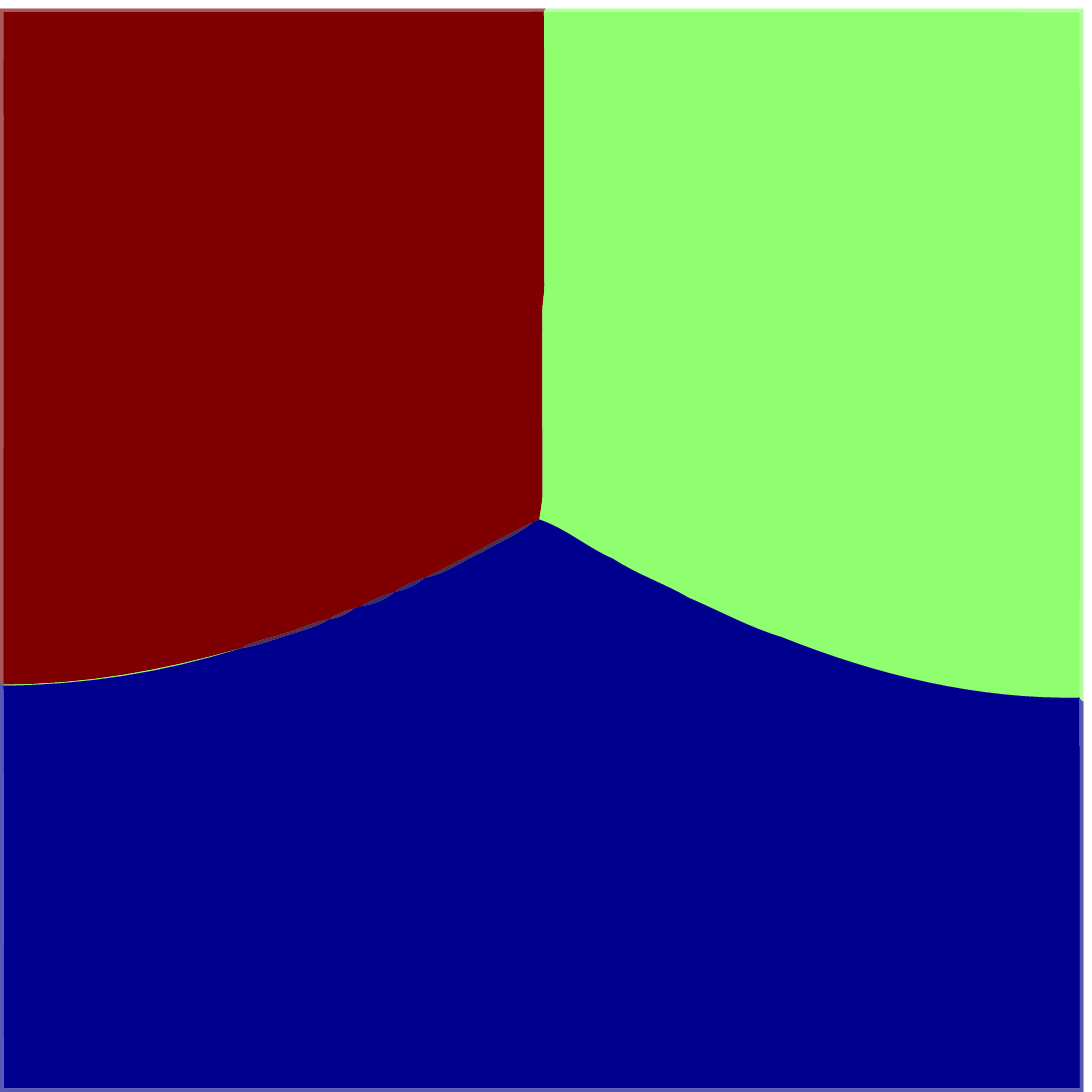} 
  ~
  \includegraphics[width=0.15\textwidth, angle = 90,origin=c]{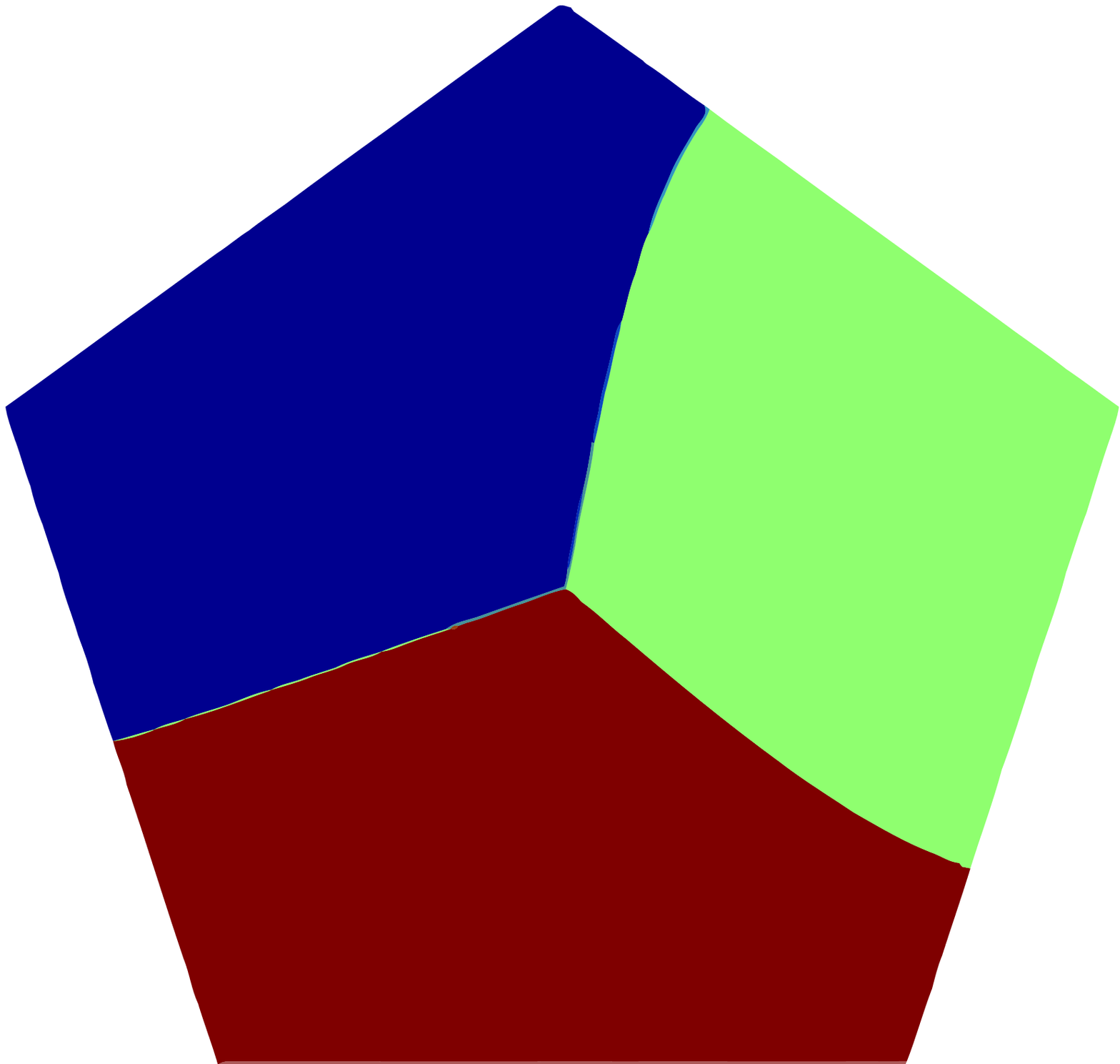}
  ~
  \includegraphics[width=0.15\textwidth]{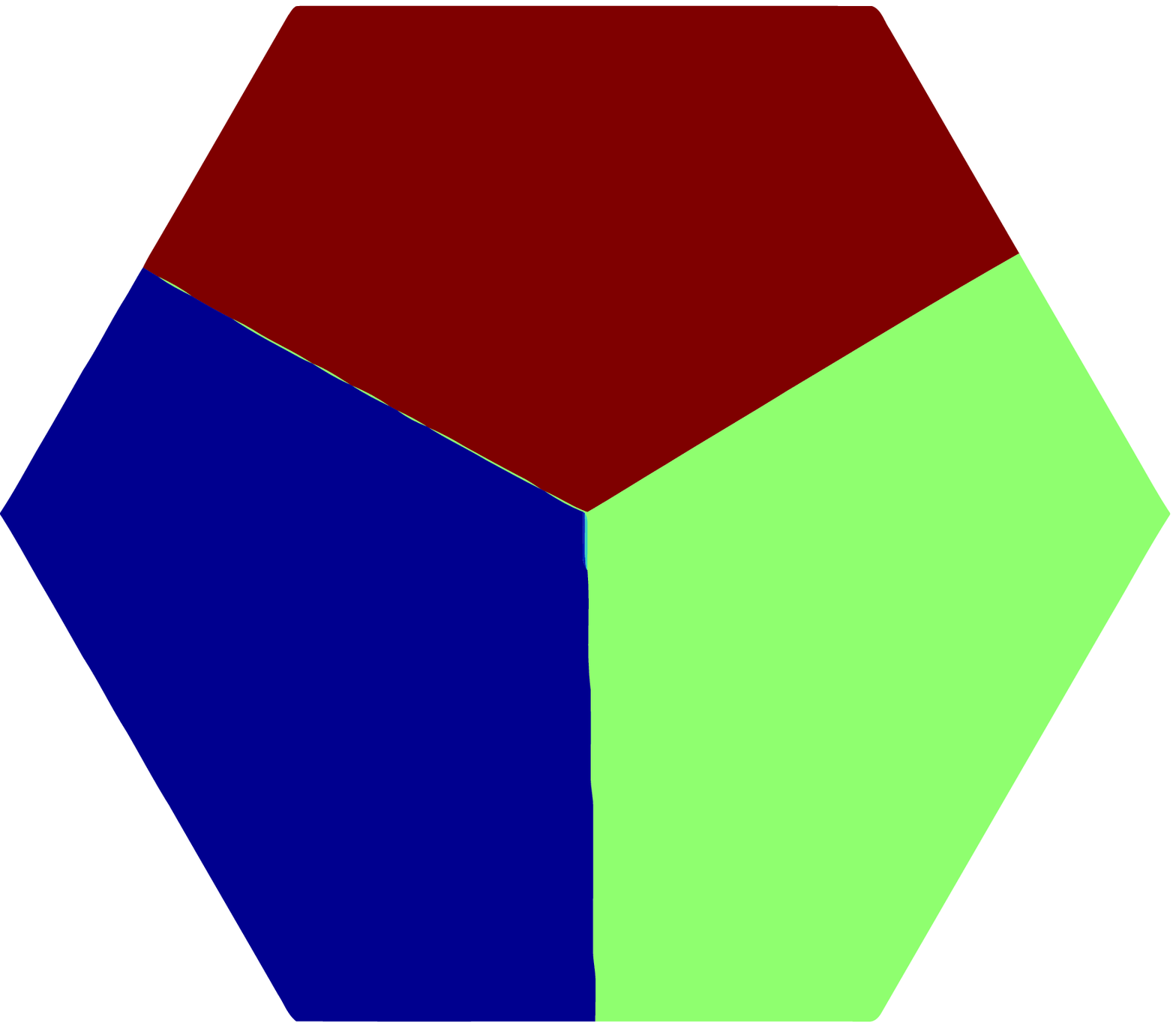}
  ~
  \includegraphics[width=0.15\textwidth]{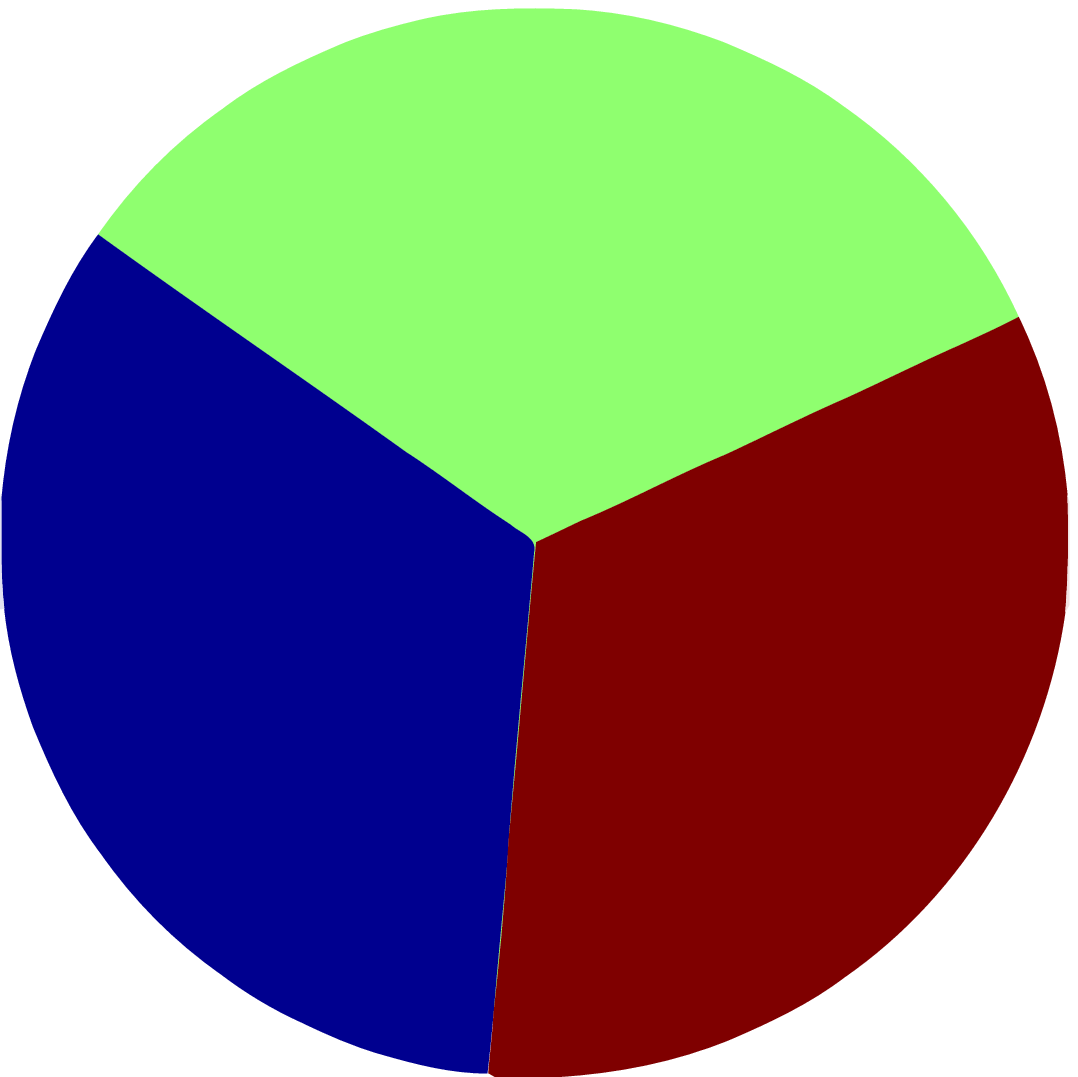}\\
  \includegraphics[width=0.15\textwidth]{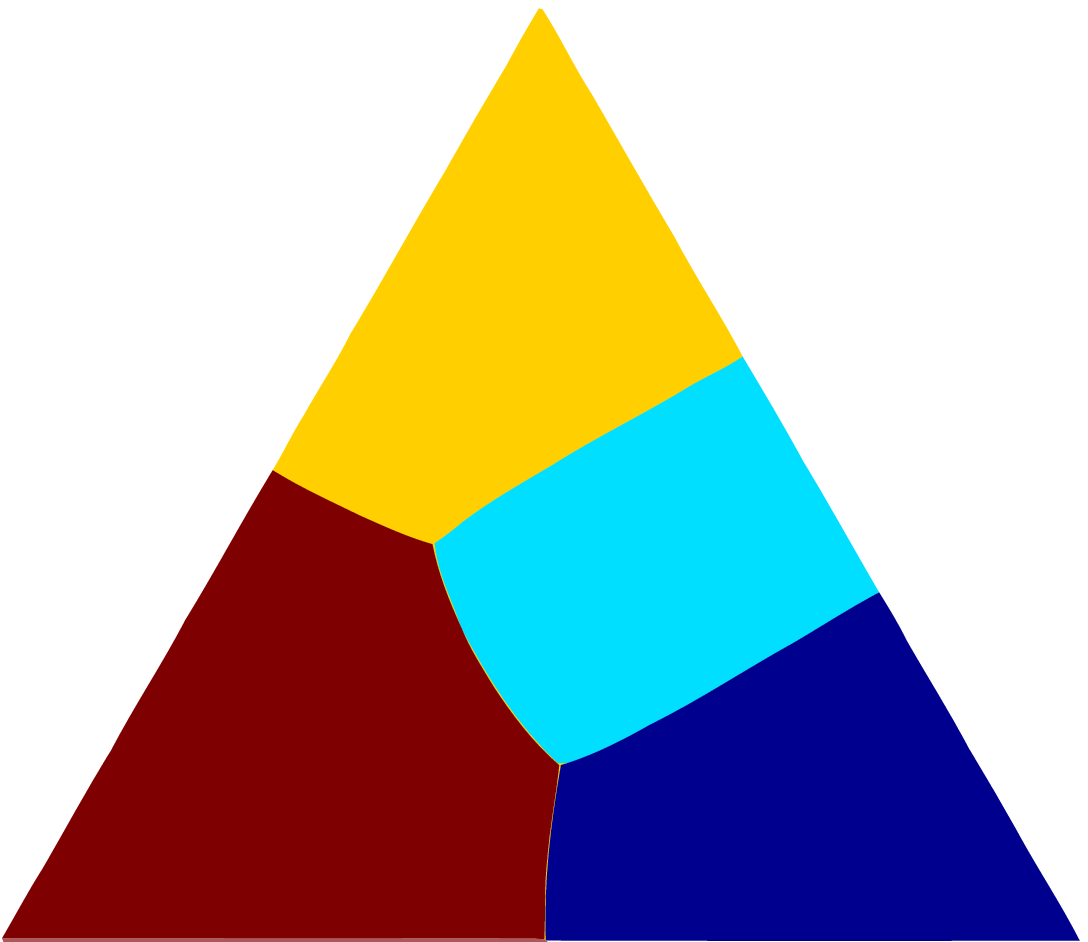}
  ~
  \includegraphics[width=0.15\textwidth]{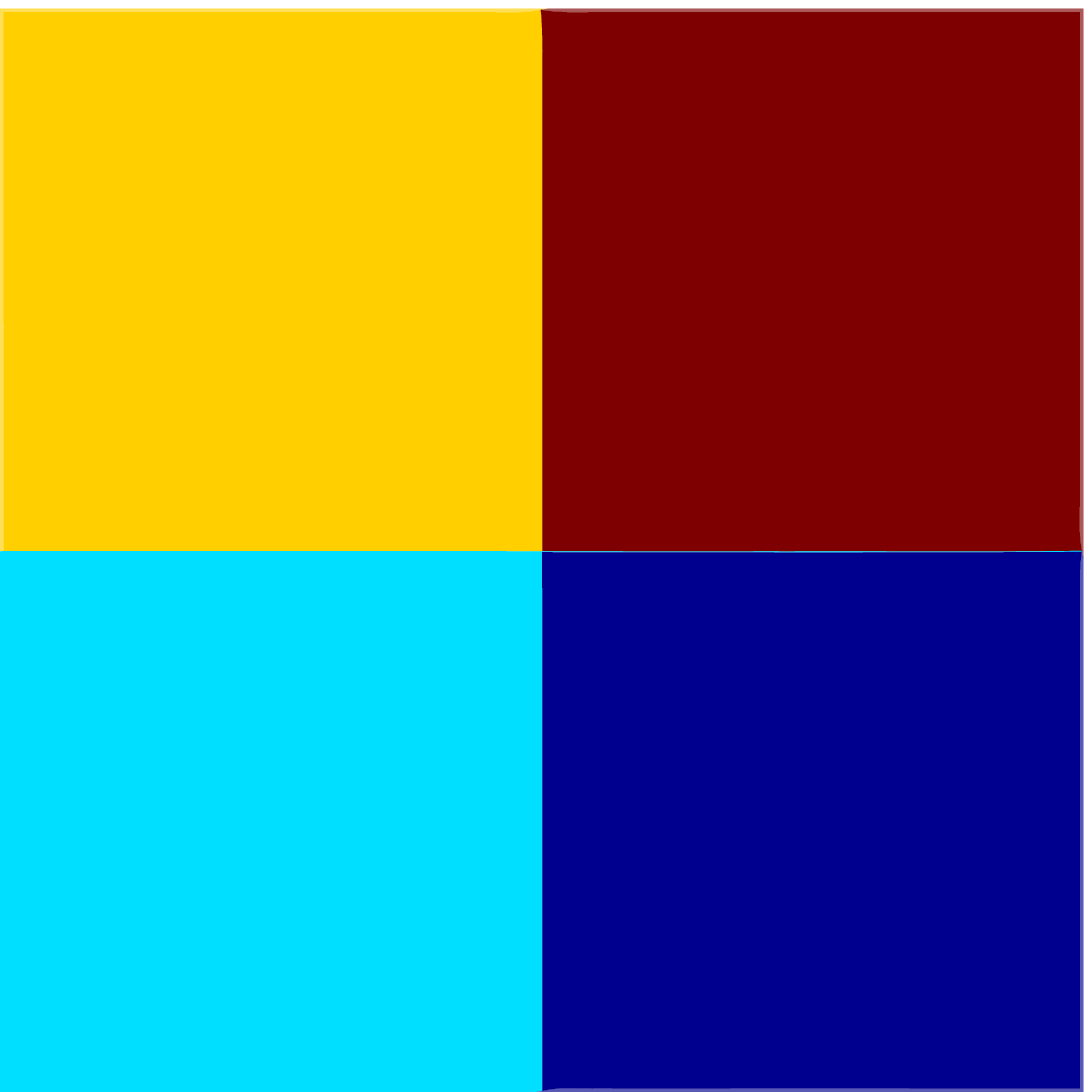} 
  ~
  \includegraphics[width=0.15\textwidth, angle = 90,origin=c]{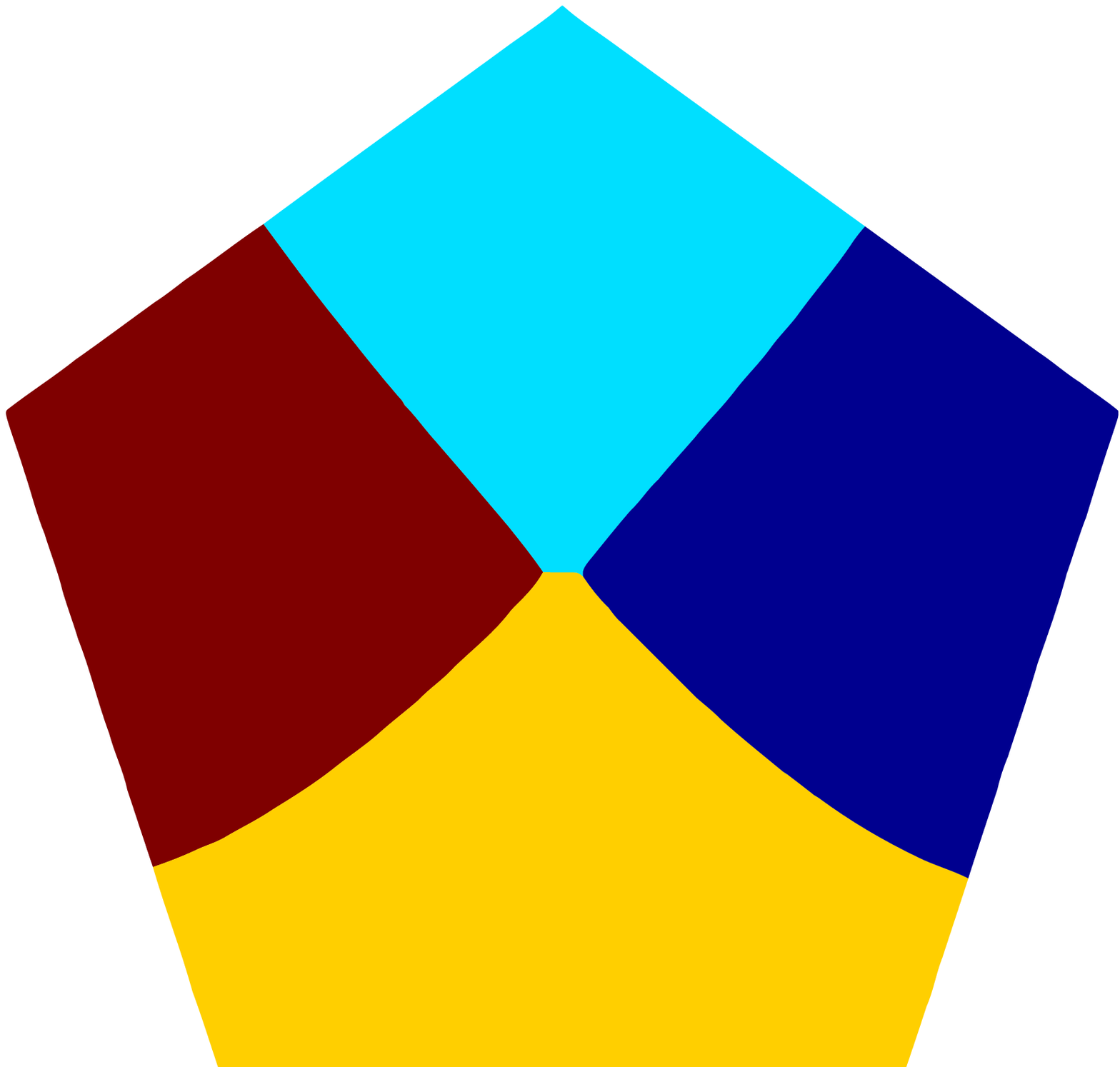}
  ~
  \includegraphics[width=0.15\textwidth]{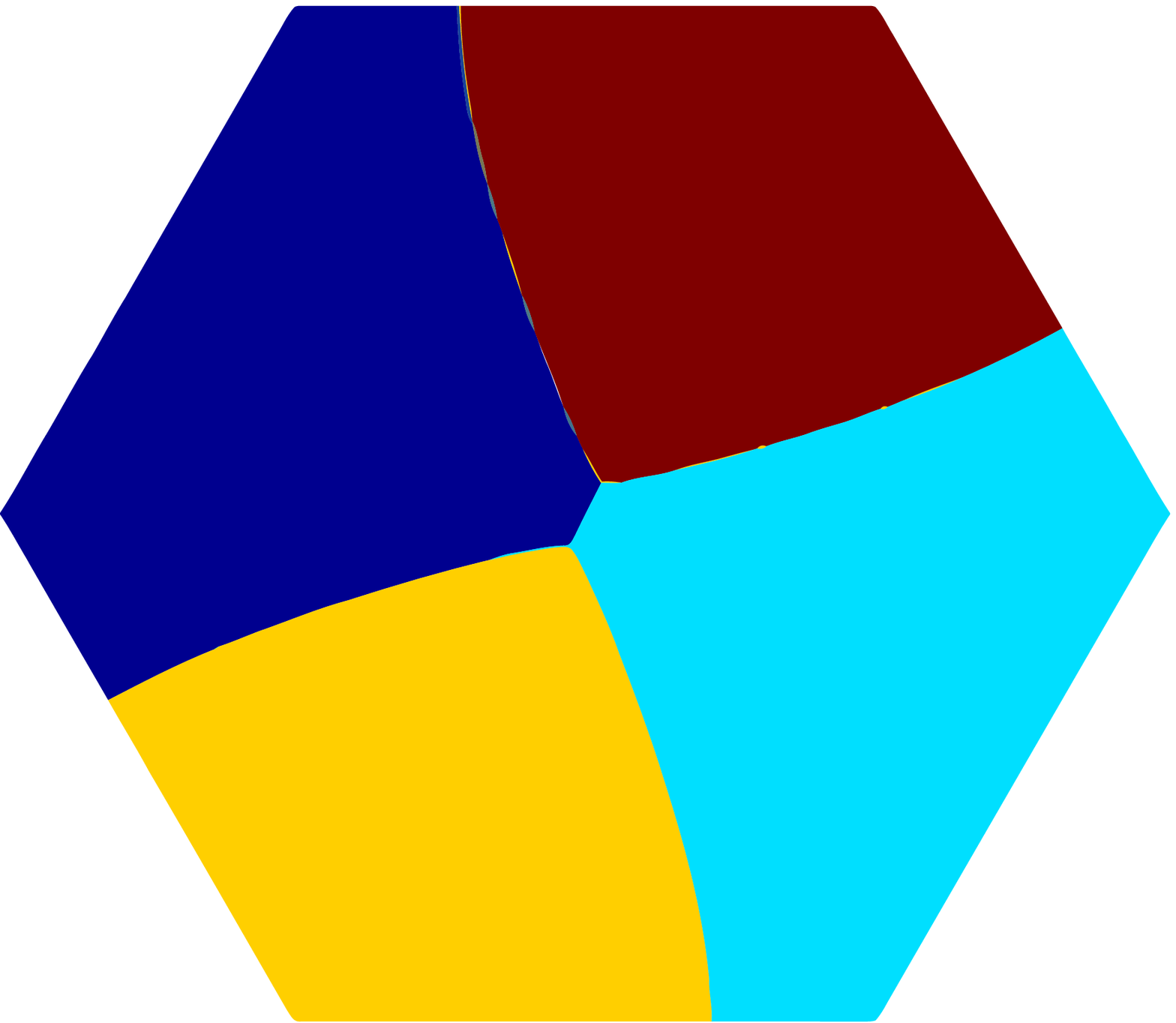}
  ~
  \includegraphics[width=0.15\textwidth]{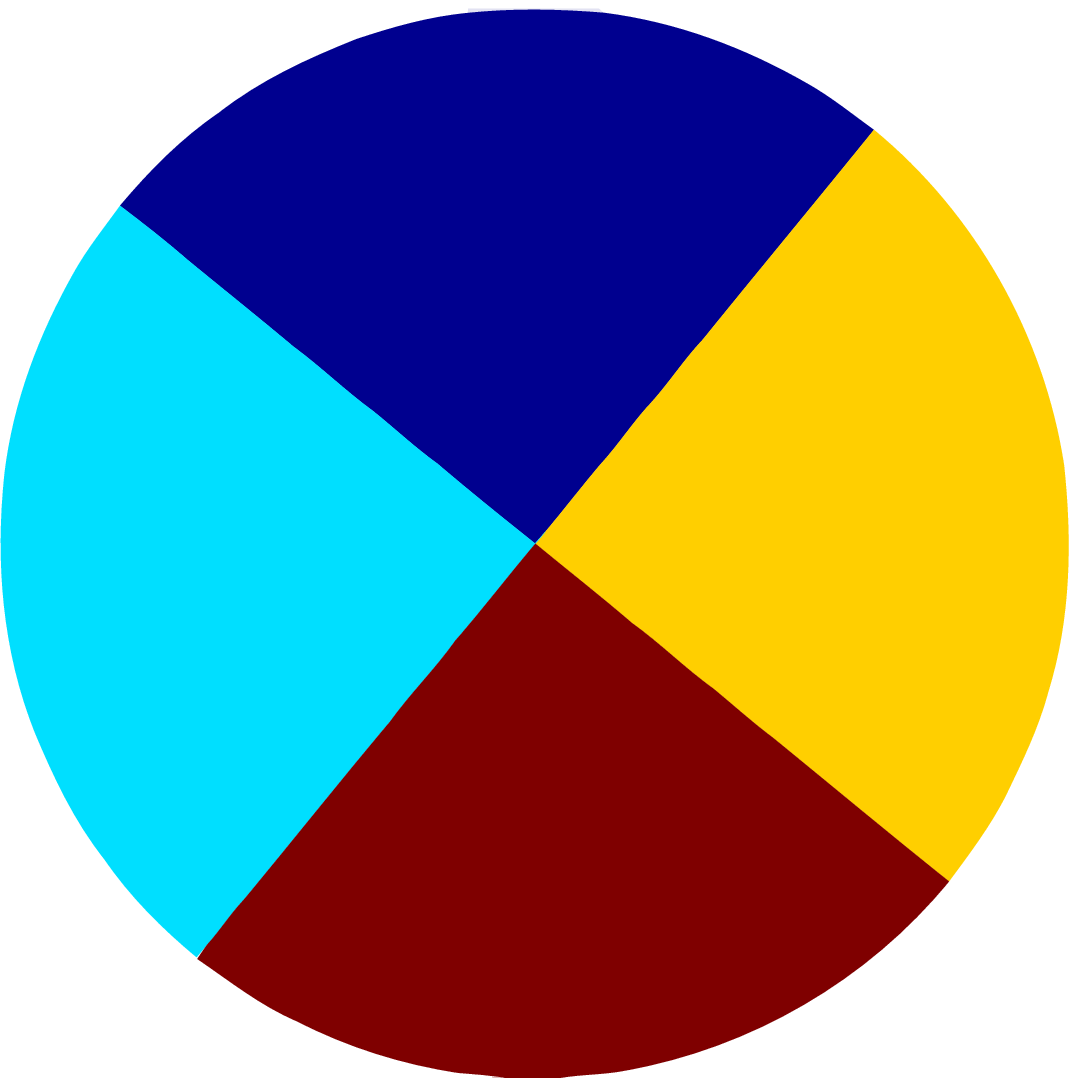}\\
  \includegraphics[width=0.15\textwidth]{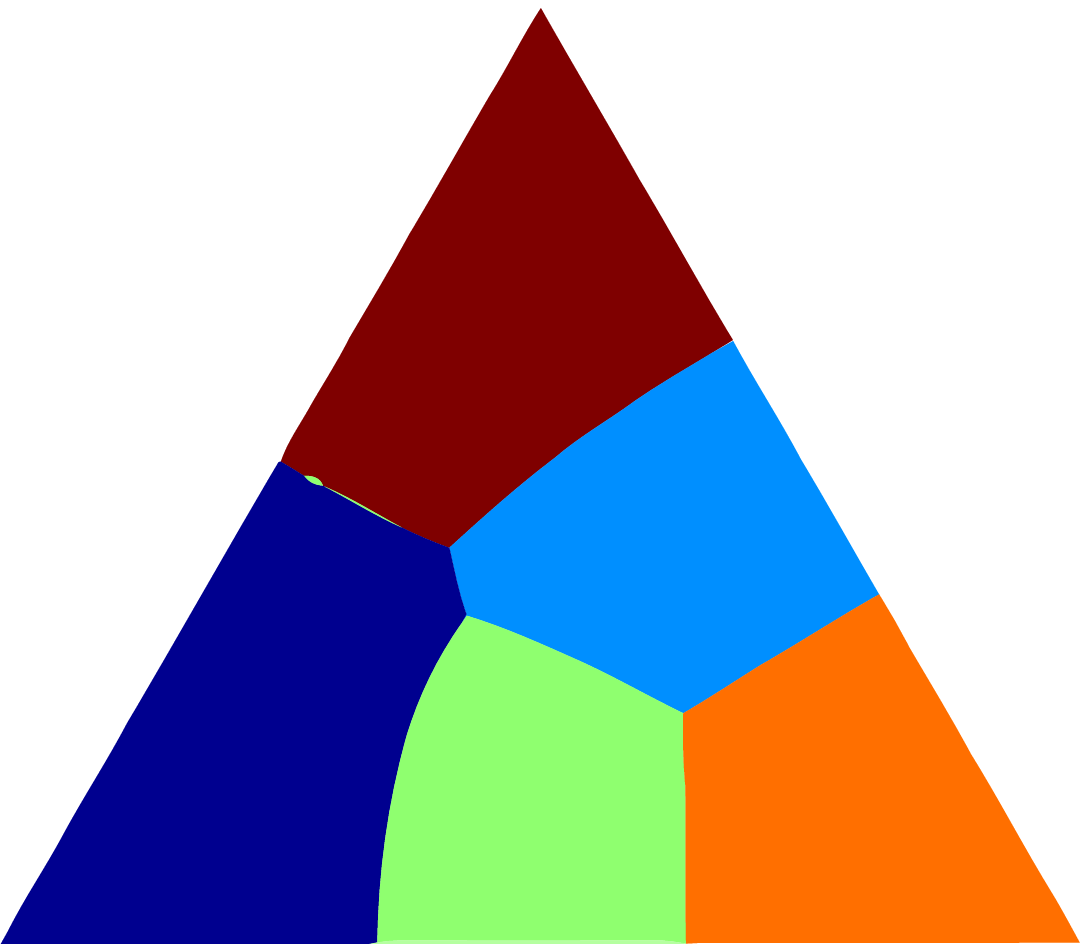}
  ~
  \includegraphics[width=0.15\textwidth]{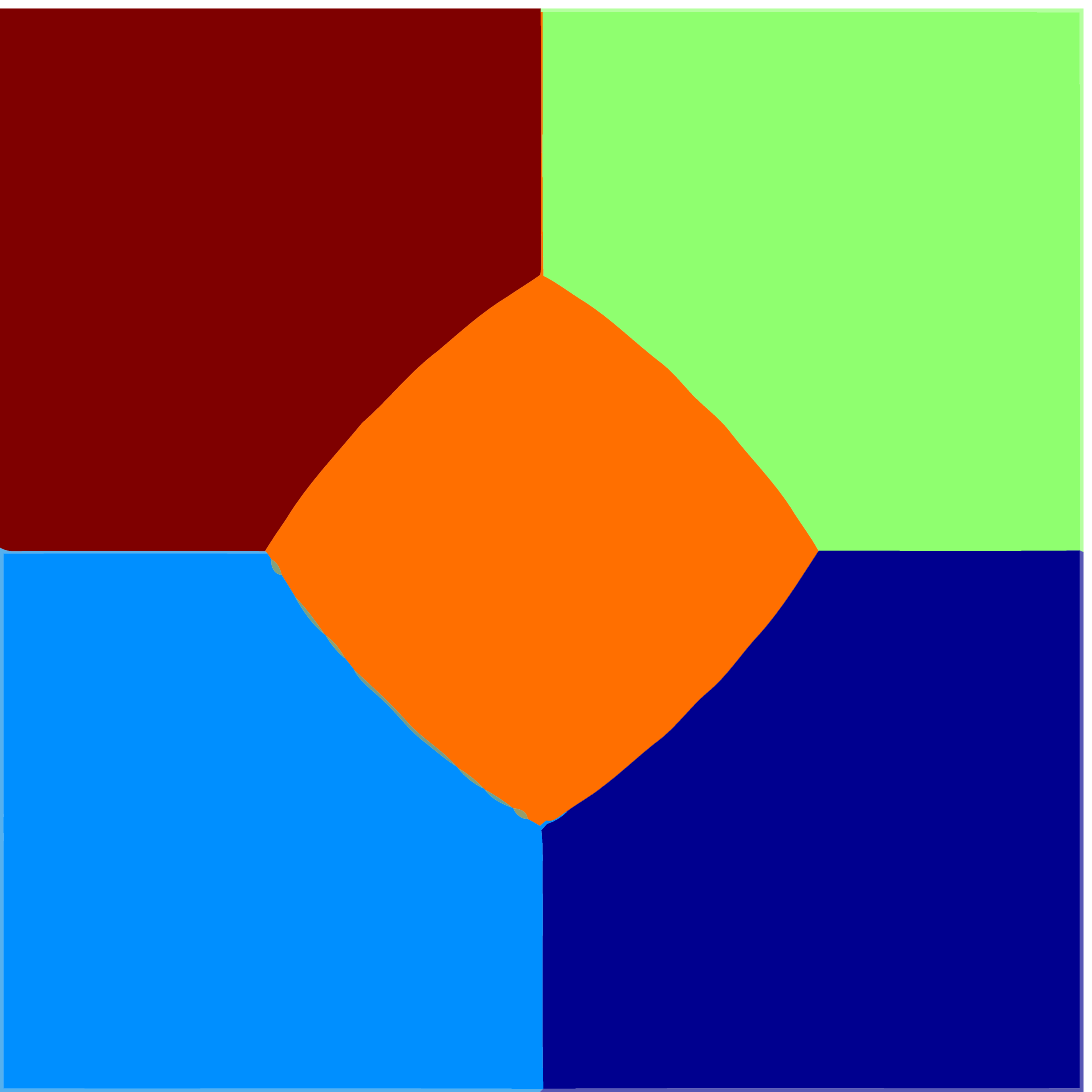} 
  ~
  \includegraphics[width=0.15\textwidth, angle = 90,origin=c]{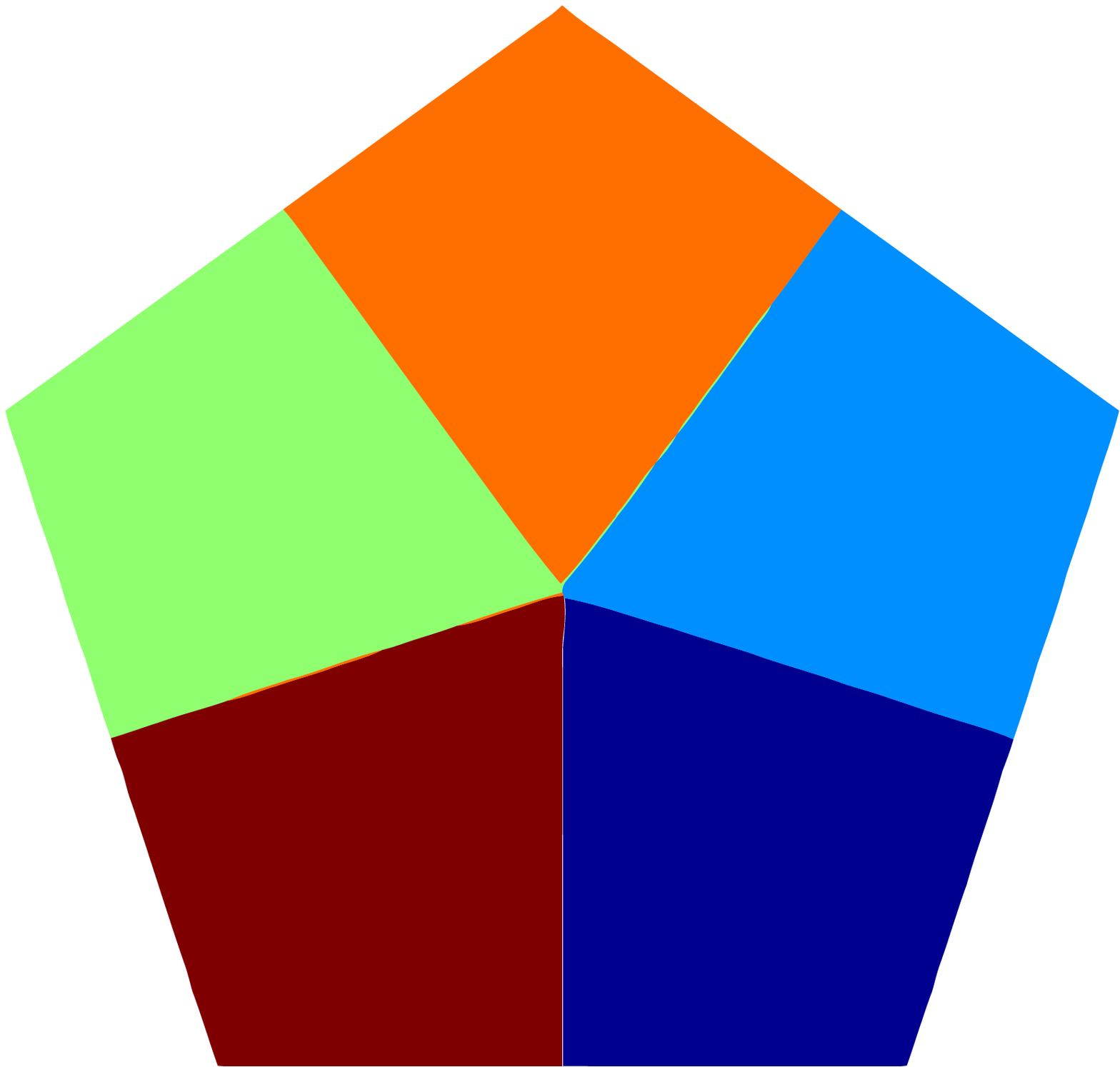}
  ~
  \includegraphics[width=0.15\textwidth]{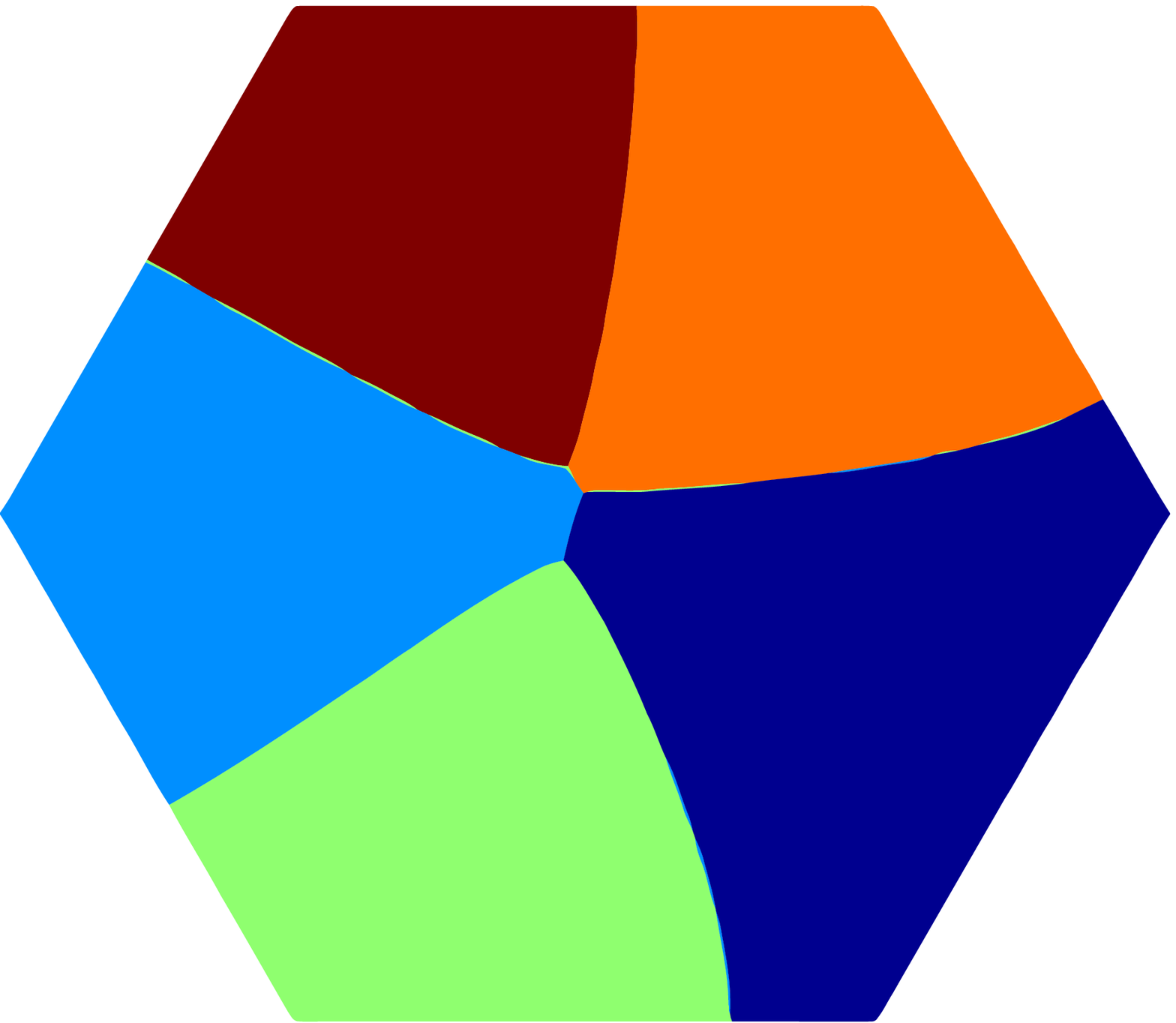}
  ~
  \includegraphics[width=0.15\textwidth]{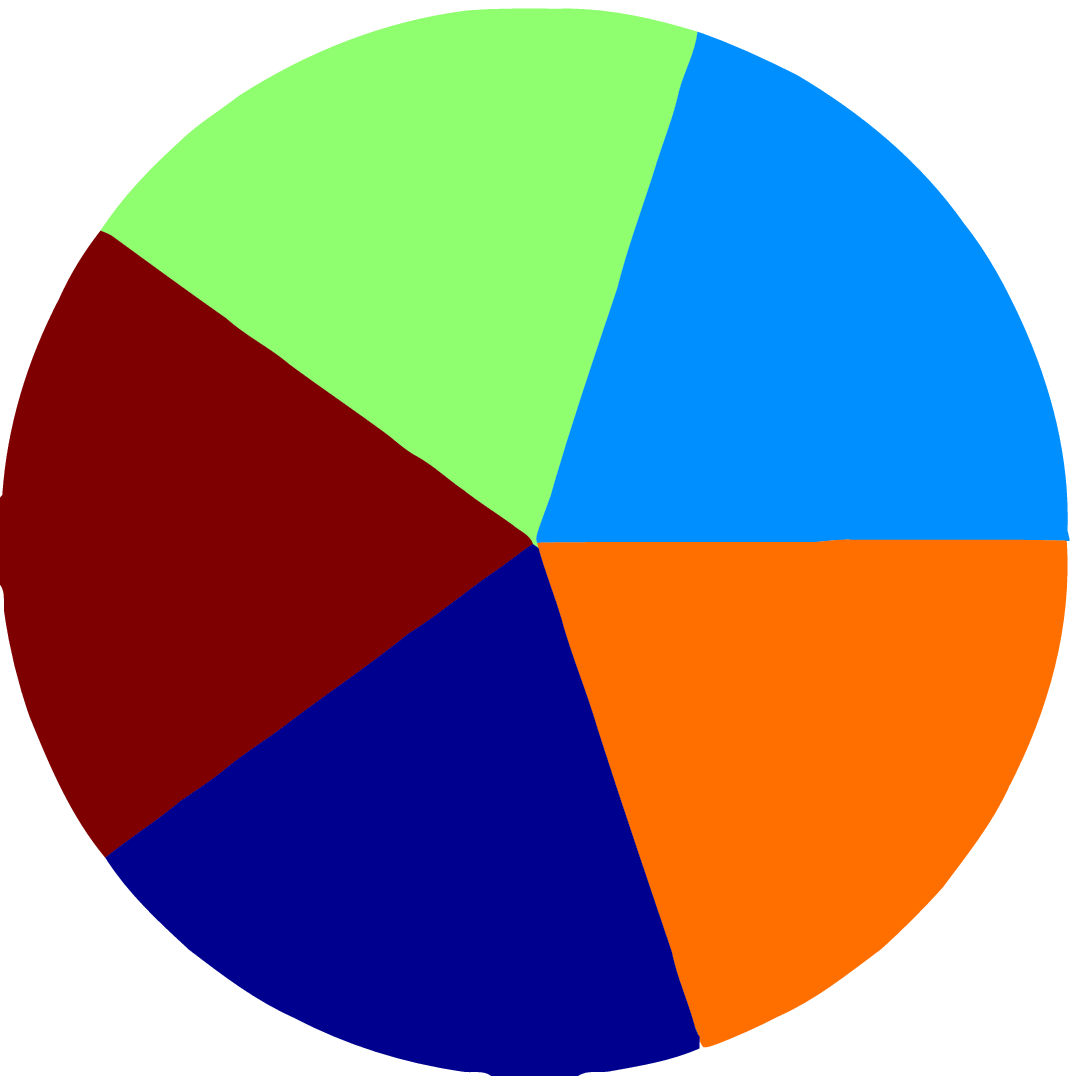}\\
  \includegraphics[width=0.15\textwidth]{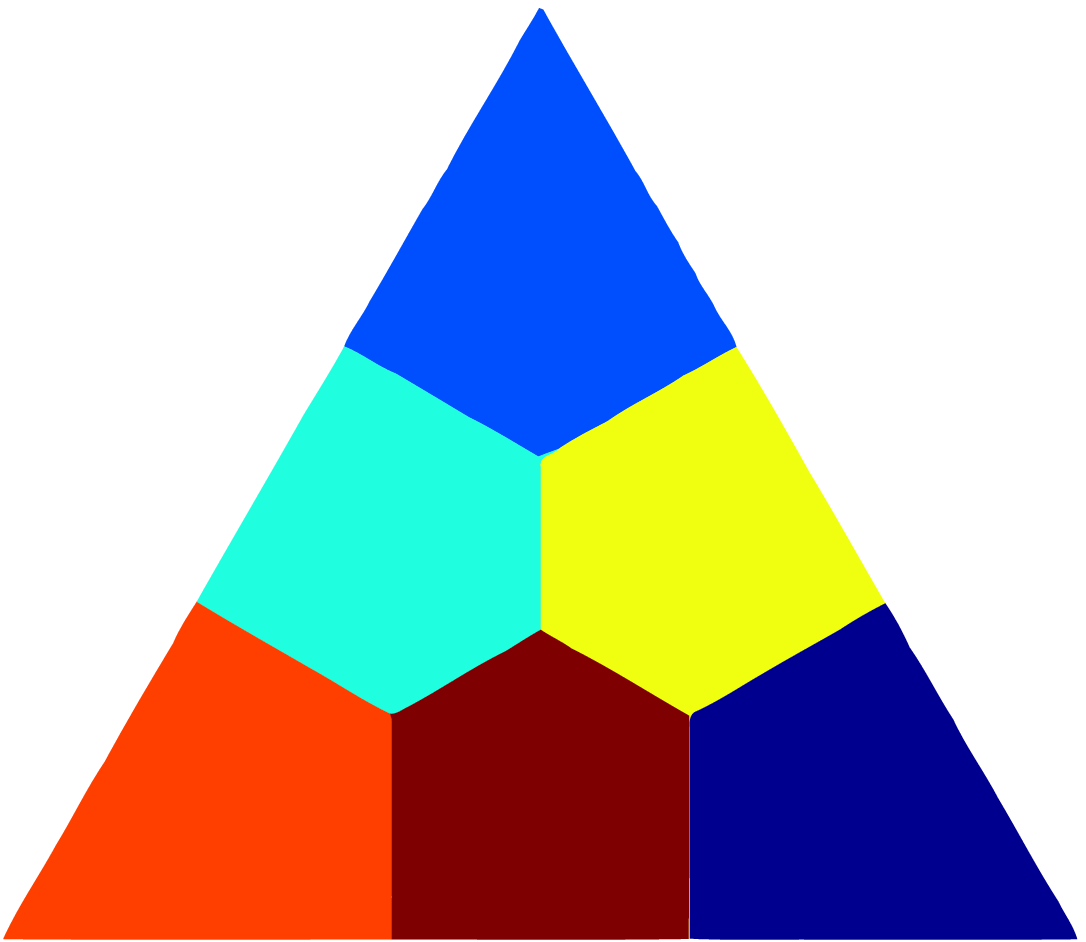}
  ~
  \includegraphics[width=0.15\textwidth]{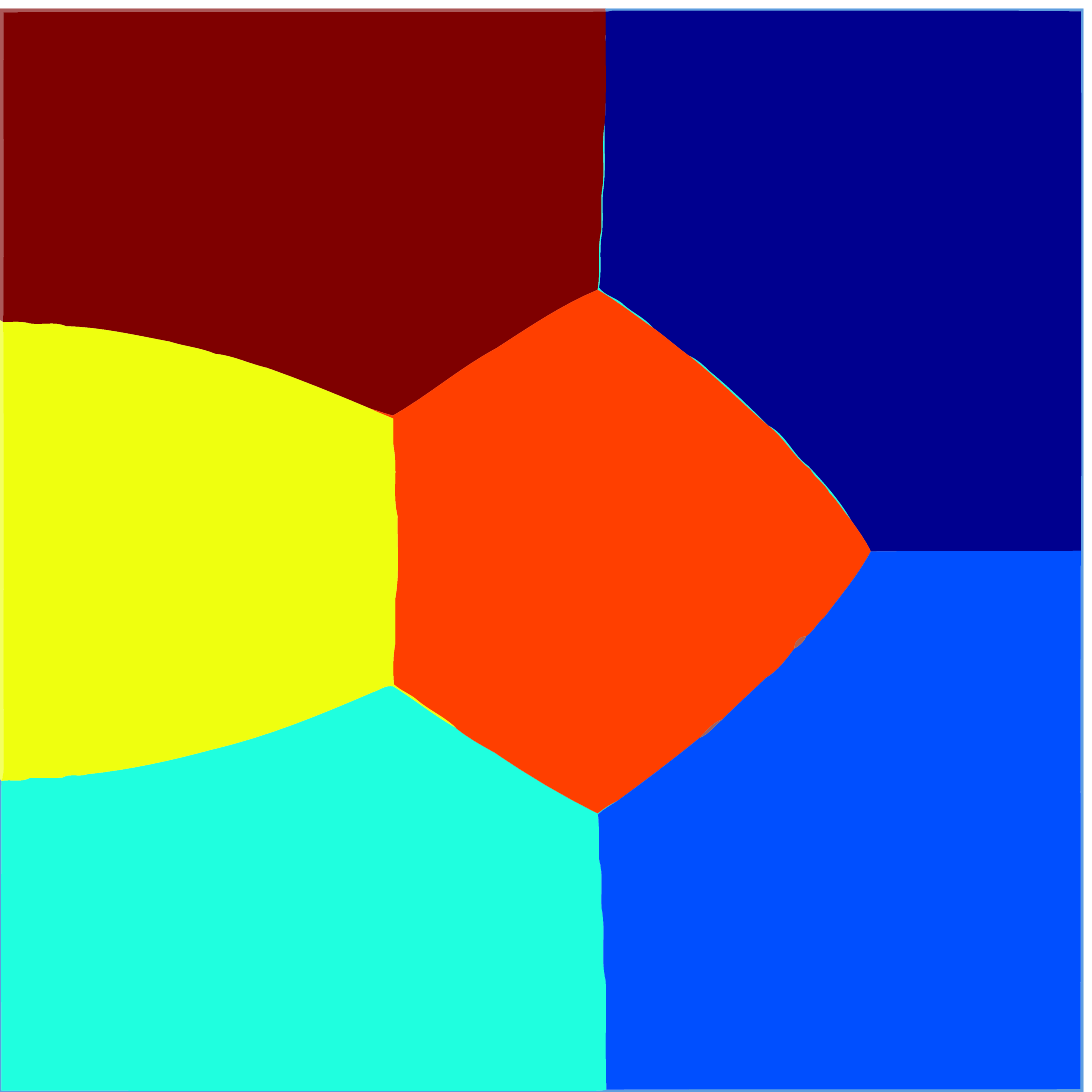} 
  ~
  \includegraphics[width=0.15\textwidth, angle = 90,origin=c]{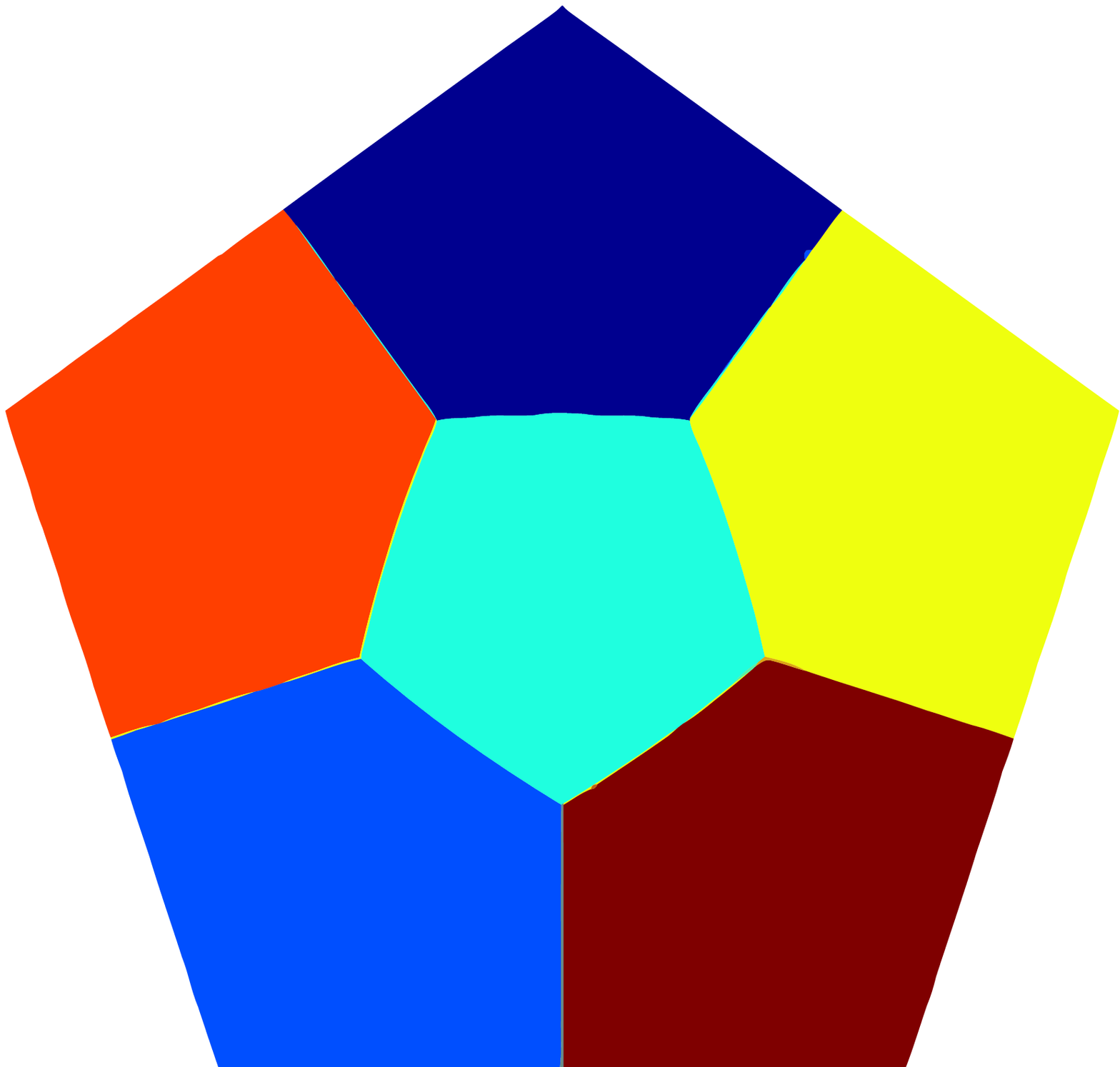}
  ~
  \includegraphics[width=0.15\textwidth]{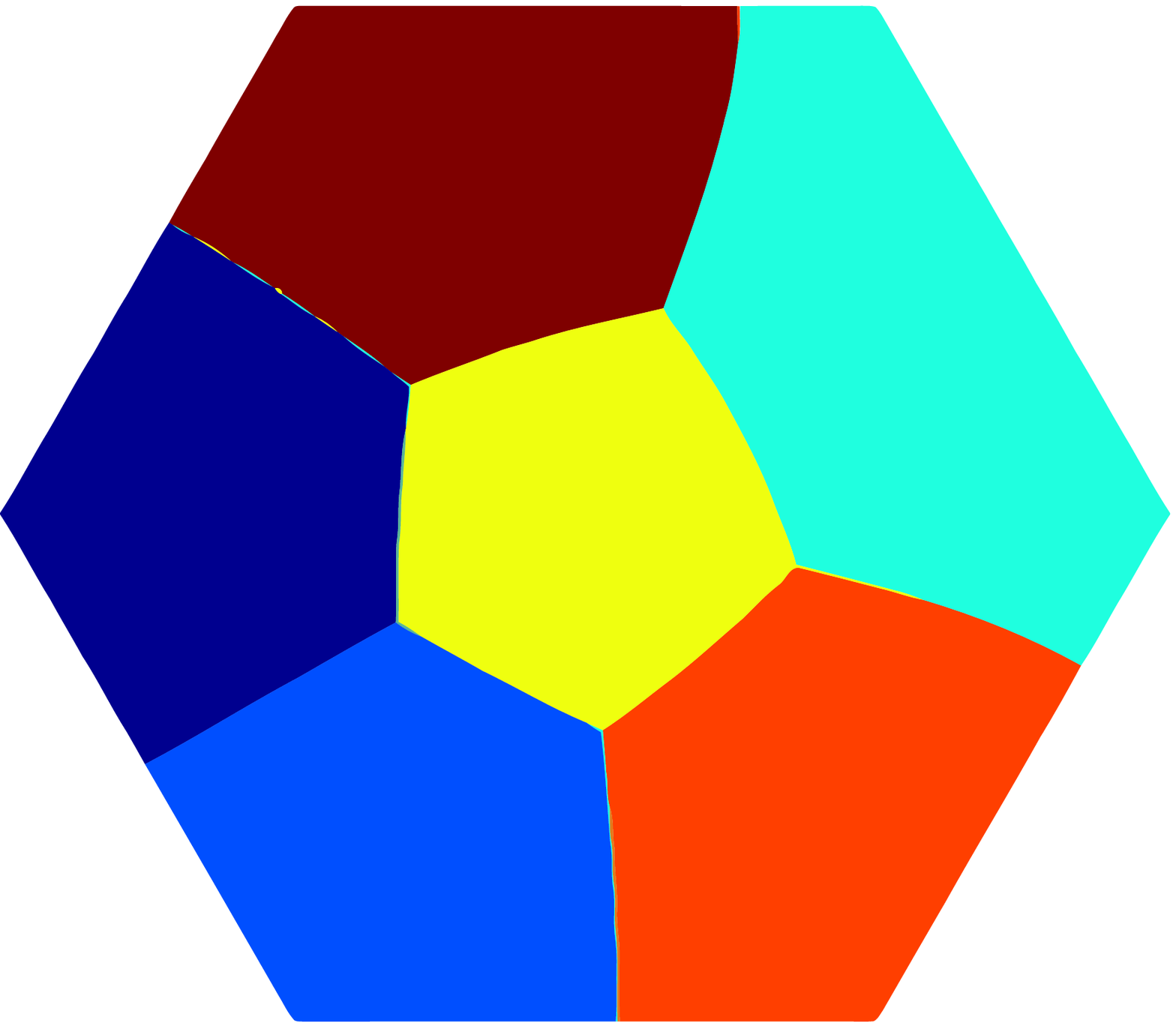}
  ~
  \includegraphics[width=0.15\textwidth]{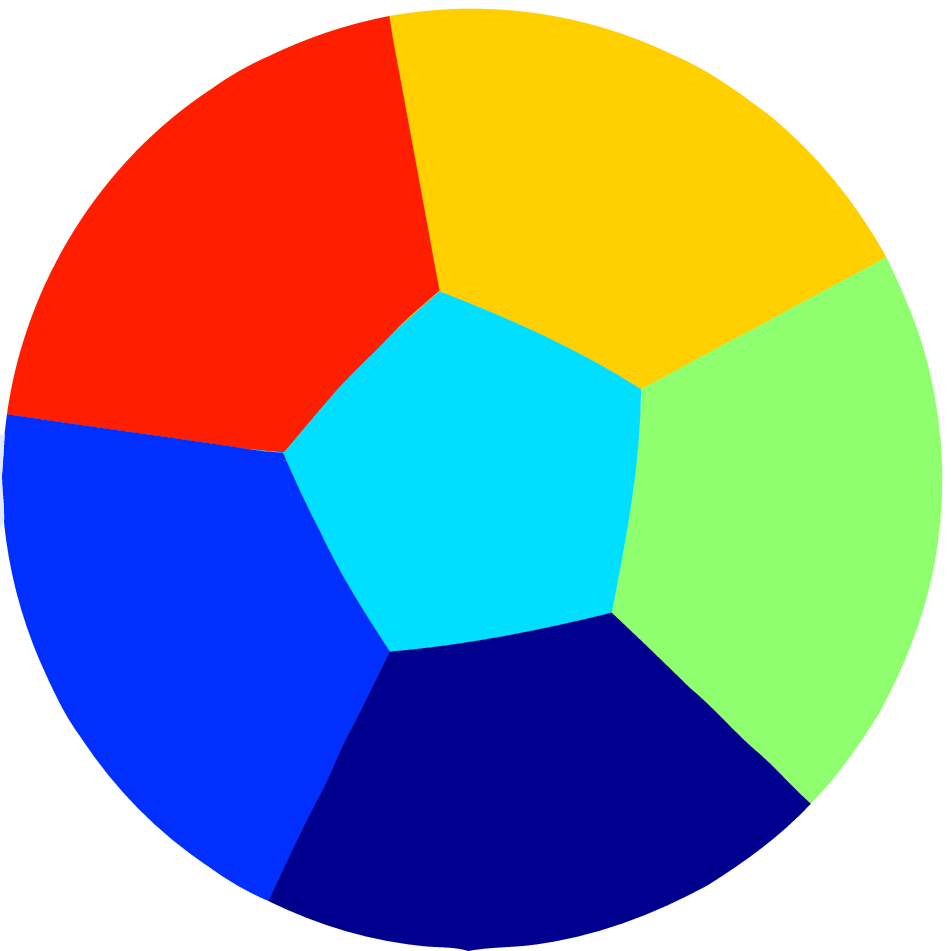}\\
  \includegraphics[width=0.15\textwidth]{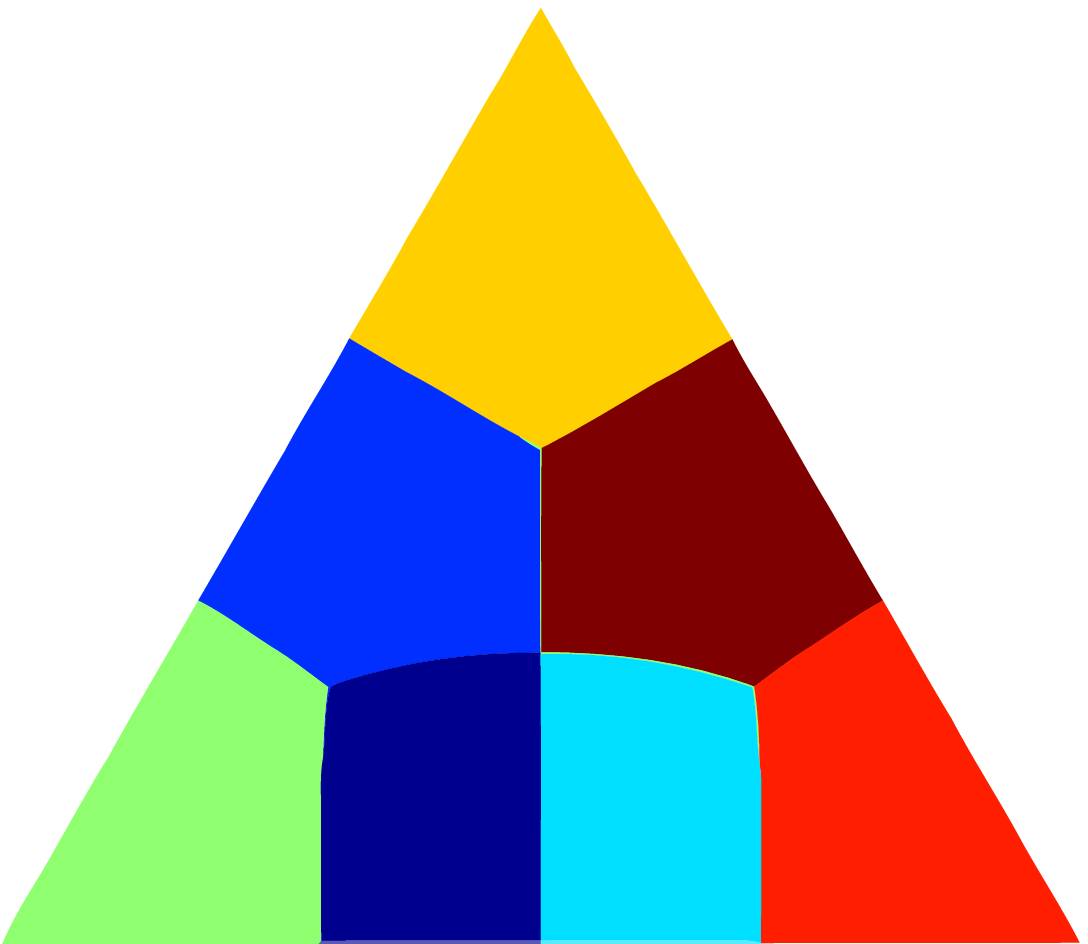}
  ~
  \includegraphics[width=0.15\textwidth]{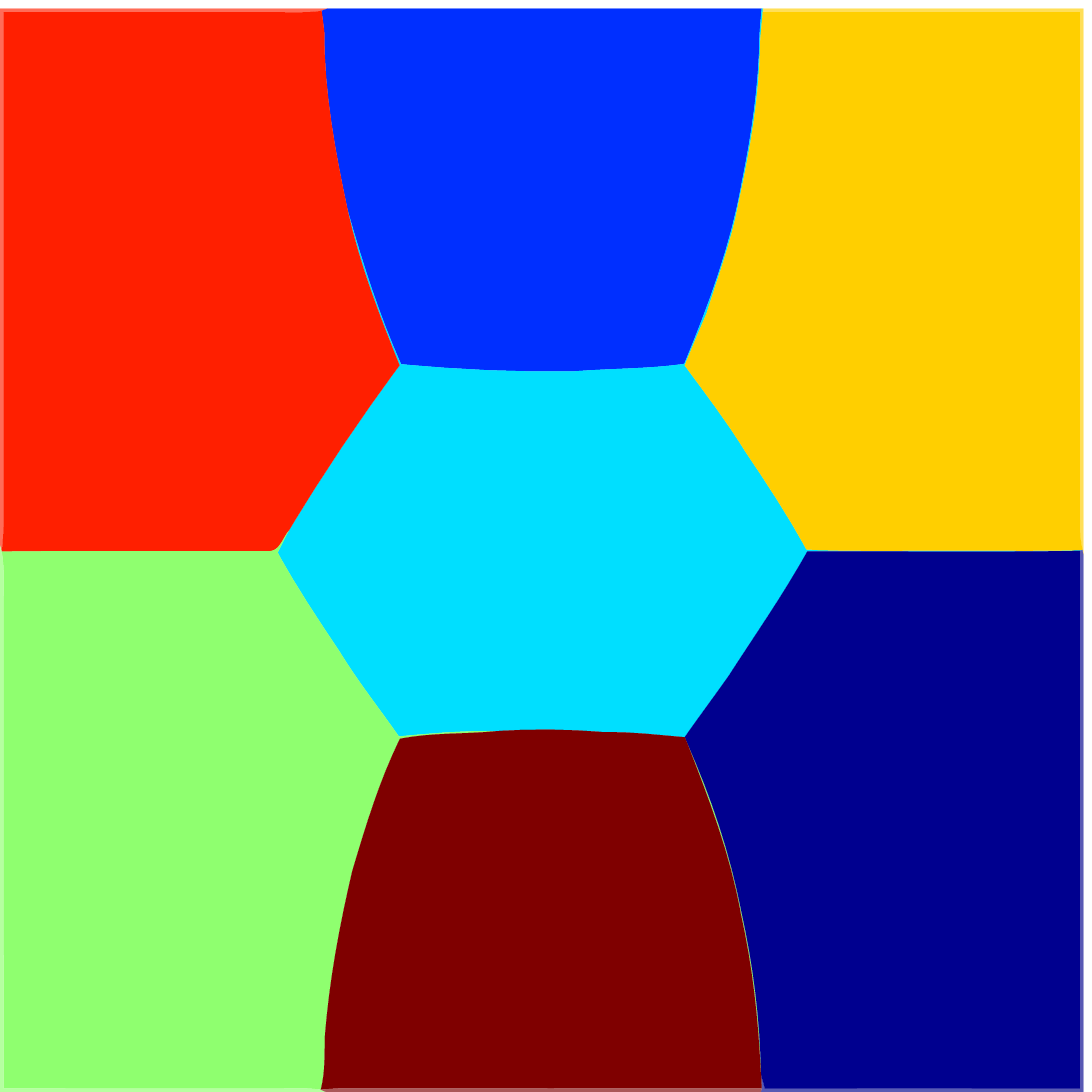} 
  ~
  \includegraphics[width=0.15\textwidth, angle = 90,origin=c]{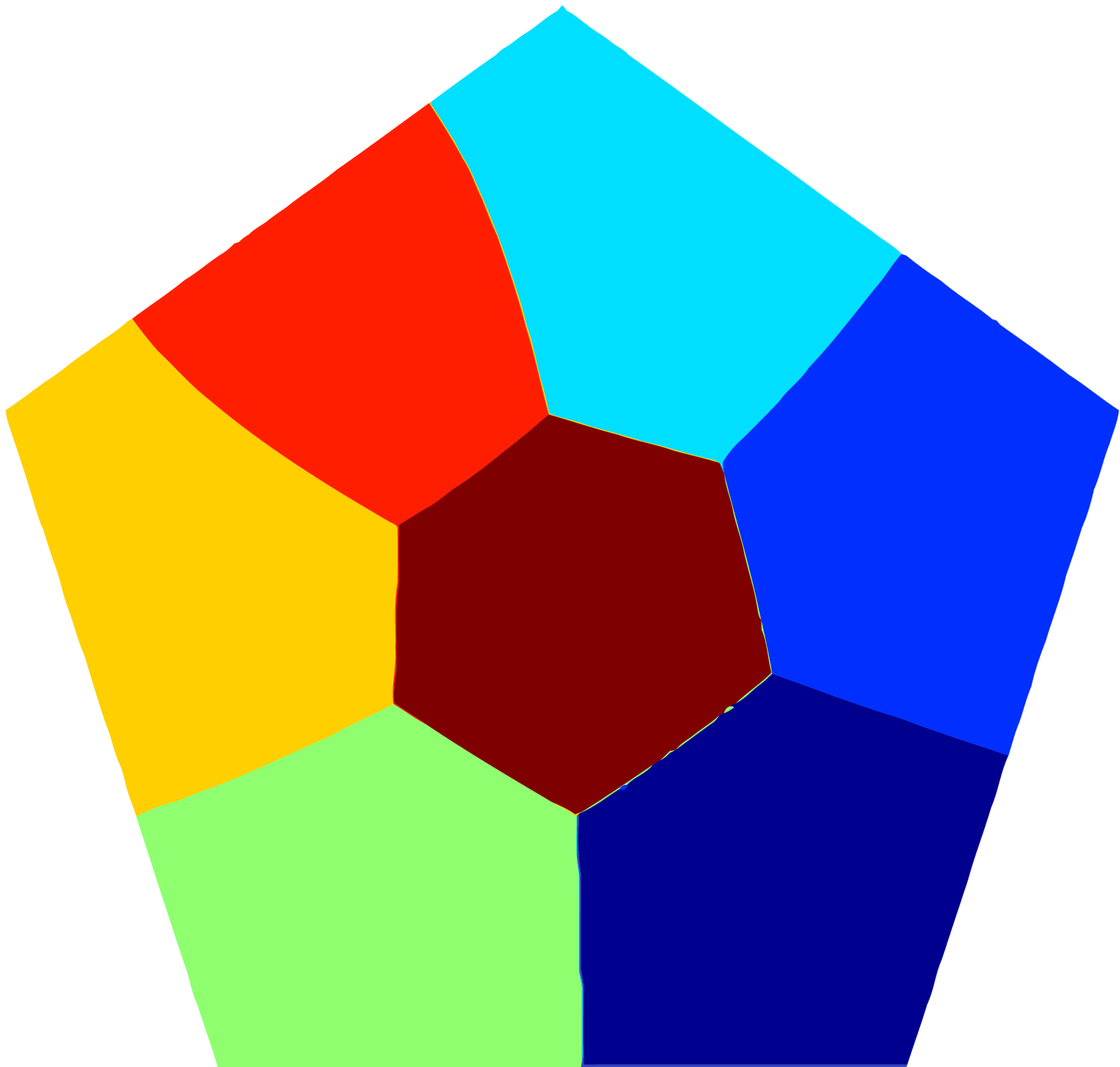}
  ~
  \includegraphics[width=0.15\textwidth]{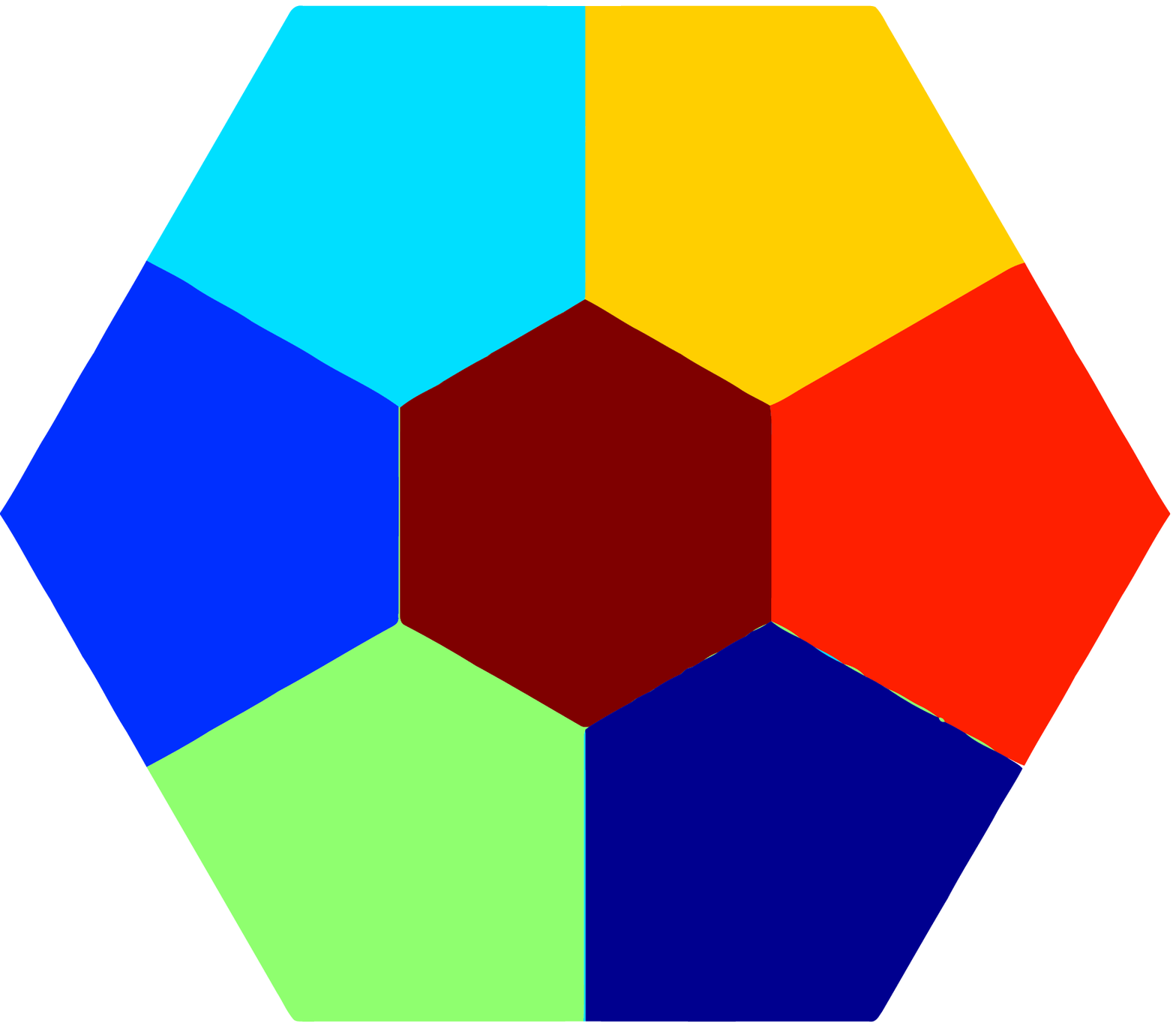}
  ~
  \includegraphics[width=0.15\textwidth]{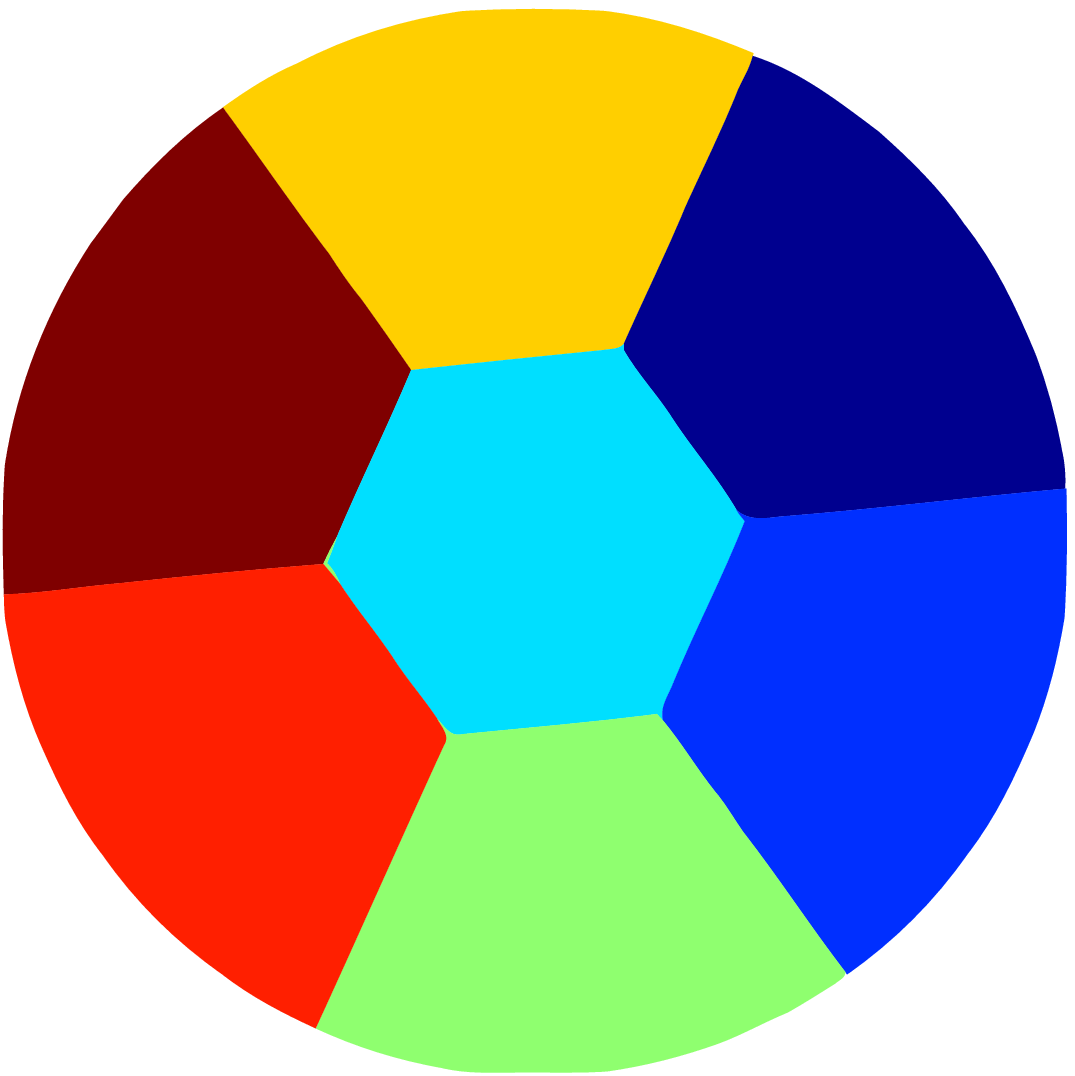}\\
  \includegraphics[width=0.15\textwidth]{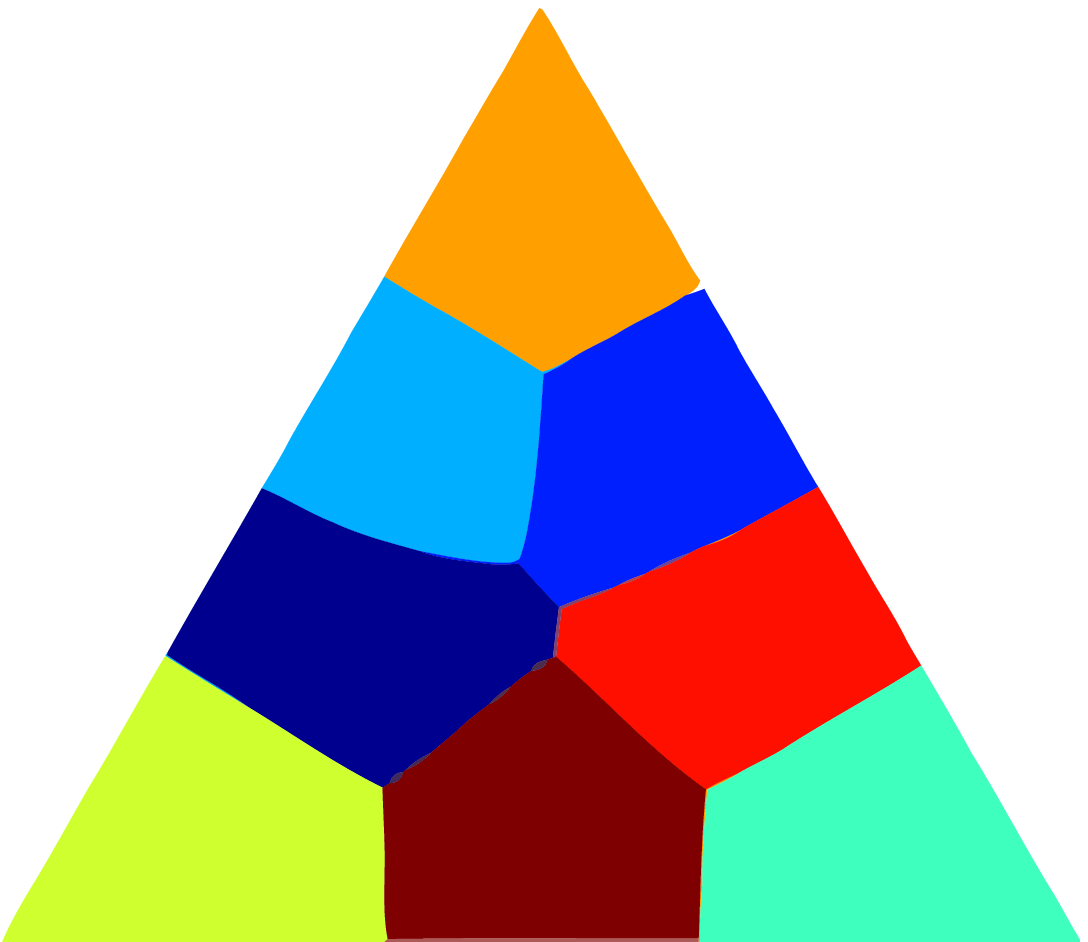}
  ~
  \includegraphics[width=0.15\textwidth]{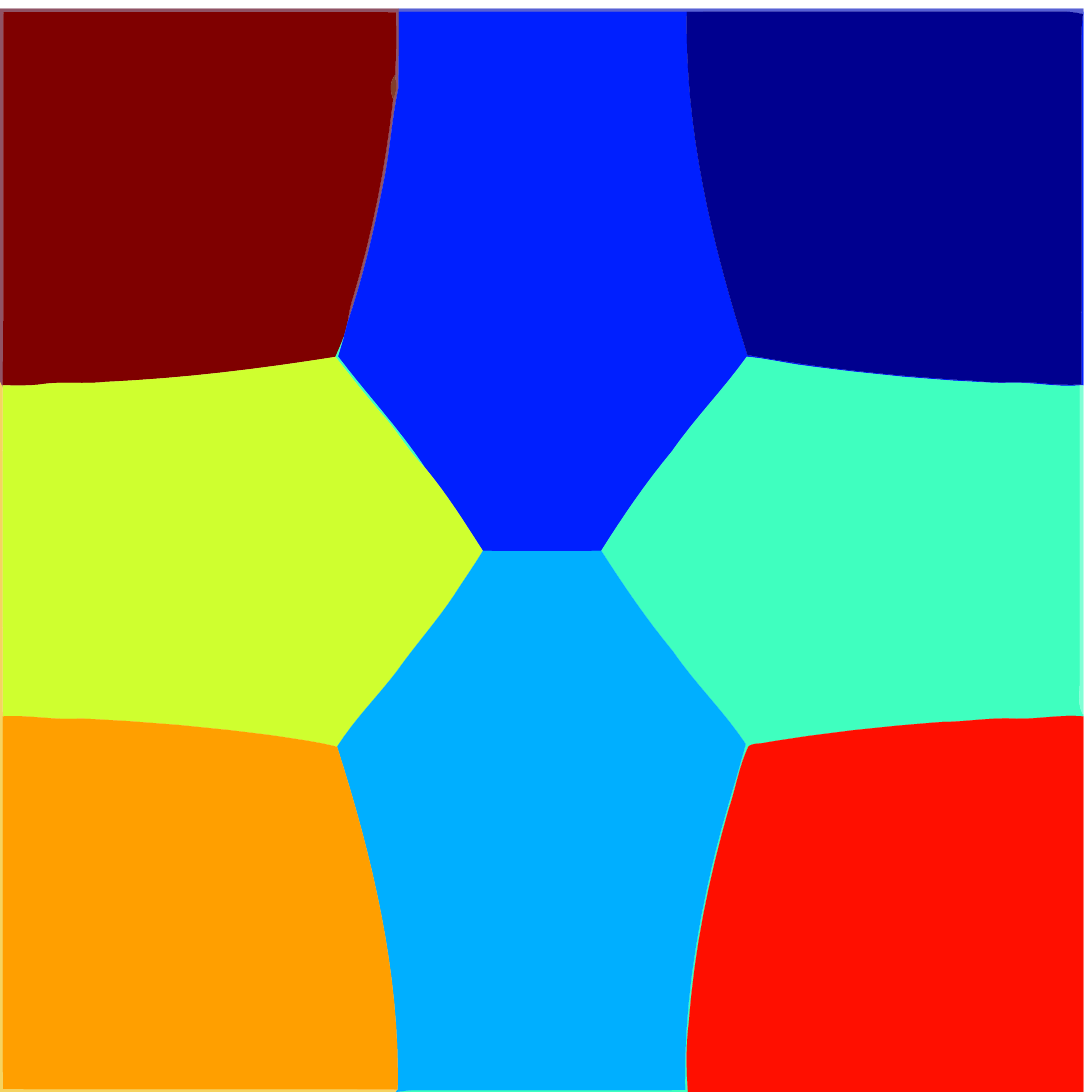} 
  ~
  \includegraphics[width=0.15\textwidth, angle = 90,origin=c]{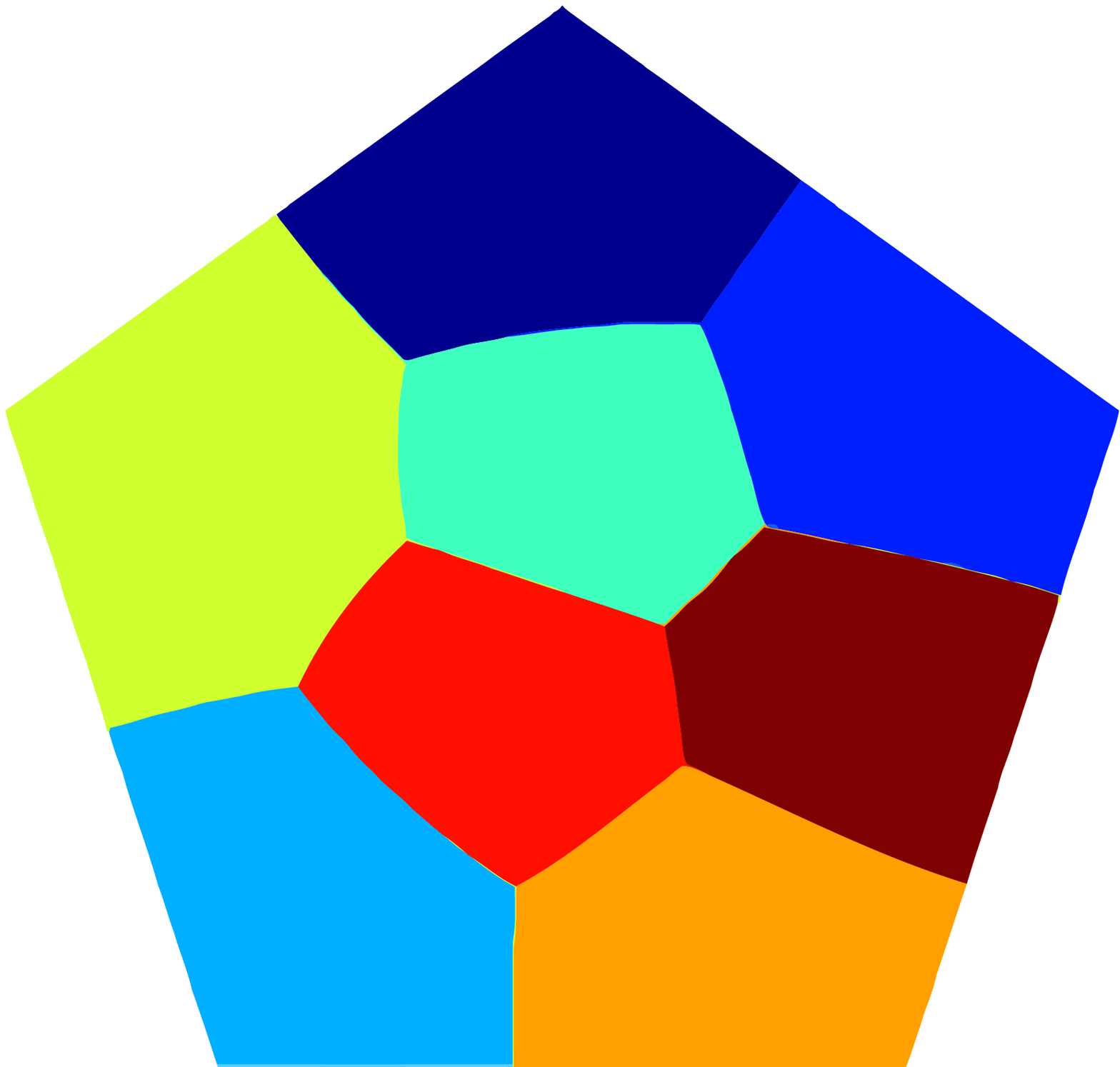}
  ~
  \includegraphics[width=0.15\textwidth]{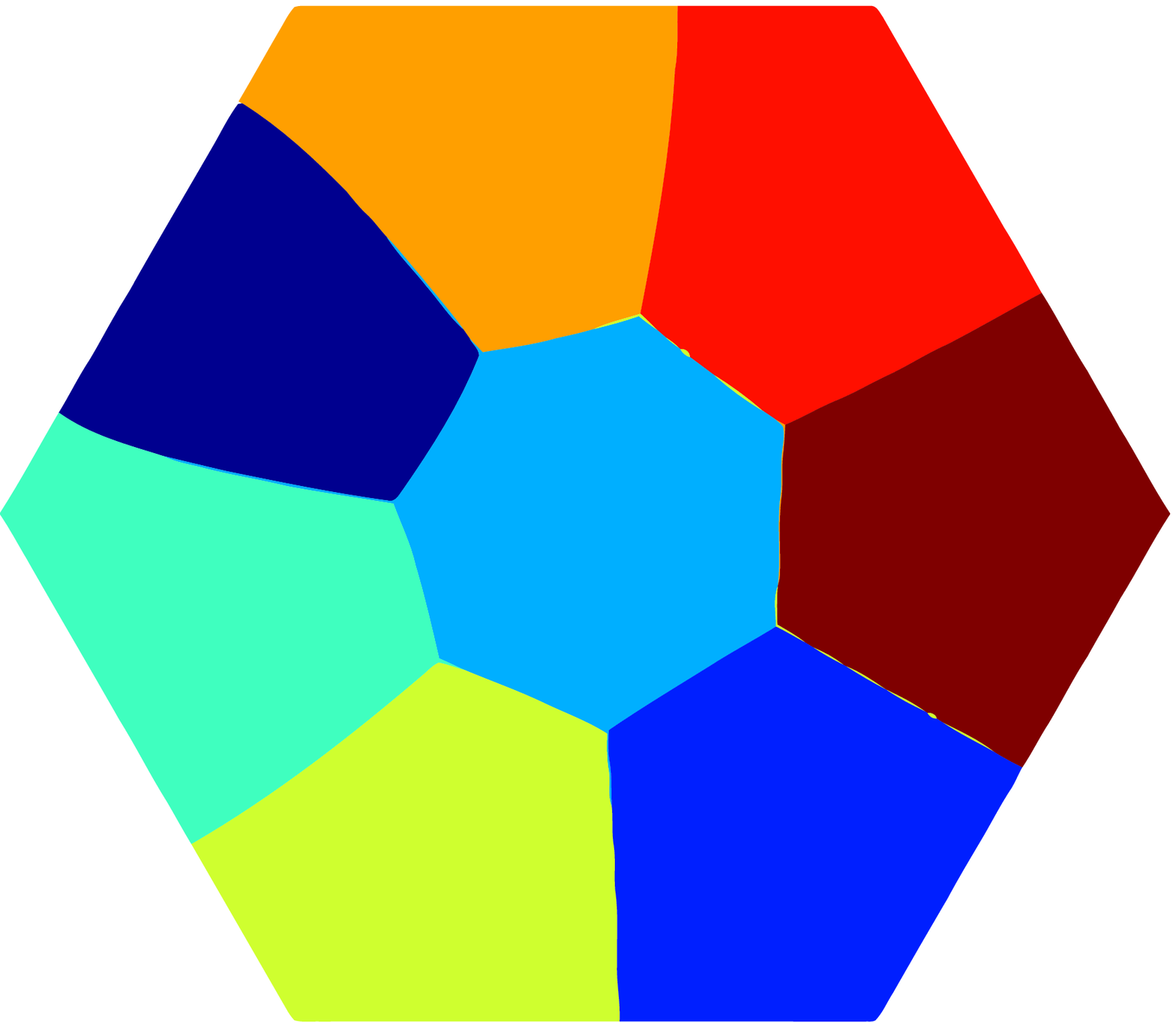}
  ~
  \includegraphics[width=0.15\textwidth]{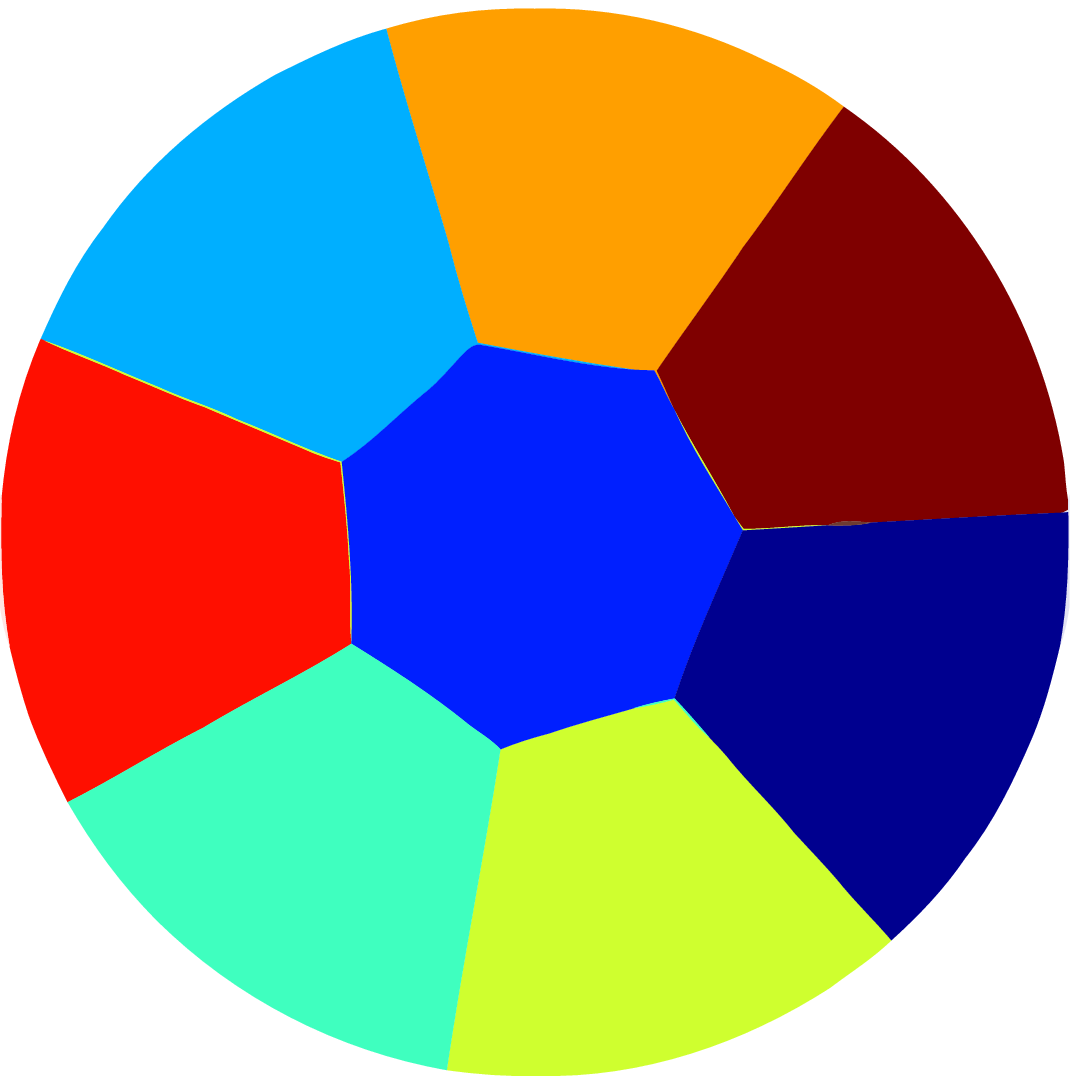}\\
  \includegraphics[width=0.15\textwidth]{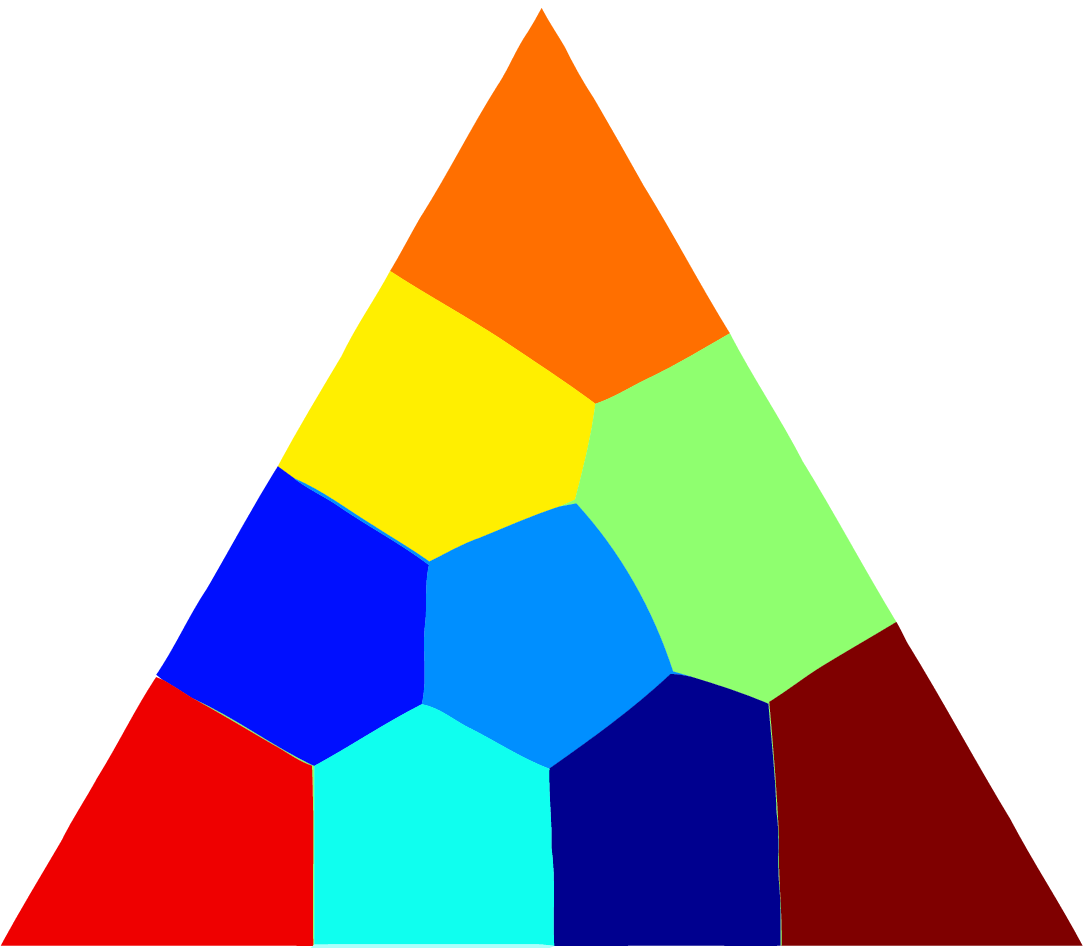}
  ~
  \includegraphics[width=0.15\textwidth]{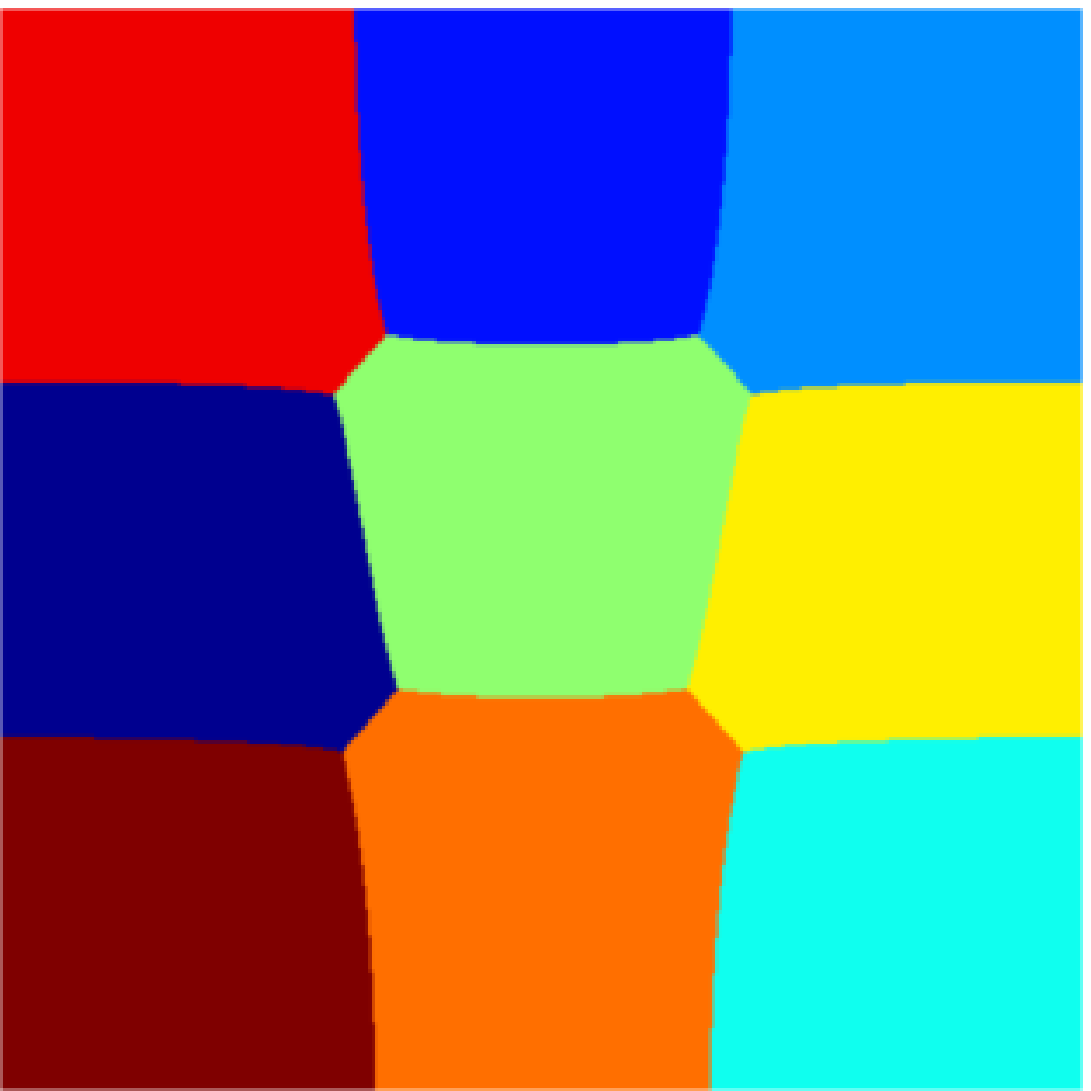} 
  ~
  \includegraphics[width=0.15\textwidth, angle = 90,origin=c]{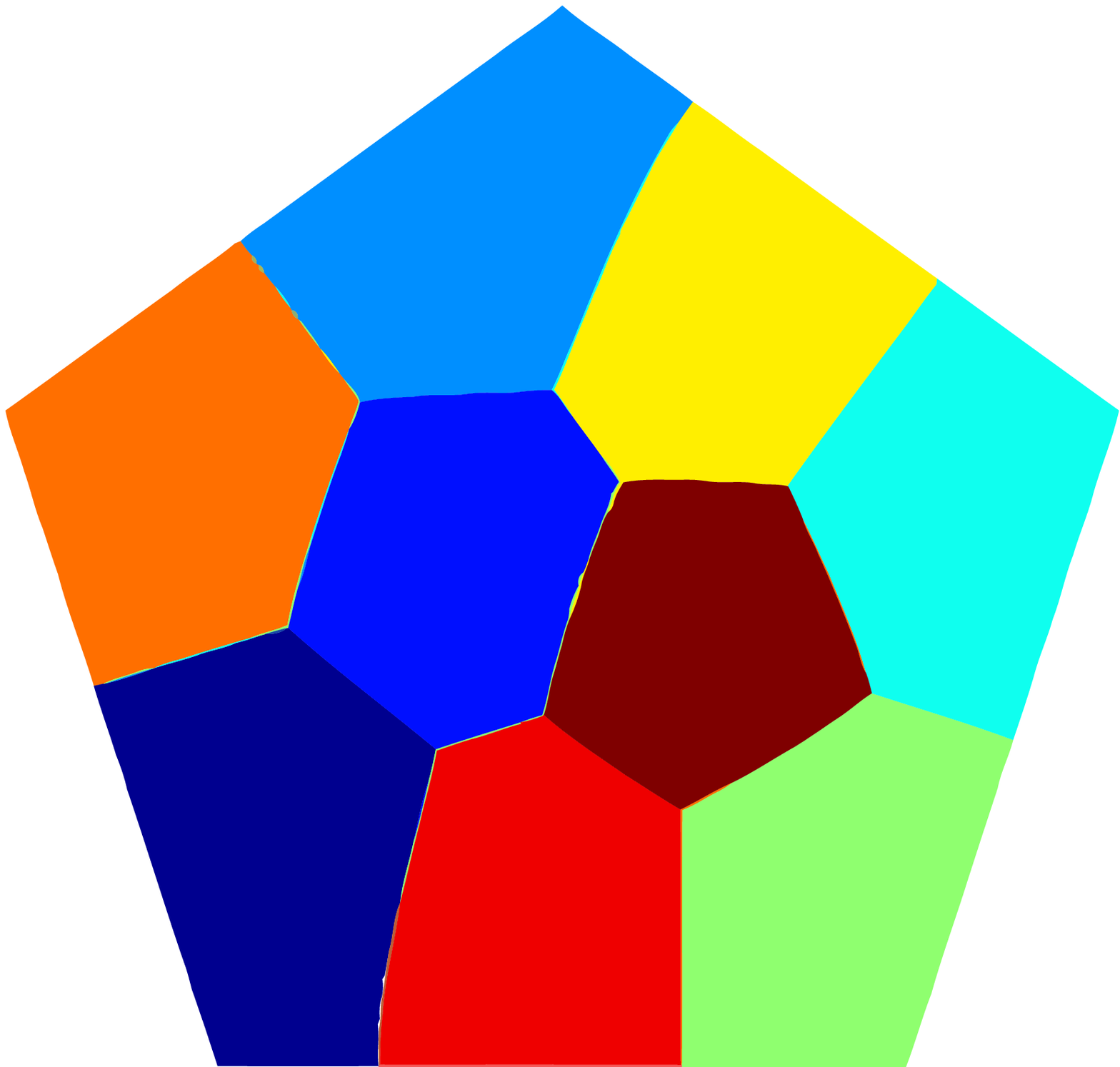}
  ~
  \includegraphics[width=0.15\textwidth]{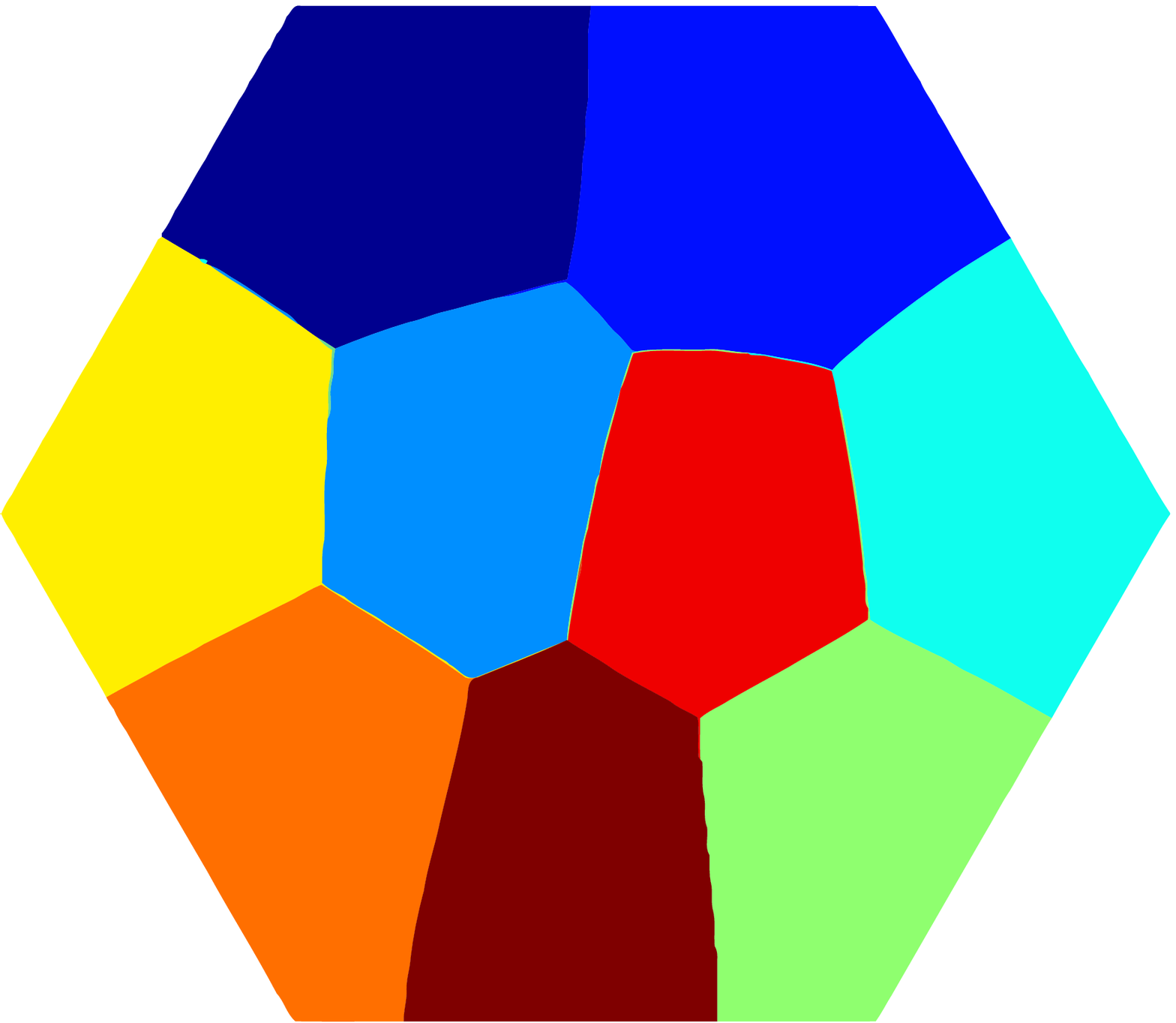}
  ~
  \includegraphics[width=0.15\textwidth]{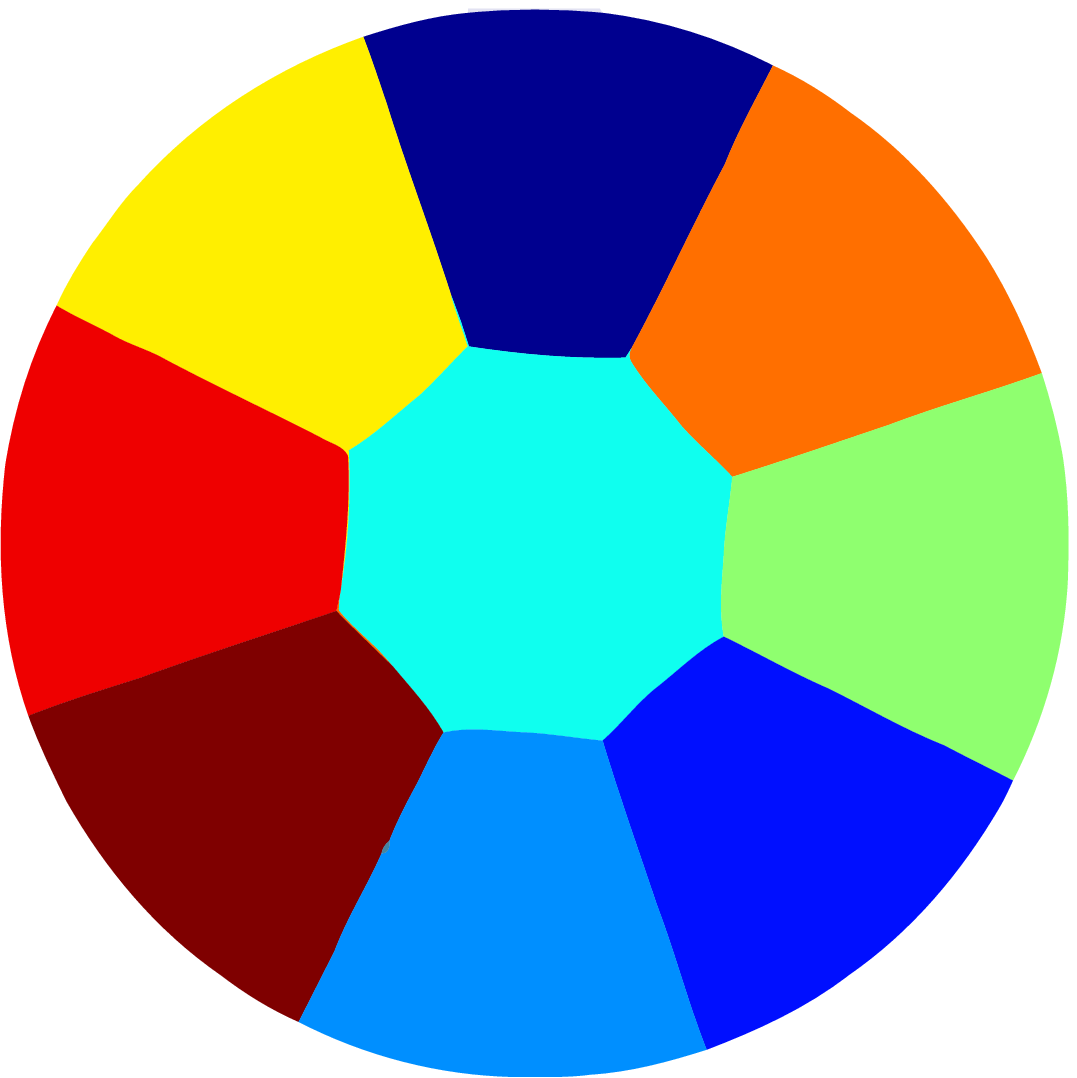}\\
  \includegraphics[width=0.15\textwidth]{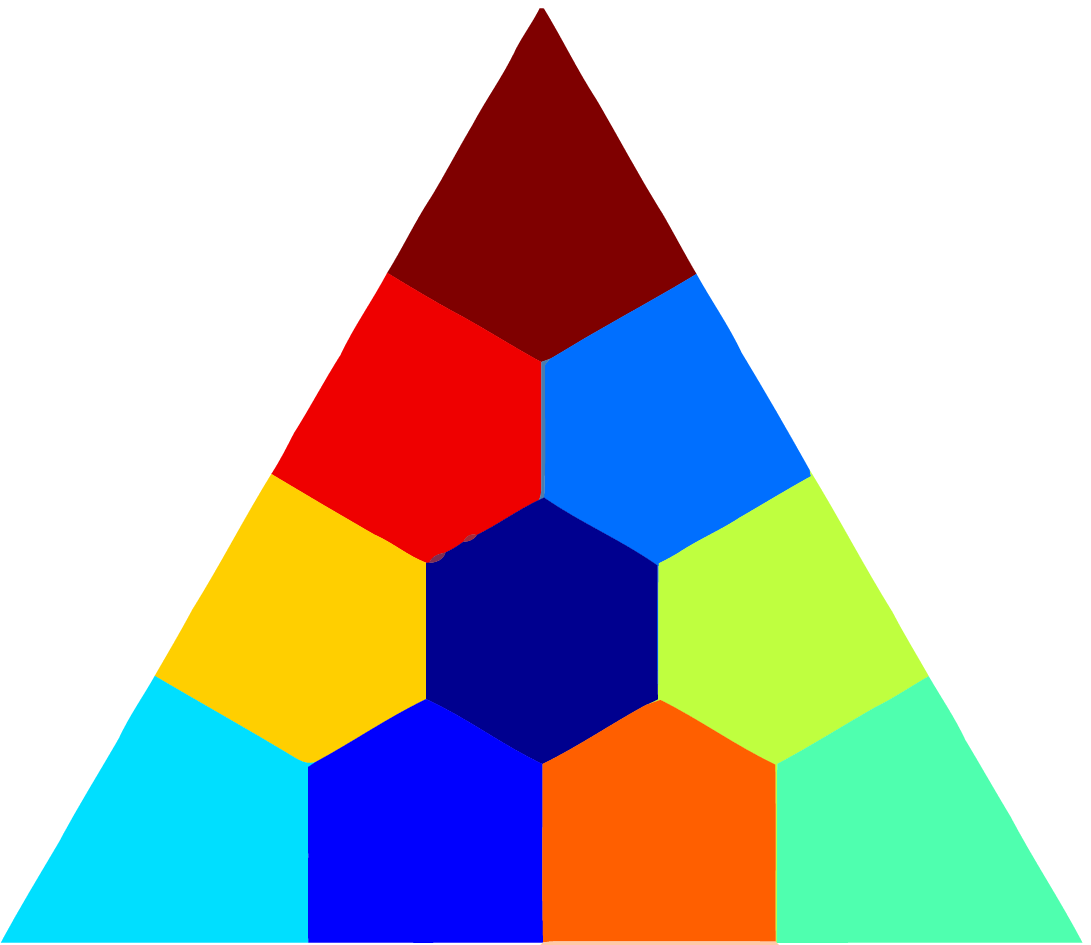}
  ~
  \includegraphics[width=0.15\textwidth]{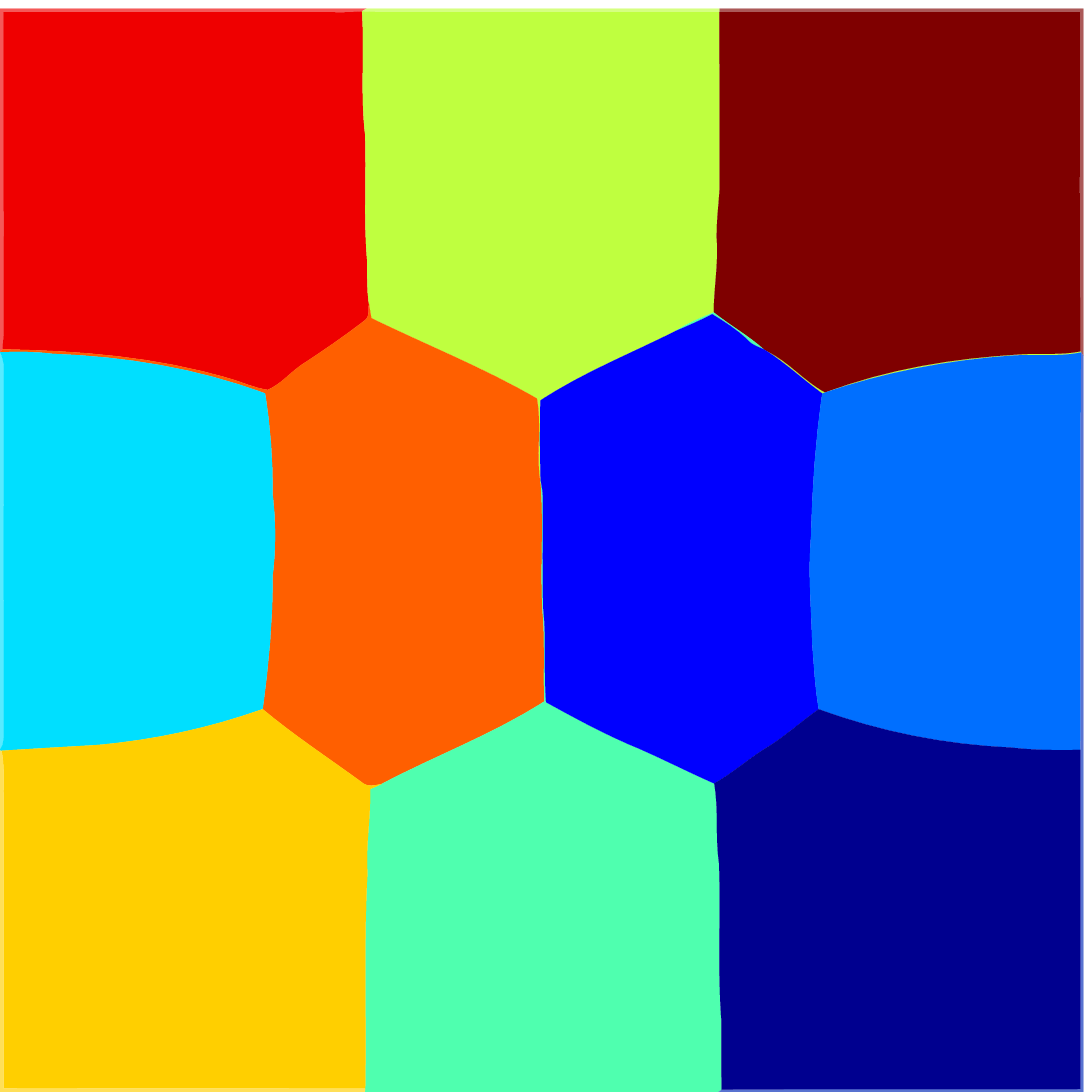} 
  ~
  \includegraphics[width=0.15\textwidth, angle = 90,origin=c]{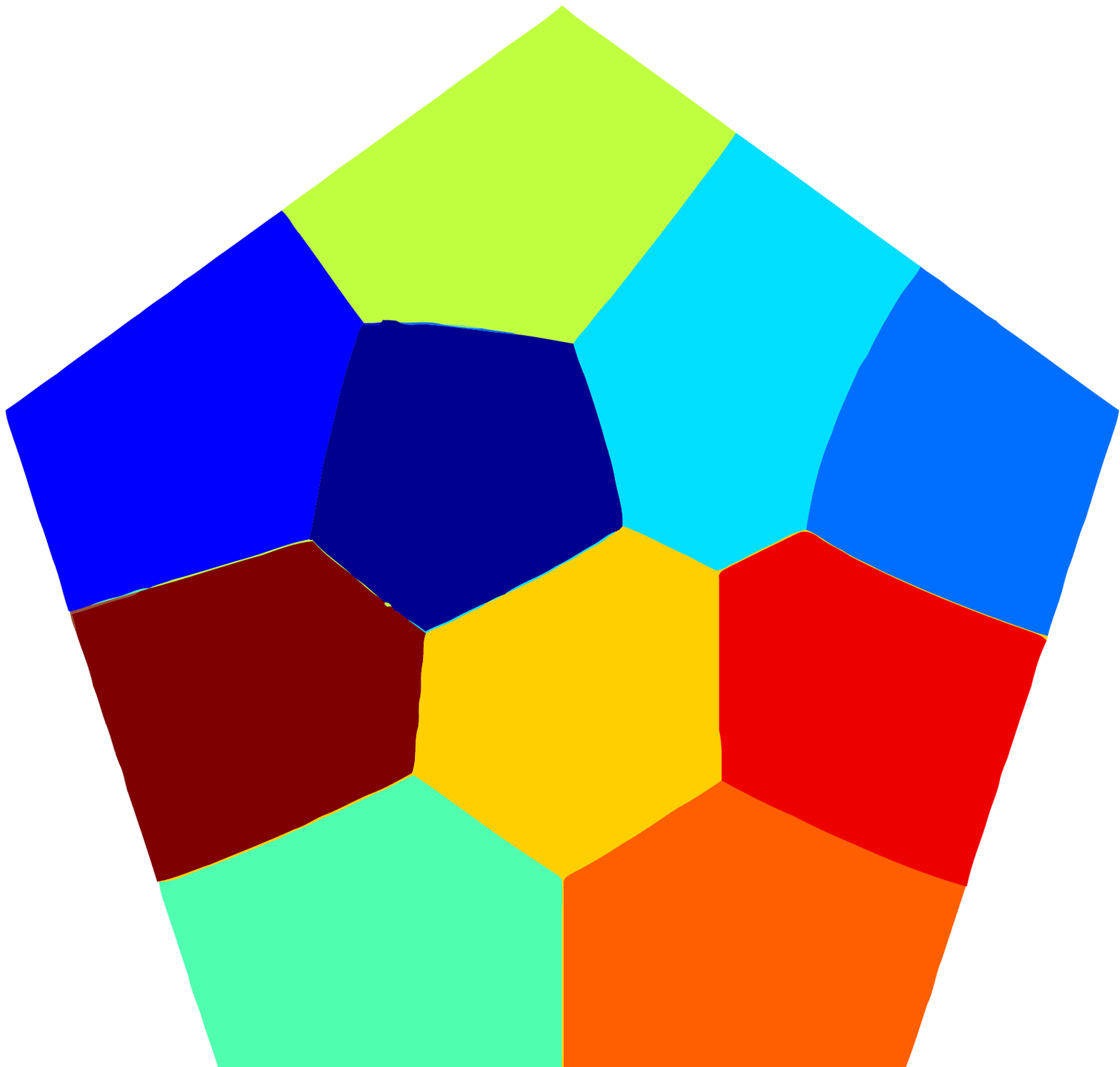}
  ~
  \includegraphics[width=0.15\textwidth]{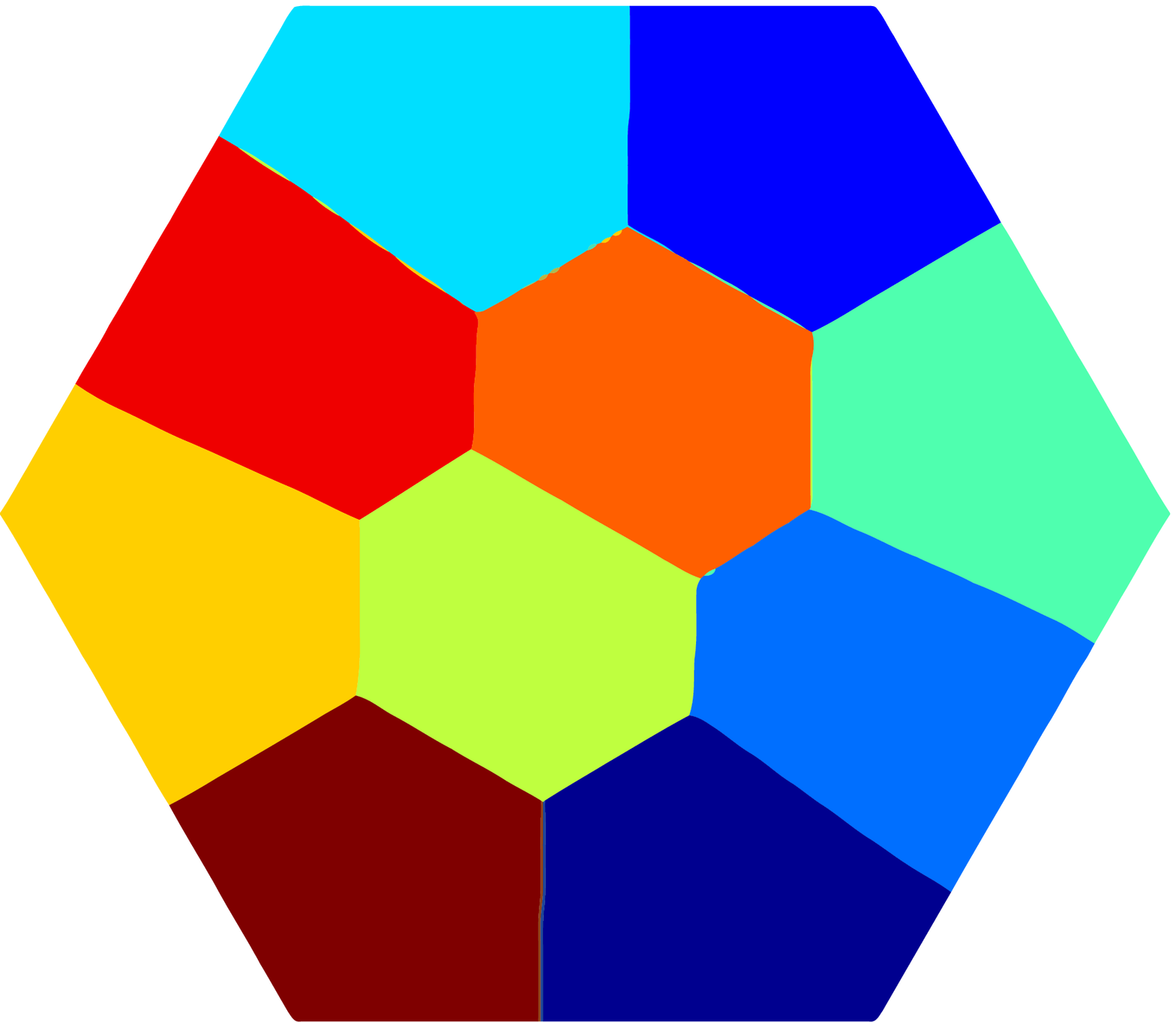}
  ~
  \includegraphics[width=0.15\textwidth]{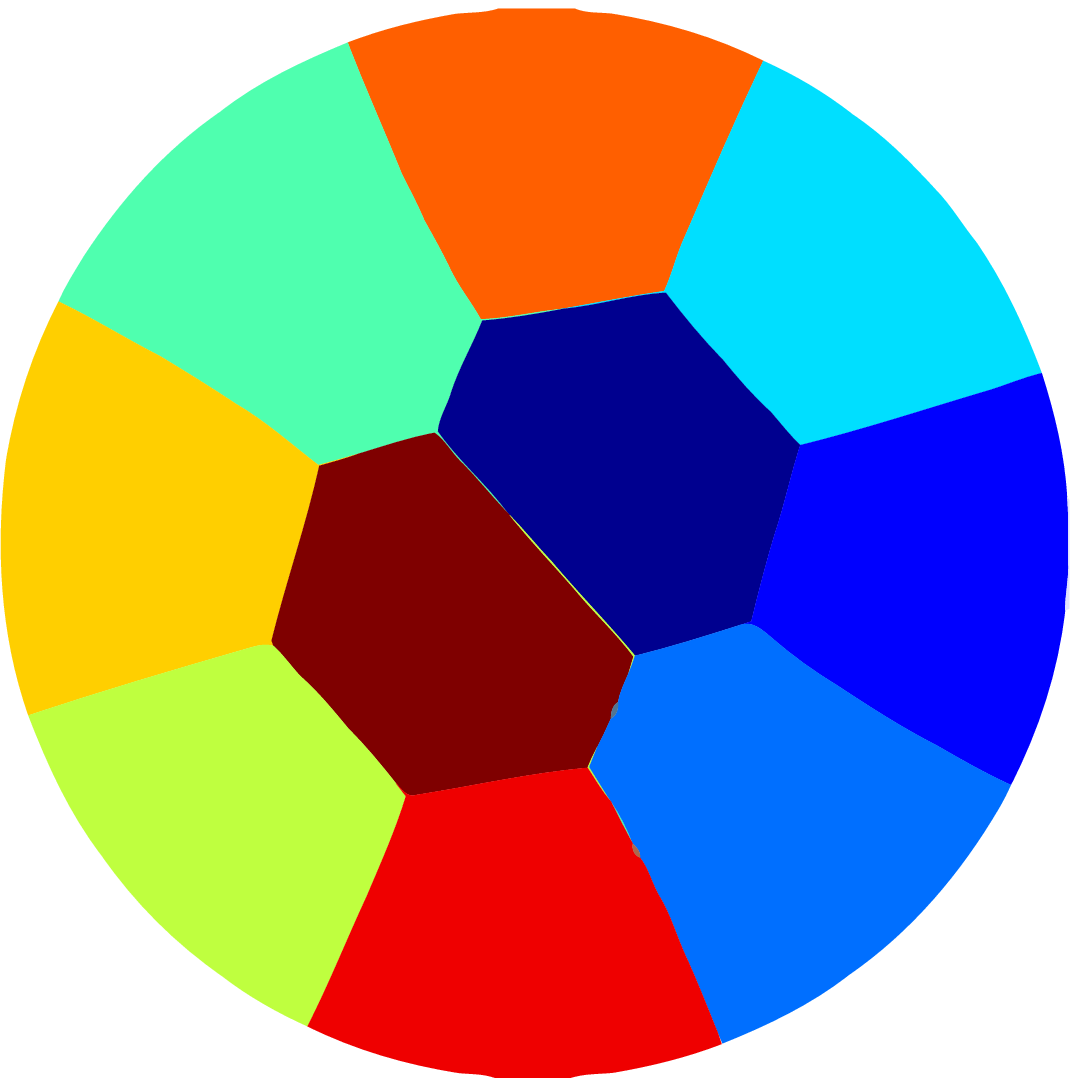}
  \caption{Optimal partitions ($\alpha=0$) for some generic domains $D$}
\label{non-rectangular}
\end{figure}

\bibliography{../master,master,revision_SINUM.bbl}
\bibliographystyle{abbrv}

\end{document}